\documentclass[letter,11pt]{article}
\usepackage{amsthm} 

\usepackage[mathlines]{lineno}
\usepackage{amsmath}
 \usepackage[pdftex]{graphicx}
 \usepackage[framemethod=tikz]{mdframed}
 \usepackage{algorithm}
 \usepackage{algpseudocode}
 \usepackage{epsfig,epstopdf}
 \usepackage{mathrsfs}
 \usepackage{bbm}
 \usepackage{seqsplit}
 \usepackage{caption}   
 \usepackage{float} 
 \usepackage{bigstrut}
 \usepackage{authblk}
  \usepackage{tcolorbox}

\usepackage[numbers]{natbib}

\usepackage[breaklinks=true,pdfstartview=FitH,pagebackref=true]{hyperref}
\usepackage{footnotebackref}

 \usepackage{tablefootnote}
 \usepackage{color}
\usepackage{subcaption}

\usepackage{mathtools}
\usepackage{amssymb}
\usepackage[shortlabels]{enumitem}
\usepackage[T1]{fontenc}
\usepackage{multirow}
\usepackage{booktabs,tabularx}
\usepackage{wrapfig}
\usepackage{bm}
\usepackage{soul}
\newcommand{\toweak}{\rightharpoonup}
\newcommand{\toset}{\rightrightarrows}
\def \rla{\right\rangle}
\def \lla{\left\langle}
\def\citep{\cite}
\def\citet{\cite}

\newcommand{\memosolved}[1]{}

 \newcommand{\ds}{\displaystyle}

\newcommand{\D}{\mathbf{D}}

\newcommand{\F}{\mathbf{F}}
\newcommand{\G}{\mathbf{G}}
\renewcommand{\H}{\mathcal{H}}
\newcommand{\J}{\mathbf{J}}
\newcommand{\K}{\mathbf{K}}

\newcommand{\R}{\mathbf{R}}
\newcommand{\JLambda}[1]{J_{\lambda #1}^{\bfLambda}}
\newcommand{\RLambda}[1]{R_{\lambda #1}^{\bfLambda}}
\newcommand{\tildef}{\tilde{f}}

\renewcommand{\a}{\mathbf{a}}
\renewcommand{\b}{\mathbf{b}}
\newcommand{\n}{\mathbf{n}}
\renewcommand{\u}{\mathbf{u}}
\renewcommand{\v}{\mathbf{v}}

\newcommand{\w}{\mathbf{w}}
\newcommand{\x}{\mathbf{x}}
\newcommand{\y}{\mathbf{y}}
\newcommand{\z}{\mathbf{z}}

\newcommand{\I}{\mathcal{I}}
\renewcommand{\J}{\mathcal{J}}
\newcommand{\bfLambda}{\bm{\Lambda}}
\newcommand{\bfDelta}{\bm{\Delta}}
\newcommand{\sigmabf}[1]{\sigma_{_{#1}}}

\newcommand{\innerlambda}[2]{\lla #1, #2 \rla _{\bfLambda}}
\newcommand{\inner}[2]{\lla #1, #2 \rla}
\newcommand{\norm}[1]{{\left\|{#1}\right\|}}
\newcommand{\normlambda}[1]{\left\|{#1}\right\|_{\bfLambda}}
\newcommand{\dr}[2]{T_{#1,#2}}
\newcommand{\Deltak}[1]{\Delta {#1}^k_i}

\newcommand{\jhsolved}[1]{}

\newcommand{\alert}[1]{{#1}}

\renewcommand{\Re}{{\rm I}\! {\rm R}}

\DeclareMathOperator*{\argmin}{arg\,min}

\DeclareMathOperator*{\Fix}{Fix}

\DeclareMathOperator*{\interior}{int}
\DeclareMathOperator*{\co}{co}

\DeclareMathOperator*{\ran}{ran}
\DeclareMathOperator*{\dom}{dom}
\DeclareMathOperator*{\gra}{gra}
\DeclareMathOperator*{\Id}{Id}
\DeclareMathOperator*{\bfId}{\bm{\mathrm{Id}}}
\newcommand{\prox}{{\rm prox}}
\DeclareMathOperator*{\zer}{zer}

\usepackage{ifthen}

\newcommand{\sigmaf}[1][]{%
  \sigma_{_{f%
  \ifthenelse{\equal{#1}{}}{}{{}_{\mathrm{#1}}}%
  }}%
}

\newcommand{\Lf}[1][]{%
  L_{_{f%
  \ifthenelse{\equal{#1}{}}{}{{}_{\mathrm{#1}}}%
  }}%
}

\newcommand{\rhof}[1][]{%
  \rho_{_{f%
  \ifthenelse{\equal{#1}{}}{}{{}_{\mathrm{#1}}}%
  }}%
}

\usepackage[margin=1in]{geometry}
 
\usepackage[capitalize,nameinlink]{cleveref}
\crefname{assumption}{Assumption}{Assumptions}
\crefformat{equation}{\textup{#2(#1)#3}}
\crefrangeformat{equation}{\textup{#3(#1)#4--#5(#2)#6}}
\crefmultiformat{equation}{\textup{#2(#1)#3}}{ and \textup{#2(#1)#3}}
{, \textup{#2(#1)#3}}{, and \textup{#2(#1)#3}}
\crefrangemultiformat{equation}{\textup{#3(#1)#4--#5(#2)#6}}%
{ and \textup{#3(#1)#4--#5(#2)#6}}{, \textup{#3(#1)#4--#5(#2)#6}}{, and \textup{#3(#1)#4--#5(#2)#6}}

\Crefformat{equation}{#2Equation~\textup{(#1)}#3}
\Crefrangeformat{equation}{Equations~\textup{#3(#1)#4--#5(#2)#6}}
\Crefmultiformat{equation}{Equations~\textup{#2(#1)#3}}{ and \textup{#2(#1)#3}}
{, \textup{#2(#1)#3}}{, and \textup{#2(#1)#3}}
\Crefrangemultiformat{equation}{Equations~\textup{#3(#1)#4--#5(#2)#6}}%
{ and \textup{#3(#1)#4--#5(#2)#6}}{, \textup{#3(#1)#4--#5(#2)#6}}{, and \textup{#3(#1)#4--#5(#2)#6}}

\crefdefaultlabelformat{#2\textup{#1}#3}

 \numberwithin{equation}{section}

\newtheorem{theorem}{Theorem}
\newtheorem{definition}{Definition}

\newtheorem{proposition}{Proposition}
\newtheorem{lemma}{Lemma}

\theoremstyle{definition}
\newtheorem{example}{Example}
\newtheorem{remark}{Remark}

\def\TheTitle{Douglas--Rachford algorithm for nonmonotone multioperator inclusion problems}

\ifpdf
\hypersetup{
  pdftitle={\TheTitle},
  pdfauthor={Jan Harold Alcantara and Akiko Takeda}}
\fi

 \title{\TheTitle}
 \date{November 27, 2026}
\author{Jan Harold Alcantara\thanks{\url{janharold.alcantara@riken.jp}.  Center for Advanced Intelligence Project, RIKEN, Tokyo, Japan.} 
\qquad Akiko Takeda\thanks{\url{takeda@mist.i.u-tokyo.ac.jp}.
Department of Mathematical Informatics, Graduate School of Information Science and Technology, University of Tokyo,
Tokyo, Japan, and Center for Advanced Intelligence Project, RIKEN, Tokyo, Japan.}
}

\begin{document}

\maketitle

 \begin{abstract}
The Douglas--Rachford algorithm is a classic splitting method for finding a zero of the sum of two maximal monotone operators. It has also been applied to settings that involve one weakly and one strongly monotone operator. In this work, we extend the Douglas--Rachford algorithm to address multioperator inclusion problems involving $m$ ($m\geq 2$) weakly and strongly monotone operators, reformulated as a two-operator inclusion in a product space. By selecting appropriate parameters, we establish the convergence of the algorithm to a fixed point, from which solutions can be extracted. Furthermore, we illustrate its applicability to  sum-of-$m$-functions minimization problems characterized by weakly convex and strongly convex functions. For general nonconvex problems in finite-dimensional spaces, comprising Lipschitz continuously differentiable functions and a proper closed function, we provide global subsequential convergence guarantees.

   \ifdefined\submit
   	
\else 
\fi 
\end{abstract}

\section{Introduction}
In this paper, we consider the problem 
\begin{equation}
    \text{Find~}x\in \H ~\text{such that } 0\in A_1 (x) + A_2(x) + \cdots + A_m (x)
    \label{eq:inclusion}
\end{equation}
where $A_1,A_2,\dots,A_m:\H \toset \H$ are set-valued operators on a real Hilbert space $\H$. We assume that each operator  is  accessible through its resolvent, and therefore we focus on so-called \emph{backward algorithms} for solving \eqref{eq:inclusion}. 

A popular backward algorithm for solving \eqref{eq:inclusion} when $m=2$ is the classical \textit{Douglas--Rachford} (DR) algorithm, which was initially proposed in 1956 by Douglas and Rachford \cite{DouglasRachford56a} as a numerical method for solving linear systems related to heat conduction. Later, Lions and Mercier (1979) extended its scope, making it applicable to finding zeros of the sum of two maximal monotone operators \cite{LionsMercier1979}. In particular, it can be used to minimize the sum of two convex functions, as this task is equivalent, under a suitable qualification condition, to finding the zeros of the sum of the subdifferential operators of the functions.

Extensions \alert{beyond the classical maximal monotone setting} have been explored in subsequent works. For the specific case of a two-term optimization problem involving a weakly convex and a strongly convex function in $\H=\Re^n$,  \cite{GuoHanYuan2017} established that the ``shadow sequence'' of the DR algorithm, with a sufficiently small step size, is globally convergent to the optimal solution when the sum of the functions is strongly convex.  The subdifferential operators of these functions belong to the class of \textit{generalized monotone operators}, which was the central focus of \cite{DaoPhan2019} and \cite{GiselssonMoursi2021}. These works specifically extended the analysis of the DR algorithm to accommodate this broader class of operators in real Hilbert spaces (not necessarily finite dimensional), providing convergence guarantees under generalized monotonicity conditions. Specifically, when $A_1$ and $A_2$ are maximal $\sigma_1$-monotone and maximal $\sigma_2$-monotone operators (see \cref{defn:gen_monotone}) with $\sigma_1+\sigma_2 >0$, the shadow sequence of the DR algorithm is guaranteed to globally converge to a zero of $A_1 + A_2$, provided the step size is sufficiently small.

On the other hand, for the $m$-operator inclusion problem \eqref{eq:inclusion}, a traditional strategy is to first reformulate it as a two-operator problem via \textit{Pierra's product space reformulation} \cite{Pierra1976,Pierra1984}:
    \begin{equation}
        \text{Find}~\x \in \H^m ~\text{such that } 0\in \F (\x) + \G(\x),  
        \label{eq:pierra}
    \end{equation}
where $\x = (x_1,\dots,x_m)\in \H^m$, $\F (\x) \coloneqq  A_1(x_1) \times \cdots \times A_m(x_m)$ and $\G \coloneqq N_{\D_m} $, the normal cone operator to $\D_m \coloneqq \{ (x_1,\dots,x_m)\in \H^m: x_1=\cdots = x_m\}$. The defined operators retain key properties:  $\F$ is maximal monotone when each $A_i$ is maximal monotone, while $\G$ is maximal monotone due to the convexity of $\D_m$ \cite[Proposition 26.4]{Bauschke2017}. Consequently, the shadow sequence of the standard DR algorithm applied to \eqref{eq:pierra} is globally convergent to a zero of $\F+\G$, which corresponds to a solution of \eqref{eq:inclusion}. However, one major drawback of the reformulation \eqref{eq:pierra} is its incompatibility with the theory for sum of two generalized monotone operators. \alert{Specifically, in the reformulation \eqref{eq:pierra}, we have
$\G=N_{\D^m}$, which is maximal monotone and therefore corresponds to
$\sigma_{\rm \G}=0$. Hence, to apply the existing two-operator convergence theory
for maximal $\sigma_{\rm \F}$- and $\sigma_{\rm \G}$-monotone operators with
$\sigma_{\rm \F}+\sigma_{\rm \G}>0$, one would need $\sigma_{\rm \F}>0$.
On the other hand, if each $A_i$ is maximal $\sigma_i$-monotone, then
$\sigma_{\rm \F}=\min\{\sigma_1,\ldots,\sigma_m\}$ (see \cref{prop:F_monotone}). Thus, the
condition $\sigma_{\rm \F}>0$ requires $\sigma_i>0$ for all $i=1,\ldots,m$,
and therefore excludes the case where at least one operator is weakly monotone.}

\paragraph{Contributions of this work} In this work, our primary goal is to extend the existing convergence theory for the two-operator inclusion problem involving generalized maximal monotone operators to the case of the 
$m$-operator inclusion problem \eqref{eq:inclusion}. The main contributions are as follows:
\begin{enumerate}[(I)]
    \item We establish the convergence theory for the DR algorithm applied to a certain two-operator reformulation of \eqref{eq:inclusion}, distinct from Pierra's product space reformulation \eqref{eq:pierra}. Specifically, \cref{thm:convergence_optimal} shows that when the operators $A_i$ are maximal $\sigma_i$-monotone such that $\sigma_1+\cdots+\sigma_m >0$, the derived DR algorithm with an appropriate step size achieves global convergence to a fixed point, which corresponds to a solution of \eqref{eq:inclusion}. These results cannot be recovered by directly applying \cite{DaoPhan2019,GiselssonMoursi2021} to our reformulation. By contrast, our refined analysis provides stronger guarantees: it relaxes the requirements on the $\sigma_i$ and can permit larger step-size ranges, whereas a direct application of \cite{DaoPhan2019,GiselssonMoursi2021} would require significantly stricter conditions and yield smaller step sizes (see \cref{rem:comparewithnaive}).

    \item A secondary contribution of this work is the introduction of a flexible product space reformulation for \eqref{eq:inclusion} that does not require generalized maximal monotonicity assumptions. Building on Campoy's product space reformulation \cite{Campoy2022}, which originates from \cite{Kruger1985}, the proposed formulation is valid for arbitrary $m$-inclusion problems. Unlike previous approaches, it is independent of (generalized) monotonicity conditions but reduces to Campoy's formulation when \alert{the relevant operator is
convex-valued; this convex-valuedness may fail in the finite-dimensional nonconvex setting of \cref{sec:nonconvexopt_finitedimensional}.}

    \item We apply our results to sum-of-$m$-functions unconstrained optimization problems (see \eqref{eq:optimization}) involving weakly and strongly convex functions. For general nonconvex problems in finite-dimensional spaces, we prove global subsequential convergence under the condition that all but one function have Lipschitz continuous gradients, with the remaining function being any proper closed function. 
\end{enumerate}

\paragraph{Organization of the paper} In \cref{sec:prelim}, we review some background materials on set-valued operators, generalized monotonicity and extended real-valued functions. We recall Campoy's product space reformulation in \cref{sec:pierra_campoy}, and present our flexible reformulation in \cref{sec:alternative_productspace}. Based on this, the proposed Douglas--Rachford algorithm is presented in \cref{sec:warpedresolvents}. Our convergence analysis and main results for the inclusion problem are presented in \cref{sec:mainresults_inclusion}, and the applications to nonconvex optimization are discussed in \cref{sec:mainresults_optimization}. Concluding remarks are given in \cref{sec:conclusion}.

\section{Preliminaries}
\label{sec:prelim}
Throughout this paper, $\H$ denotes a real Hilbert space endowed with the inner product $\lla \cdot, \cdot \rla$ and induced norm $\|\cdot\|$. \alert{
If $\Lambda:\H\to\H$ is a self-adjoint positive definite linear operator, then
it induces the inner product
$
    \langle x,y\rangle_{\Lambda}\coloneqq \langle x,\Lambda y\rangle
$
and the corresponding norm
$\norm{x}_{\Lambda}\coloneqq \sqrt{\langle x,x\rangle_{\Lambda}}$.
For any real numbers $\alpha,\beta\in \Re$ and any $x,y \in \H$, we recall the
following identity:
    \begin{equation}
        \norm{\alpha x + \beta y}_{\Lambda}^2
        = \alpha (\alpha +\beta) \norm{x}_{\Lambda}^2
        + \beta (\alpha+\beta) \norm{y}_{\Lambda}^2
        - \alpha\beta \norm{x - y}_{\Lambda}^2.
        \label{eq:identity_squarednorm}
    \end{equation}
When $\alpha+\beta \neq 0$, \eqref{eq:identity_squarednorm} is equivalent to 
    \begin{equation}
       \textstyle
       \alpha \norm{x}_{\Lambda}^2 + \beta \norm{y}_{\Lambda}^2
       =
       \frac{\alpha\beta}{\alpha + \beta}\norm{x - y}_{\Lambda}^2
       +
       \frac{1}{\alpha + \beta} \norm{\alpha x + \beta y}_{\Lambda}^2.
       \label{eq:identity_squarednorm2}
    \end{equation}
A sequence $\{x^k\}$ is said to be $\Lambda$-Fej\'er monotone with respect to a
nonempty subset $S\subseteq \H$ if, for all $z\in S$ and $k\in\mathbb{N}$, $ \norm{x^{k+1}-z}_{\Lambda}\leq \norm{x^k-z}_{\Lambda}.$ }We use  $\to$ and $\toweak$ to denote strong and weak convergence, respectively. 

\subsection{Set-valued operators}
 A set-valued operator $A:\H\toset \H$ maps each point $x\in \H$ to a subset $A(x)$ of $\H$, which is not necessarily nonempty. The image of a subset $D\subseteq \H$ is given by $A(D) \coloneqq \bigcup _{x\in D}A(x)$. The \emph{domain} and \emph{range}  of $A$ are given respectively by $\dom (A)  \coloneqq \{ x\in \H: A(x) \neq \emptyset\}$ and $ \ran (A) \coloneqq \{ y\in \H : y\in A(x) ~\text{for some }x\in \H\}$.
The \emph{graph} of $A$ is the subset of $\H\times \H$ given by 
$\gra (A) \coloneqq \{ (x,y) \in \H \times \H : y\in A(x)\}.$
The \emph{inverse} of $A$, denoted by $A^{-1}$, is the set-valued operator whose graph is given by 
$\gra (A^{-1}) = \{ (y,x)\in \H\times \H: (x,y) \in \gra (A)\}.$
The \emph{zeros} and \emph{fixed points} of $A$ are given respectively by
$\zer (A)  \coloneqq A^{-1}(0) = \{ x: 0\in A(x)\},$ and $\Fix (A)  \coloneqq \{ x\in \H : x\in A(x)\}$.
The \emph{resolvent} of $A:\H\toset\H$ with parameter $\gamma >0$, denoted by $J_{\gamma A}:\H\toset\H$, is defined by 
$ J_{\gamma A} \coloneqq (\Id+\gamma A)^{-1},$
where $\Id:\H\to\H$ is the identity operator $\Id(x)=x$. The \emph{reflected resolvent} of $A$ with parameter $\gamma>0$ is given by 
$ R_{\gamma A} \coloneqq 2J_{\gamma A} - \Id .$

\subsection{Generalized monotone operators}

\begin{definition}
\label{defn:gen_monotone}
    Let $A:\H\toset \H$ and let $\sigma\in \Re$. We say that $A$ is \emph{$\sigma$-monotone} if \alert{for all  $(x,u),(y,v)\in \gra(A)$, one has}
    $\lla x - y, u - v \rla \geq \sigma \norm{x-y}^2  .$   
        \alert{We say that} $A$ is \emph{monotone} when $\sigma=0$, \emph{strongly monotone} if $\sigma>0$ and \emph{weakly monotone} if $\sigma<0$.  Moreover, $A$ is \emph{maximal $\sigma$-monotone} if $A$ is $\sigma$-monotone and there is no $\sigma$-monotone operator whose graph properly contains $\gra (A)$.
        % ; that is, for every $(x,u)\in \H\times \H$,
        % \[(x,u)\in \gra(A) \quad \Longleftrightarrow \quad \lla x - y, u - v \rla \geq \sigma \norm{x-y}^2 \quad \forall (y,v)\in \gra(A)  .\]
        \alert{We say that} $A$ is \emph{maximal monotone} when $\sigma=0$, \emph{maximal strongly monotone} if $\sigma>0$ and \emph{maximal weakly monotone} if $\sigma<0$.
    % \begin{enumerate}[(i)]
    %     \item $A$ is \emph{$\sigma$-monotone} if 
    %     \[\lla x - y, u - v \rla \geq \sigma \norm{x-y}^2 \quad \forall (x,u),(y,v)\in \gra(A) .\]
    %     When $\sigma=0$, we say that $A$ is \emph{monotone}. 
        % \item $A$ is \emph{$\sigma$-comonotone} if 
        % \[\lla x - y, u - v \rla \geq \sigma \norm{u-v}^2 \quad \forall (x,u),(y,v)\in \gra(A) .\]
        % \item $A$ is \emph{maximal $\sigma$-monotone} if $A$ is $\sigma$-monotone and there is no $\sigma$-monotone operator whose graph properly contains $\gra (A)$; that is, for every $(x,u)\in \H\times \H$,
        % \[(x,u)\in \gra(A) \quad \Longleftrightarrow \quad \lla x - y, u - v \rla \geq \sigma \norm{x-y}^2 \quad \forall (y,v)\in \gra(A)  .\]
        % When $\sigma=0$, we say that $A$ is \emph{maximal monotone}. 
        % \item $A$ is \emph{maximal $\sigma$-comonotone} if $A$ is $\sigma$-comonotone and there is no $\sigma$-monotone operator whose graph properly contains $\gra (A)$; that is, for every $(x,u)\in \H\times \H$,
        % \[(x,u)\in \gra(A) \quad \Longleftrightarrow \quad \lla x - y, u - v \rla \geq \sigma \norm{u-v}^2 \quad \forall (y,v)\in \gra(A)  .\]
    % \end{enumerate}
    
\end{definition}

We summarize some facts about maximal monotone operators.

\begin{lemma}
\label{lemma:maximalmonotone_properties}
    Let $A,B:\H\toset \H$ be maximal monotone operators. Then the following holds
    \begin{enumerate}[(i)]
        % \item $J_{\gamma A}$ is monotone and nonexpansive on $\dom (J_{\gamma A})$.
        % \item $\dom (J_{\gamma A}) = \ran( \Id+\gamma A) =\H$ if and only if $A$ is maximal monotone.
        \item $A(x)$ is convex for any $x\in \H$.
        % \item \alert{If} \begin{equation}\label{eq:cone=span} \alert{ \operatorname{cone}(\dom (A)-\dom (B)) = \overline{\operatorname{span}}(\dom (A)-\dom (B)),} \end{equation} \alert{then $A+B$ is maximally monotone. In particular, \eqref{eq:cone=span} holds if $\dom(A)$ and $\dom(B)$ are convex, and} \[ \alert{0\in \operatorname{sri}(\dom(A)-\dom(B)).} \] \alert{Here, $\operatorname{cone}$ denotes the conic hull, $\overline{\operatorname{span}}$ denotes the closure of the linear span, and $\operatorname{sri}$ denotes the strong relative interior.}
        \item \alert{If $\dom (B) =\H$, then $A+B$ is maximal monotone.}
    \end{enumerate}
\end{lemma}
\begin{proof}
    Part (i) holds by \cite[Proposition 20.36]{Bauschke2017}. 
     Part (ii) is from \cite[Corollary 25.5]{Bauschke2017}.
  \end{proof}

We also recall an important characterization of maximal $\sigma$-monotone operators.
\begin{lemma}
\label{lemma:maximal_sigma_equivalent}
   Let $A:\H\toset \H$ and let $\sigma\in \Re$. 
   Then $A$ is maximal $\sigma$-monotone if and only if $A-\sigma \Id$ is maximal monotone.
\end{lemma}
\begin{proof}
    See \cite[Lemma 2.8]{Bauschke2021}.
  \end{proof}

\begin{lemma}
\label{lemma:maximal_single-valued-resolvent}
    Let $A:\H \toset \H$ be $\sigma$-monotone, and let $\gamma>0$ such that $1+\gamma \sigma>0$. Then $\dom (J_{\gamma A})= \H$ if and only if $A$ is maximal $\sigma$-monotone.
\end{lemma}
\begin{proof}
    See \cite[Proposition 3.4(ii)]{DaoPhan2019}
  \end{proof}
\subsection{Extended real-valued functions}
Let  $f:\H\to (-\infty,\infty]$ be an extended real-valued function. The \emph{domain} of $f$ is given by the set $\dom (f) = \{ x\in\H : f(x)<\infty\}$. We say that $f$ is a \emph{proper} function if $\dom (f)\neq \emptyset$, and that $f$ is \emph{closed} if it is lower semicontinuous. \alert{The function} $f$ is said to be a \emph{$\sigmaf$-convex} function if
$f-\frac{\sigmaf}{2}\norm{\cdot}^2$ is convex for some
$\sigmaf\in\Re$.  If $\sigmaf>0$, then $f$ is
\emph{$\sigmaf$-strongly convex}. If $\sigmaf\leq 0 $, we denote
$\rho_{f}\coloneqq -\sigmaf$ and call $f$ a
\emph{$\rho_{f}$-weakly convex function}. In other words, $f$ is
$\rho_{f}$-weakly convex for $\rho_{f} \geq 0$ if $f+\frac{\rho_{f}}{2}\norm{\cdot}^2$ is convex. 

%\medskip 

The \emph{subdifferential} of $f$ \alert{(see \cite[Definition 8.3]{RW98})} is the set-valued operator $\partial f : \H \toset \H$ given by 
\alert{\begin{align}
    \partial f (x) \coloneqq 
        \{ z \in \H : \exists
	\{(x^k,z^k)\} \text{ s.t. }
	x^k\xrightarrow{f} x,  ~z^k\in \hat{\partial}f(x^k),
	~\text{and}~z^k\to z\},\label{eq:generalsubdiff}
\end{align}
if $ x\in \dom (f)$, and $\partial f (x)\coloneqq \emptyset$ otherwise. Here,} $x^k\xrightarrow{f} x$ means $x^k\to x$ and $f(x^k)\to f(x)$, and
\begin{equation*}
  \textstyle   \hat{\partial}f (x) \coloneqq \left\lbrace z\in\H:
	\liminf_{\bar{x}\to x, \bar{x}\neq x}\frac{f(\bar{x}) - f(x)
	- \lla z,\bar{x}-x\rla }{\norm{\bar{x}-x}} \geq 0\right\rbrace.
\end{equation*}
When $f$ is convex, \cref{eq:generalsubdiff} coincides with the classical subdifferential in convex analysis:
$\partial f(x) = \{z\in \H : f(y) \geq f(x) + 
\lla z, y-x \rla,
~\forall y\in \H \}.$
 The \emph{indicator function} of a set $D\subseteq \H$ is the function $\delta_D:\H\to (-\infty,+\infty]$, such that $\delta_D(x) = 0$ if $x\in D$ and $\delta_D(x) = +\infty$ if $x\notin D$. If $D$ is closed and convex, then $\delta_D$ is a convex function whose subdifferential coincides with the \emph{normal cone} to $D$, denoted by $N_D$:
 \[\textstyle  \partial \delta_D(x) = N_D (x) = \begin{cases}
     \{ z\in \H : \lla z,y-x \rla \leq 0, ~\forall y\in D \} & \text{if~}x\in D ,\\
\emptyset & \text{otherwise}.
 \end{cases}\]

If $f:\H\to\Re$ is continuously differentiable, the subdifferential reduces to $\partial f (x) = \{ \nabla f(x) \}$ for any $x\in \H$.  We say that $f$ is \emph{$\Lf$-smooth} if its gradient satisfies 
$ \norm{\nabla f (x) - \nabla f (y)}\leq \Lf \norm{x-y}, \alert{\text{for all }} x,y\in \H.$
If $f$ is $\Lf$-smooth, we have from \cite[Lemma 2.64(i)]{Bauschke2017} the following inequality
\begin{equation}
 \textstyle    \left| f(y)-f(x)-\lla \nabla f(x) , y-x \rla
			\right|  \leq \frac{\Lf}{2}\norm{y-x}^2,
    \label{eq:descentlemma}
\end{equation}
\alert{for all $x,y\in\H$,} which is also known as the \emph{descent lemma}. \alert{
If $f$ is convex, then by \cite[Corollary 18.17]{Bauschke2017}, it is $\Lf$-smooth
if and only if
\begin{equation}
   \textstyle
   \inner{\nabla f(x) - \nabla f(y)}{x-y}
   \geq \frac{1}{L_f}\norm{\nabla f(x) - \nabla f(y)}^2,
    \label{eq:descentlemma_convex2}
\end{equation}
for all $x,y\in\H$. This condition is also equivalent to
\begin{equation}
 \textstyle
 f(y)-f(x)-\lla \nabla f(x) , y-x \rla
 \geq
 \frac{1}{2\Lf}\norm{\nabla f(y)-\nabla f(x)}^2,
    \label{eq:descentlemma_convex}
\end{equation}
for all $x,y\in\H$. Indeed, for a fixed $x$, 
the descent lemma applied to the convex $\Lf$-smooth function
$\phi(z)\coloneqq f(z)-\inner{\nabla f(x)}{z}$ gives $\phi\left(y-\frac{1}{L_f}\nabla \phi(y)\right) \leq \phi(y) - \frac{1}{2L_f}\norm{\nabla \phi(y)}^2$. Since $x$ minimizes $\phi$, then $\phi(x)\leq \phi(y) - \frac{1}{2L_f}\norm{\nabla \phi(y)}^2$, which implies \eqref{eq:descentlemma_convex}.
Conversely, \eqref{eq:descentlemma_convex2} follows by writing
\eqref{eq:descentlemma_convex} with the roles of $x$
and $y$ interchanged, and then adding it to \eqref{eq:descentlemma_convex}.}

For a proper, closed function $f:\H \to (-\infty,+\infty]$, the \emph{proximal mapping} of $f$ is given by 
\begin{equation}
    \prox_{\gamma f} (x) \coloneqq \alert{\argmin_{w\in \H}}\, f(w) + \frac{1}{2\gamma}\|w-x\|^2, \quad \gamma>0.
	\label{eq:prox} 
\end{equation}
% and, we define $\prox_{\gamma f} (C)
% \coloneqq \bigcup_{x\in S}\prox_{\gamma f}(x)$ for $C\subseteq \H$ . 
\alert{Recall that, by Fermat's rule \cite[Theorem 10.1]{RW98}, if $x\in \arg\min g$, then $0\in \partial g(x)$}.
From the optimality condition of \eqref{eq:prox}, 
we have that if $y\in \prox_{\gamma f}(x)$, then $x-y \in \gamma \partial f (y)$. 
That is,
\begin{equation}
 \alert{\forall x\in \H}, \quad \prox_{\gamma f}(x) \subseteq J_{\gamma \partial f} (x) . 
 \label{eq:prox_subset_J}
\end{equation}
Note that equality in \eqref{eq:prox_subset_J} holds whenever $f$ is convex. By contrast, strict inclusion can occur for nonconvex $f$. For instance, if $f(t)=-\tfrac12 t^2$, then $\prox_{\gamma f}(t)=\varnothing$ for every $\gamma>1$, whereas $J_{\gamma\partial f}(t)=\{\,t/(1-\gamma)\,\}$ for every $\gamma\neq 1$.

\section{A general product space reformulation  and the Douglas--Rachford~algorithm}
\label{sec:productspace_dr}
In \cref{sec:pierra_campoy}, we recall the product space reformulation by \cite{Campoy2022} (inspired by \cite{Kruger1985}) that relies on maximal monotonicity of the operators. In the absence of this assumption, we present an alternative product space reformulation in \cref{sec:alternative_productspace}. Some fundamental formulas for resolvents of operators defining the reformulation are established in \cref{sec:warpedresolvents}. 
\subsection{Campoy's product space reformulation}\label{sec:pierra_campoy}
Denote 
$\H^{m-1} = \H \times \overset{(m-1)}{\cdots} \times \H$, which is a Hilbert space with inner product 
\[\lla \x,\y \rla = \sum_{i=1}^{m-1}\lla x_i,y_i\rla \quad \alert{\text{for all }}  \x = (x_1,\dots, x_{m-1}), ~\y = (y_1,\dots,y_{m-1}) , \]
and define
$\D_{m-1} \coloneqq \{ \x = (x_1,x_2,\dots, x_{m-1})\in \H^{m-1} : x_1 = \cdots = x_{m-1}\}. $
We also denote by $\bfDelta_{m-1}:\H \to \H^{m-1}$ the embedding operator $x\mapsto (x,\cdots ,x)$.  
The following result is from \cite[Theorem 3.3]{Campoy2022}. 
\begin{theorem}
    \label{thm:campoy}
    Let $A_1,\dots, A_m$ be maximal monotone operators.
    Define the set-valued operators $\F, \widetilde{\G}:\H^{m-1} \toset \H^{m-1}$ by
    \begin{align}
       \F (\x) & \coloneqq A_1(x_1) \times \cdots \times A_{m-1}(x_{m-1}), \label{eq:F} \\
        \widetilde{\G} (\x) & \coloneqq \widetilde{\K} (\x)+ N_{\D_{m-1}}(\x), \label{eq:G}
    \end{align}
where 
\begin{equation}
    \widetilde{\K}(\x) \coloneqq \frac{1}{m-1}A_m (x_1) \times \cdots \alert{\times} \frac{1}{m-1}A_m(x_{m-1}).
    \label{eq:K}
\end{equation}  
 Then $\F$ and $\widetilde{\G}$ are maximal monotone. Moreover, \begin{equation}
         \zer (\F + \widetilde{\G}) = \bfDelta_{m-1}\left( \zer \left( \sum_{i=1}^m A_i\right) \right).
         \label{eq:equal_zeros}
    \end{equation}
\end{theorem}

By \eqref{eq:equal_zeros}, the $m$-operator inclusion problem \eqref{eq:inclusion} can be equivalently recast as a two-operator problem 
\begin{equation}
    \text{Find } \x \in \H^{m-1} ~\text{such that } \mathbf{0}\in \F (\x) + \widetilde{\G} (\x).
    \label{eq:FG_reformulation}
\end{equation}
On the other hand, the Pierra's product space reformulation \eqref{eq:pierra} is a two-operator inclusion problem defined on the space $\H^m$. Consequently, the ambient space of Campoy's reformulation \eqref{eq:FG_reformulation} has dimension reduced by $\dim(\mathcal{H})$, which is more desirable in practice \cite{MalitskyTam2023}. Note that the reformulation \eqref{eq:FG_reformulation} has also been used in \cite[Theorem 2]{Kruger1985}.

We remark that it is straightforward to verify that ``$\supseteq$'' in \eqref{eq:equal_zeros} holds without the maximal monotonicity assumption. To motivate the product space reformulation in \cref{sec:alternative_productspace}, we briefly recall the proof of the inclusion ``$\subseteq$'', highlighting the role  of maximal monotonicity. For any $\x \in \zer (\F + \widetilde{\G})$, we have that $\mathbf{0} \in \F (\x) + \widetilde{\K}(\x) + N_{\D_{m-1}}(\x)$. Then $\x \in \D_{m-1}$ so that $\x = (x,\cdots,x)$ and there exist $\u\in \F (\x)$, $\v \in \widetilde{\K}(\x)$ and $\w \in N_{\D_{m-1}}(\x)$ such that $\u + \v + \w = \mathbf{0}$. By the definition of $\F$ and $\widetilde{\G}$, it follows that $\u=(u_1,\dots,u_{m-1})$ where $u_i \in A_i (x)$, and $\v = \frac{1}{m-1}(v_1,\dots,v_{m-1})$ where $v_i\in A_m(x)$ for $i=1,\dots,m-1$.
Noting that $\w\in N_{\D_{m-1}}(\x)$, the normal cone to $\D_{m-1}$ is given by \cite[Proposition 26.4]{Bauschke2017}
\begin{equation*}
    N_{\D_{m-1}}(\x) = \begin{cases}
        \D_{m-1}^\perp = \{ \w =(w_1,\dots,w_{m-1}): \sum_{i=1}^{m-1}w_i = 0 \} & \text{if }\x\in \D_{m-1} ,\\
        \emptyset & \text{otherwise},
    \end{cases}
\end{equation*}
and $-\w = \u + \v$, we obtain 
$0 =  \sum_{i=1}^{m-1} -w_i = \sum_{i=1}^{m-1}u_i + \frac{1}{m-1}\sum_{i=1}^{m-1} v_i .$
It is clear that the first term on the rightmost side belongs to $A_1(x) + \cdots + A_{m-1}(x)$. On the other hand, we have from \cref{lemma:maximalmonotone_properties}(i) that $A_m(x)$ is a convex set by the maximal monotonicity of $A_m$. Consequently,  we see that $\frac{1}{m-1}\sum_{i=1}^{m-1} v_i \in A_m(x)$ since $v_i\in A_m(x)$. Putting these together, we see that $0\in A_1(x) + \cdots + A_{m  }(x) $, \textit{i.e.,} $x\in \zer \left( \sum_{i=1}^m A_i\right) $. Thus, we have shown that ``$\subseteq$'' holds in \eqref{eq:equal_zeros}. 

Observe that the convexity of $A_m(x)$, which is a consequence of the maximal monotonicity of $A_m$, plays a crucial role to guarantee that \eqref{eq:equal_zeros} holds. In the absence of this assumption, the set on the left-hand side of \eqref{eq:equal_zeros} may properly contain the right-hand side. 
As a simple example, consider  $\H=\Re$, $A_1\equiv 0$, $A_2 (x) = \frac{1}{2}x-1$ and $A_3(x)=0$ if $x<1$, $A_3(x) = 1$ if $x>1$ and $A_3(1) = \{ 0,1\}$. 
\subsection{A product space reformulation without convex-valuedness}\label{sec:alternative_productspace}
The disadvantage of the reformulation \eqref{eq:FG_reformulation} is that it is not amenable to the general case \eqref{eq:inclusion} if none of the involved operators is maximal monotone, or at the very least, convex-valued\footnote{A set-valued operator $A:\H \toset \H$ is convex-valued if $A(x)$ is a convex subset of $\H$ for any $x\in \H$.}. To be adaptable to the general case and to allow for different weights, we revise the definition of $\widetilde{\K}$ in \eqref{eq:K}. Let $\lambda_1,\ldots ,\lambda_{m-1}\in \Re$, and denote by $\bfLambda:\H^{m-1} \to \H^{m-1}$ the diagonal operator given by 
\begin{equation}
    \bfLambda ( \x )  =  (\lambda_1 x_1, \dots, \lambda_{m-1}x_{m-1}).
    \label{eq:Lambda}
\end{equation} 
Let $\K:\H^{m-1}\toset \H^{m-1}$ be the operator such that $\K (\x) = \{ \bfLambda (\bfDelta_{m-1}(v)) : v\in A_m(x_1)\} $ when $\x \in \D_{m-1}$, and $\K(\x)$ is empty otherwise. That is,

\begin{equation}
    \K (\x ) \coloneqq \begin{cases}
        \{(\lambda_1 v, \ldots, \lambda_{m-1} v): v\in A_m(x_1) \} & \text{if}~\x =(x_1,\dots,x_{m-1})\in \D_{m-1} \\
        \emptyset & \text{otherwise}.
    \end{cases}
    \label{eq:Knew}
\end{equation}
Using this to redefine $\widetilde{\G}$, we can obtain a result parallel to \cref{thm:campoy} without requiring maximal monotonicity. 

\begin{theorem}
    \label{thm:campoy_new}
    Let $A_1,\dots, A_m$ be set-valued operators on $\H$, and let $\F$ be as defined in \eqref{eq:F}. Define $\G : \H^{m-1}\toset \H^{m-1}$ by 
    \begin{equation}
        \G (\x) \coloneqq \K(\x) + N_{\D_{m-1}}(\x),
        \label{eq:Gnew}
    \end{equation}
    where $\K$ is given in \eqref{eq:Knew} for some given $\lambda_1,\dots,\lambda_{m-1} \in \Re$ such that $\sum_{i=1}^{m-1} \lambda_i = 1$. Then
    \begin{equation}
         \zer (\F + \G) = \bfDelta_{m-1}\left( \zer \left( \sum_{i=1}^m A_i\right) \right).
         \label{eq:equal_zeros2}
    \end{equation} 
\end{theorem}
\begin{proof}
\alert{For the inclusion ``$\supseteq$'', let $x\in \zer(\sum_{i=1}^m A_i)$ and set
$\x=\Delta_{m-1}(x)$. Then there exist $u_i\in A_i(x)$, $i=1,\ldots,m$, such
that $\sum_{i=1}^m u_i=0$. Define
$\u=(u_1,\ldots,u_{m-1})\in \F(\x)$ and
$\v=(\lambda_1u_m,\ldots,\lambda_{m-1}u_m)\in \K(\x)$. Since
$\sum_{i=1}^{m-1}\lambda_i=1$, we have
$\sum_{i=1}^{m-1}(u_i+\lambda_i u_m)=\sum_{i=1}^m u_i=0$. Hence
$\u+\v\in \D_{m-1}^{\perp}$, and so $-\u-\v\in \D_{m-1}^{\perp}=N_{\D_{m-1}}(\x)$.
Therefore, $0\in \F(\x)+\K(\x)+N_{\D_{m-1}}(\x)=\F(\x)+\G(\x)$, i.e.,
$\x\in \zer(\F+\G)$.}
    
    To prove the other inclusion, note that given $\x \in \zer (\F+\G)$, we have that $\x = (x,\dots,x) \in \D_{m-1}$ and there exist $\u \in \F(\x)$, $\v \in \K(\x)$ and $\w\in N_{\D_{m-1}}(\x)$ such that $\u+\v+\w=0$. Note that $\v =(\lambda_1 v, \ldots, \lambda_{m-1}v)$ for some $v\in A_m(x)$, and $\sum_{i=1}^{m-1}\lambda_i v = v \in A_m(x)$. \alert{Since $\u+\v=-\w\in \D_{m-1}^{\perp}$, we obtain $\sum_{i=1}^{m-1}(u_i+\lambda_i v)=0$. Using $\sum_{i=1}^{m-1}\lambda_i=1$, this gives $\sum_{i=1}^{m-1}u_i+v=0$. Hence $0\in \sum_{i=1}^m A_i(x)$, and therefore $x\in \zer(\sum_{i=1}^m A_i)$.}
  \end{proof}

 \begin{remark}[\alert{Other choices of $\K$}]\label{rem:otherchoicesofK} \alert{Since $\K$ in \eqref{eq:Knew} is used only through the operator $\G=\K+N_{\D_{m-1}}$, one could also define $\K$ on the whole space $\H^{m-1}$ without changing $\G$, provided the definition agrees with \eqref{eq:Knew} on $\D_{m-1}$. Indeed, $N_{\D_{m-1}}(\x)=\emptyset$ whenever $\x\notin \D_{m-1}$.} 
 
 \alert{Another equivalent choice is obtained from the lifting technique in \cite[Technique~1]{Condat2023}. Namely, define}
\begin{equation}\label{eq:Kcondat}
    \alert{ \K_{\hat{\bfLambda}}(\x):=\{(\hat{\lambda}_1v,\hat{\lambda}_2v,\ldots,\hat{\lambda}_{m-1}v):v\in A_m(x_1)\},} 
\end{equation}
 \alert{with $\hat{\lambda}_i \in [0,1]$ and $\sum_{i=1}^{m-1}\hat{\lambda}_i = 1$. In particular, the choice in \cite[Technique~1]{Condat2023} is $\hat{\lambda}_1 = 1$ and $\hat{\lambda}_i = 0$ for $i=2,\dots,m-1$.} 
 
 \alert{If $\G_{\hat{\bfLambda}}:=\K_{\hat{\bfLambda}}+N_{\D_{m-1}}$, then $\G_{\hat{\bfLambda}}=\G$. Indeed, for $\x\in \D_{m-1}$ and $v\in A_m(x_1)$,} 
 \begin{align}
      \alert{\tilde{\n}} & \alert{\coloneqq (\lambda_1v,\ldots,\lambda_{m-1}v)-(\hat{\lambda}_{1}v,\ldots,\hat{\lambda}_{m-1}v) }\notag \\
      & \alert{= ((\lambda_1-\hat{\lambda}_{1})v,\ldots,(\lambda_{m-1}-\hat{\lambda}_{m-1})v) \in \D_{m-1}^{\perp},}\label{eq:ntilde}
 \end{align}
 \alert{because $\sum_{i=1}^{m-1}\lambda_i=\sum_{i=1}^{m-1}\hat{\lambda}_i=1$.}
 
 \alert{Now, let $\u\in \K(\x)+N_{\D_{m-1}}(\x)$. Then there exist
$v\in A_m(x_1)$ and $\n\in \D_{m-1}^{\perp}$ such that}
$
    \alert{
    \u=(\lambda_1v,\ldots,\lambda_{m-1}v)+\n.}
$
\alert{We can write}
\[
    \alert{
    \u
    =
    (\hat{\lambda}_1v,\ldots,\hat{\lambda}_{m-1}v)
    +
    \left(
    \n+
   \tilde{\n}
    \right),}
\]
\alert{where $  \n+
   \tilde{\n} \in \D_{m-1}^{\perp}
=N_{\D_{m-1}}(\x)$ by \eqref{eq:ntilde}, and hence
$\u\in \K_{\hat{\bfLambda}}(\x)+N_{\D_{m-1}}(\x)$. The reverse inclusion is proved
in the same way. Therefore, $\G_{\hat{\bfLambda}}=\G$.} 
 \end{remark}
 
With \eqref{eq:equal_zeros2}, an equivalent reformulation of \eqref{eq:inclusion} is given by 
\begin{equation}
    \text{Find } \x \in \H^{m-1} ~\text{such that } \mathbf{0}\in \F (\x) + \G (\x),
    \label{eq:FG_reformulation_new}
\end{equation}
without any monotonicity assumptions on the $A_i$'s. The key to this result is that we enforce taking the same element $v\in A_m(x_1)$ when $\x\in \D_{m-1}$ to define the coordinates of elements in $\K(\x)$. This is in contradistinction to the operator $\hat{\K}:\H^{m-1}\toset \H^{m-1}$ defined by 
\begin{equation}
    \alert{\forall \x \in \H^{m-1},} \quad \hat{\K} (\x) \coloneqq \lambda_1 A_m(x_1) \times \cdots \times \lambda_{m-1}A_m(x_{m-1}).
    \label{eq:Khat}
\end{equation}
 Note that $\hat{\K}$ is the natural generalization of $\widetilde{\K}$ given in \eqref{eq:K} in the sense that it permits different weights. However, the domain of $\hat{\K}$ is $\dom (A_m)^{m-1}$, which is larger than the domain of $\K$, namely $\dom (A_m)^{m-1} \cap \D_{m-1}$. Moreover, the image of $\hat{\K}$ at each point $\x \in \D_{m-1}$ is larger than that of $\K$, that is, $\K (\x) \subseteq \hat{\K}(\x)$ for all $\x\in \D_{m-1}$. Nevertheless,  \alert{we have the following proposition stating that whenever $A_m$ is convex-valued and the weights are in $[0,1]$, we can replace $\K$ in \eqref{eq:Gnew} with $\hat{\K}$  and still obtain the same operator $\G$. }

\begin{proposition}
\label{prop:G_Khat}
    Let $\hat{\K}:\H^{m-1}\toset \H^{m-1}$ and $\G:\H^{m-1} \toset \H^{m-1}$ be given by \eqref{eq:Khat} and \eqref{eq:Gnew}, respectively, and suppose that $\lambda_1,\dots, \lambda_{m-1} \in [0,1]$ such that $\sum_{i=1}^{m-1} \lambda_i= 1$.  If $A_m:\H\toset \H$ is convex-valued, then  
    \begin{equation}
        \alert{\forall \x \in \H^{m-1}}, \quad \G (\x) = \hat{\K}(\x)+ N_{\D _{m-1}}(\x).
        \label{eq:G_Khat}
    \end{equation}
\end{proposition}
\begin{proof}
    Both the left-hand and the right-hand sides of \eqref{eq:G_Khat} are empty when $\x\notin \D_{m-1}$. Suppose now that $\x =(x,\dots,x)\in \D_{m-1}$. As mentioned above, $\K(\x) \subseteq \hat{\K}(\x)$, and therefore the inclusion $\G(\x) \subseteq \hat{\K}(\x) + N_{\D_{m-1}}(\x)$ holds. Let $\y \in \hat{\K}(\x) + N_{\D_{m-1}}(\x)$. Then there exists $\v =(v_1,\dots,v_{m-1})\in A_m(x)\times \cdots \times A_m(x)$ such that $\y - \bfLambda \v \in N_{\D_{m-1}}(\x)$.
    Let $v \coloneqq \sum_{i=1}^{m-1}\lambda_i v_i.  $ Since $A_m(x)$ is convex, it follows that $v\in A_m(x)$ and $\bfLambda(\bfDelta_{m-1} (v)) = (\lambda_1 v,\dots,\lambda_{m-1}v) \in \K(\x)$. Moreover, 
    \begin{equation}
       \textstyle  \sum_{i=1}^{m-1} (y_i-\lambda_i v )  =   \sum_{i=1}^{m-1} y_i -  v  =   \sum_{i=1}^{m-1} y_i -  \sum_{i=1}^{m-1} \lambda_i v_i  =0, 
        \label{eq:y-Lambdav}
    \end{equation}
    where the first equality holds since $\sum_{i=1}^{m-1} \lambda_i= 1$, the second holds by the definition of $v$, and the last equality holds since  $\y - \bfLambda \v \in N_{\D_{m-1}}(\x)$. From \eqref{eq:y-Lambdav}, it follows that $\y - \bfLambda(\bfDelta_{m-1} (v))  \in N_{\D_{m-1}}(\x)$. Hence, $\y \in \K(\x) + N_{\D_{m-1}}(\x)$, and therefore $\y \in \G(\x)$. This proves the other inclusion. 
  \end{proof}

\begin{remark}\label{rem:coincidewithcampoyreformulation}
\alert{Suppose that $A_m$ is maximal $\sigma_m$-monotone. By \cref{lemma:maximal_sigma_equivalent} and \cref{lemma:maximalmonotone_properties}(i), $A_m-\sigma_m \Id$ is convex-valued. Consequently, $A_m$ is also convex-valued}. Hence, by \cref{prop:G_Khat}, if $\lambda_1=\cdots=\lambda_{m-1}=1/(m-1)$, the reformulation \eqref{eq:FG_reformulation_new} coincides with Campoy's product space reformulation \eqref{eq:FG_reformulation}. \alert{The discussion in \cref{sec:mainresults_inclusion}} focuses on maximal $\sigma$-monotone operators and can therefore be viewed as an analysis of the Douglas--Rachford algorithm applied to the weighted product space reformulation of Campoy \alert{under generalized monotonicity}. In \cref{sec:nonconvexopt_finitedimensional}, we consider $A_m$ to be the subdifferential of a proper closed function, in which case the operator is generally \textit{not convex-valued}.
\end{remark}

 \subsection{Douglas--Rachford Algorithm}\label{sec:warpedresolvents}
 We now consider the Douglas--Rachford (DR) algorithm to the two-operator reformulation \eqref{eq:FG_reformulation_new} of \eqref{eq:inclusion}. The DR algorithm relies on the computability of elements of the resolvents $J_{\gamma \F}$ and $J_{\gamma \G}$.
The resolvent $J_{\gamma \F}$ is easily derivable due to the structure of $\F$. 
On the other hand, $J_{\gamma \G}$ is not straightforward due to the presence of arbitrary weights $\lambda_1,\dots,\lambda_{m-1}$. To resolve this issue, we use the notion of \textit{warped resolvent} introduced in \alert{\cite{Bui2020,Giselsson2021}}.

\begin{definition}
\label{defn:warpedresolvent}
    Let $A : \H \toset \H$ and $\Lambda:\H \to \H$ be an invertible linear operator on $\H$. The \emph{$\Lambda$-warped resolvent} of $A$ with parameter $\lambda>0$ is defined by $ J_{\lambda A}^{\Lambda}  \coloneqq (\Id +\lambda \alert{\Lambda^{-1}  A} )^{-1}.$
\end{definition}

We now show that for $\bfLambda$ given in \eqref{eq:Lambda}, we can calculate the $\bfLambda$-warped resolvents of $\F$ and $\G$. 
In fact, using the separable structure of $\F$, it is straightforward to derive the following proposition.
\begin{proposition}
\label{prop:F_resolvent}
       Let $\F:\H^{m-1}\toset \H^{m-1}$ be given by \eqref{eq:F} and let $\bfLambda$ be defined by \eqref{eq:Lambda} for some  $\lambda_1,\dots,\lambda_{m-1}\in (0,+\infty)$. For any $\lambda>0$, 
    \begin{align}
        \JLambda{\F}(\x) & = J_{\frac{\lambda}{\lambda_1}A_1} (x_1) \times \cdots \times J_{\frac{\lambda}{\lambda_{m-1}}A_{m-1}}(x_{m-1}), \label{eq:JF_scaled}
    \end{align}
    for any $\x = (x_1,\dots, x_{m-1})\in \H^{m-1}$.  
\end{proposition}

The warped resolvent of $\G$ is derived in the next proposition. 
\begin{proposition}
\label{prop:G_resolvent}
    Let $\G:\H^{m-1}\toset \H^{m-1}$ be given by \eqref{eq:Gnew},  and let $\bfLambda$ be defined by \eqref{eq:Lambda} for some $\lambda_1,\dots,\lambda_{m-1}\in (0,+\infty)$.  
    Then
    \begin{equation}
        J_{\lambda \G}^{\bfLambda}(\x)  = \bfDelta_{m-1} \left( J_{\lambda A_m}\left( \bar{\lambda}^{-1}\sum_{i=1}^{m-1}\lambda_i x_i\right) \right), \quad \bar{\lambda}\coloneqq \sum_{i=1}^{m-1}\lambda_i
        \label{eq:JG_scaled}
    \end{equation}
    for any $\lambda>0$ and any $\x = (x_1,\dots, x_{m-1})\in \H^{m-1}$. Consequently, if $\dom (J_{\alert{ \lambda }A_m}) = \H$, then $\dom (\JLambda{\G}) = \H^{m-1}$. 
\end{proposition}

\begin{proof}
    Let $\x\in \H^{m-1}$. If  $\a \in \JLambda{\G}(\x)$, then $\x \in (\bfId + \lambda \bfLambda^{-1}  \G)(\a)$ so that there exists $\u \in \G(\a) $ such that $\bfLambda\x = \bfLambda \a + \lambda \u$. Meanwhile, since $\G = \K + N_{\D_{m-1}}$, then $\a \in \dom (\G) \subseteq \D_{m-1}$ and there exist $\v \in \K (\a)$, $\n \in N_{\D_{m-1}} (\a)= \D_{m-1}^{\perp}$ such that $\u = \v + \n$. It follows that $\a = (a,\dots, a)$ for some $a\in \H$ and $\v = (\lambda_1v,\dots,\lambda_{m-1}v)$ for some $v \in A_m(a)$. Since $\bfLambda\x = \bfLambda \a + \lambda \u$, we have that $ \lambda \n = \bfLambda\x - \bfLambda \a - \lambda \v  \in \D_{m-1}^\perp$ and therefore $\sum_{i=1}^{m-1}\lambda_i x_i - \bar{\lambda}a - \lambda  \bar{\lambda}v =0$. That is, $\bar{\lambda}^{-1} \sum_{i=1}^{m-1}\lambda_i x_i  = a+ \lambda v $. Since $v\in A_m(a)$, it follows that $a\in J_{\lambda A_m}\left(\bar{\lambda}^{-1}\sum_{i=1}^{m-1}\lambda_i x_i \right) $. In summary, we have shown that if $\a \in \JLambda{\G}(\x)$, then $\a = \bfDelta_{m-1}(a)$ for some $a \in J_{\lambda A_m}\left(\bar{\lambda}^{-1}\sum_{i=1}^{m-1}\lambda_i x_i \right)$, which proves ``$\subseteq$'' in \eqref{eq:JG_scaled}. The other inclusion can be proved by reversing the arguments. For clarity, we include the proof as follows. If $\a = \bfDelta_{m-1}\left(a\right)  $ for some $a \in  J_{\lambda A_m}\left(\bar{\lambda}^{-1}\sum_{i=1}^{m-1}\lambda_i x_i \right)$, then $\bar{\lambda}^{-1}\sum_{i=1}^{m-1}\lambda_i x_i \in a+\lambda A_m(a)$, so that $\bar{\lambda}^{-1}\sum_{i=1}^{m-1}\lambda_i x_i = a+\lambda v$ for some $v\in A_m(a)$. Setting $\v \coloneqq (\lambda_1 v,\dots, \lambda_{m-1}v) \in \K(\a)$, it is easy to verify that $\n \coloneqq \frac{1}{\lambda} (\bfLambda\x - \bfLambda \a - \lambda \v) \in \D_{m-1}^\perp $. Then $\u \coloneqq \v +\n \in \G(\a)$ and $\bfLambda \x = \bfLambda \a + \lambda \u$. Hence, $\a \in \JLambda{\G}(\x)$. This completes the proof. 
  \end{proof}

With the above resolvent formulas, we are now ready to present the Douglas--Rachford algorithm, which is given by the fixed-point iterations
\begin{equation}
    \x^{k+1} \in \dr{\F}{\G} (\x^k), 
    \label{eq:DR_scaled}
\end{equation}
where $\dr{\F}{\G}:\H^{m-1} \toset \H^{m-1}$ is given by 
    \begin{equation}
        \dr{\F}{\G} (\x) \coloneqq \{ \x+\mu (\y - \z ): \z\in \JLambda{\F}(\x), ~\y \in \JLambda{\G} (2\z-\x)\},
        \label{eq:dr_map}
    \end{equation}
$\mu \in (0,2)$, $\lambda >0$ and $\bfLambda$ is the diagonal operator \eqref{eq:Lambda} for some given $\lambda_1,\dots,\lambda_{m-1} \in (0,\infty)$. By the definition of $\dr{\F}{\G}$, we may also write the iterations \eqref{eq:DR_scaled} as 
    \begin{subequations} \label{eq:dr_stepbystep}
            \begin{align}
             \z^k & \in \JLambda{\F}(\x^k) \label{eq:zstep} \\
            \y^k & \in \JLambda{\G}(2\z^k-\x^k) \label{eq:ystep}\\
            \x^{k+1} & = \x^k + \mu (\y^k - \z^k). \label{eq:xstep}
            \end{align}
            \end{subequations}
Using \cref{prop:F_resolvent,prop:G_resolvent}, \eqref{eq:DR_scaled} can be described as in \cref{alg:dr_scaled}\footnote{We note that the forthcoming results in this paper can be generalized to the case when the $x$-update rule is changed to $x_i^{k+1} = x_i^k + \mu_i (y^k-z_i^k)$, where $\mu_i\in (0,2)$. For simplicity, we restrict our discussion to $\mu_1=\cdots=\mu_{m-1}$. }.

\begin{algorithm}%[tb]
    Input initial point $(x^0_1,\dots,x^0_{m-1})\in \H^{m-1}$ and parameters $\mu\in (0,2)$ and $\lambda,\lambda_1,\dots,\lambda_{m-1}\in (0,+\infty)$ with $\sum_{i=1}^{m-1}\lambda_i = 1$. \\
    For $k=0,1,2,\dots ,$
    \begin{equation*}
        \left[\begin{array}{rll}
            z_i^{k}&  \in J_{\frac{\lambda}{\lambda_i}A_i}(x_i^k), & (i=1,\dots,m-1) \\
			\ds y^{k}&   \in J_{\lambda A_m}\left( \sum_{i=1}^{m-1} \lambda_i (2z_i^k-x_i^k)\right) \\
			x_i^{k+1} & = x_i^k + \mu ( y^{k} - z_i^{k} ) & (i=1,\dots,m-1) .
        \end{array}\right.
		\end{equation*}
	\caption{Douglas--Rachford for $m$-operator inclusion problem \eqref{eq:inclusion}.}
	\label{alg:dr_scaled}
\end{algorithm}

Observe that the mapping $\dr{\F}{\G}$ can also be written in terms of the reflected warped resolvents 
\begin{equation}
    \RLambda{\F} = 2\JLambda{\F}-\bfId \quad \text{and} \quad \RLambda{\G} = 2\JLambda{\G}-\bfId.
    \label{eq:reflected_warped_resolvent}
\end{equation}
In particular, 
\begin{equation}
    \dr{\F}{\G}  = \frac{(2-\mu) \bfId + \mu \RLambda{\G} \RLambda{\F} }{2}.
    \label{eq:T_alternative}
\end{equation}
 % Note that switching $\F$ and $\G$ lead to an entirely different algorithm with possibly different convergence guarantee. 
 For the special case that $\lambda_1=\cdots=\lambda_{m-1} = \frac{1}{m-1}$ and $\lambda = \frac{\gamma}{m-1}$ for some $\gamma>0$, the iterations \eqref{eq:DR_scaled} simplify to
\begin{equation}
    \x^{k+1} \in \{ \x^k+ \mu (\y^k - \z^k ): \z^k\in J_{\gamma \F}(\x^k), ~\y^k \in J_{\gamma \G} (2\z^k-\x^k)\},
    \label{eq:DR_classic}
\end{equation}
which is the  \emph{classical Douglas--Rachford} algorithm for  \eqref{eq:FG_reformulation_new} when $\mu=1$.

The goal of \cref{alg:dr_scaled} is to find a fixed point of $\dr{\F}{\G}$, which corresponds to a solution of the inclusion problem \eqref{eq:FG_reformulation_new} as proved in the following proposition.

\begin{proposition}
\label{lemma:JFix=zeros}
    Let $A_i:\H\toset \H$, $i=1,\dots,m$ and let $\lambda,\lambda_1,\dots, \lambda_{m-1}\in (0,+\infty)$  with $\sum_{i=1}^{m-1}\lambda_i = 1$. 
    \begin{enumerate}[(i)]
        \item \alert{If $\x\in \Fix(\dr{\F}{\G})$, then there exists
        $\z\in \JLambda{\F}(\x)$ such that
        $
            \z\in
            \bfDelta_{m-1}\left(
            \zer\left(\sum_{i=1}^m A_i\right)
            \right).
        $}

        \item \alert{If 
          $  \z\in
            \bfDelta_{m-1}\left(
            \zer\left(\sum_{i=1}^m A_i\right)
            \right),
        $
        then there exists $\x\in \Fix(\dr{\F}{\G})$ such that
        $\z\in \JLambda{\F}(\x)$.}
    \end{enumerate}
    Consequently, if $\JLambda{\F}$ is single-valued, then 
    \begin{equation*}
        \JLambda{\F}(\Fix (\dr{\F}{\G})) = \bfDelta_{m-1} \left( \zer \left( \sum_{i=1}^m A_i\right)\right).
    \end{equation*}
\end{proposition}

\begin{proof} \alert{We first show that if $\x\in \Fix(\dr{\F}{\G})$, then there exists $\z\in \zer(\F+\G)$ such that $\z\in \JLambda{\F}(\x)$. Let $\x\in \Fix(\dr{\F}{\G})$. By \eqref{eq:dr_map}, there exists $\z\in \JLambda{\F}(\x)$ such that $\z\in \JLambda{\G}(2\z-\x)$. Hence $\x-\z\in \lambda \bfLambda^{-1}\F(\z)$ and $(2\z-\x)-\z\in \lambda \bfLambda^{-1}\G(\z)$. Thus, there exist $\u\in \F(\z)$ and $\v\in \G(\z)$ such that $\x-\z=\lambda\bfLambda^{-1}\u$ and $\z-\x=\lambda\bfLambda^{-1}\v$. Adding these two equalities gives $\lambda\bfLambda^{-1}(\u+\v)=0$, and therefore $\u+\v=0$. Hence $\bf{0}\in \F(\z)+\G(\z)$, and so $\z\in \zer(\F+\G)$. This proves (i) }

\alert{To prove (ii), let $\z\in \zer(\F+\G)$. Choose $\u\in \F(\z)$ such that $-\u\in \G(\z)$, and define $\x:=\z+\lambda\bfLambda^{-1}\u$. Then $\x-\z=\lambda\bfLambda^{-1}\u\in \lambda\bfLambda^{-1}\F(\z)$, and hence $\z\in \JLambda{\F}(\x)$. Moreover, $(2\z-\x)-\z=\z-\x=-\lambda\bfLambda^{-1}\u \in \lambda\bfLambda^{-1}\G(\z)$, and therefore $\z\in \JLambda{\G}(2\z-\x)$. By \eqref{eq:dr_map}, this gives $\x\in \Fix(\dr{\F}{\G})$. Therefore, the desired correspondence between $\Fix(\dr{\F}{\G})$ and $\zer(\F+\G)$ holds.   }
\end{proof}

In the literature, $\{ \z^k\}$ given in \eqref{eq:zstep} is commonly referred to as the ``shadow sequence'' of the DR algorithm. Its limit (if it converges) represents a solution to the problem, in view of the above proposition.

% Observe that the DR algorithm \eqref{eq:DR_scaled} is a generic algorithm aimed at finding a zero of $\F + \G$. That is, there are no (monotonicity) requirements on the operators $A_1,\dots,A_m$.  Similarly, \cref{lemma:JFix=zeros} also holds for arbitrary operators $A_1,\dots,A_m$. However, to establish the convergence of \cref{eq:DR_scaled}, we need additional assumptions on the operators, which is the main agenda in the next sections. 

Observe that the DR algorithm  \eqref{eq:DR_scaled} is defined for arbitrary $A_1,\dots,A_m$, provided the
relevant resolvents exist at the iterates, i.e., no monotonicity is needed to write the algorithm.
Likewise, \cref{lemma:JFix=zeros} identifies zeros of $\F+\G$ with fixed points of
the DR operator without invoking monotonicity. Convergence, however, does require
additional assumptions, which we establish in the next sections.
\section{Douglas--Rachford algorithm for inclusion problems under generalized monotonicity}
\label{sec:mainresults_inclusion}

In this section, we prove the convergence of the DR algorithm \eqref{eq:DR_scaled} under the assumption that each operator $A_i$ is maximal $\sigma_i$-monotone. 

\subsection{Further properties under generalized monotonicity}

We show that generalized (maximal) monotonicity of the operators $A_i:\H\toset\H$ is inherited by the operators $\F,\G:\H^{m-1}\toset \H^{m-1}$. We establish first \alert{that} $\F$ is \alert{generalized} maximal monotone for some modulus. 
\begin{proposition}
    \label{prop:F_monotone}
    Suppose that $A_i:\H\toset \H$ is $\sigma_i$-monotone for $i=1,\dots, m-1$. Then $\F$ given by \eqref{eq:F} is $\sigmabf{\F}$-monotone with $\sigmabf{\F}\coloneqq \min_{i=1,\dots,m-1}\sigma_i$. Furthermore, if each $A_i$ is maximal $\sigma_i$-monotone, then $\F$ is maximal $\sigmabf{\F}$-monotone. 
\end{proposition}
\begin{proof}
Let $(\x,\u), (\y,\v)\in \gra (\F)$. Assuming that $A_i$ is $\sigma_i$-monotone for all $i=1,\dots,m-1$, we have
    \begin{align*}
        \lla \x -\y, \u - \v \rla = \sum_{i=1}^{m-1}\lla x_i-y_i, u_i-v_i \rla \geq \sum_{i=1}^{m-1}\sigma_i \norm{x_i-y_i}^2 \geq \sigmabf{\F}\norm{\x-\y}^2.
    \end{align*}
% where $\sigmabf{\F}\coloneqq \min_{i=1,\dots,m-1}\sigma_i$. 
Hence, $\F$ is $\sigmabf{\F}$-monotone. Assume now that each $A_i$ is maximal $\sigma_i$-monotone and let $\gamma>0$ such that $1+\gamma \sigmabf{\F}>0$. Then $1+\gamma\sigma_i>0$ for all $i=1,\dots,m-1$, and since $A_i$ is maximal $\sigma_i$-monotone, we have from \cref{lemma:maximal_single-valued-resolvent} that $\dom (J_{\gamma A_i})=\H$. It follows from \cref{prop:F_resolvent} that $\dom (J_{\gamma \F})=\H^{m-1}$. Hence, $\F$ is maximal $\sigmabf{\F}$-monotone by \cref{lemma:maximal_single-valued-resolvent}. 
  \end{proof}

As for $\G$, we first establish its monotonicity in the following result. 

\begin{proposition}
\label{prop:G_monotone}
    Suppose that $A_m$ is $\sigma_m$-monotone. Then $\G$ given by \eqref{eq:Gnew} is $\left(\frac{\sigma_m\bar{\lambda}}{m-1}\right)$-monotone, where $\bar{\lambda}\coloneqq \sum_{i=1}^{m-1}\lambda_i$. 
\end{proposition}

\begin{proof}
     Let $(\x,\u), (\y,\v)\in \gra (\G)$. Then $\x = (x,\dots, x) $ and $\y = (y,\dots, y)$ for some $x,y\in \dom (A_m)$, while $\u = \bfLambda (u',\dots,u')+\n_{\u}$ and $\v = \bfLambda(v',\dots,v') + \n_{\v}$ for some $u' \in A_m(x)$, $v'\in A_m(y)$ and $\n_{\u},\n_{\v}\in \D_{m-1}^\perp$. 
     \begin{align}
        \lla \x -\y, \u - \v \rla & = \sum_{i=1}^{m-1}\lla x-y, \lambda_i u' - \lambda_i v' \rla + \lla \x - \y , \n_{\u}-\n_{\v}\rla  = \sum_{i=1}^{m-1}\lambda_i \lla x-y,  u' - v' \rla \notag  
    \end{align}
   where the second equality holds by the definition of orthogonal complement. Using the $\sigma_m$-monotonicity of $A_m$ gives the desired conclusion. 
  \end{proof}

\alert{Now, we establish the maximal $\left(\frac{\sigma_m}{m-1}\right)$-monotonicity of $\G$ for arbitrary weights $\lambda_i\in(0,+\infty)$ with $\sum_{i=1}^{m-1}\lambda_i=1$. }
\begin{proposition}
\label{prop:G_monotone3}
Suppose that $A_m$ is maximal $\sigma_m$-monotone. Let $\G$ be given by
\eqref{eq:Gnew} with \alert{$\lambda_1,\dots,\lambda_{m-1}\in (0,+\infty)$ and
$\sum_{i=1}^{m-1}\lambda_i=1$}. Then $\G$ is maximal
$\left(\frac{\sigma_m}{m-1}\right)$-monotone.
\end{proposition}
\begin{proof}
 \alert{ By \cref{prop:G_monotone}, $\G$ is
  $\frac{\sigma_m}{m-1}$-monotone. Hence, by
  \cref{lemma:maximal_single-valued-resolvent}, the claim of the proposition holds if
  $\dom(J_{\gamma \G})=\H^{m-1}$ whenever
  $1+\gamma\frac{\sigma_m}{m-1}>0$.}

  \alert{To this end, we choose a convenient representation of $\G$. By
  \cref{rem:otherchoicesofK}, we may write
  $\G=\K_{\hat{\bfLambda}}+N_{\D_{m-1}}$, where $\K_{\hat{\bfLambda}}$ is given by
  \eqref{eq:Kcondat}. In particular, we choose
  $\hat{\lambda}_1=\cdots=\hat{\lambda}_{m-1}=\frac{1}{m-1}$.} \alert{Applying
  \cref{prop:G_resolvent} to this equal-weight representation with parameter
  $\lambda = \frac{\gamma}{m-1}$ gives}
  \[
      \alert{J_{\gamma\G}(\x)
      =
      J^{\frac{1}{m-1}\Id}_{\frac{\gamma}{m-1}\G}(\x)
      =
      \bfDelta_{m-1}\left(
      J_{\frac{\gamma}{m-1}A_m}
      \left(
      \frac{1}{m-1}\sum_{i=1}^{m-1}x_i
      \right)
      \right).}
  \]
  \alert{Thus, it suffices to show that
  $\dom(J_{\frac{\gamma}{m-1}A_m})=\H$ whenever
  $1+\gamma\frac{\sigma_m}{m-1}>0$. This follows from
  \cref{lemma:maximal_single-valued-resolvent}, since $A_m$ is maximal
  $\sigma_m$-monotone. Therefore, $\dom(J_{\gamma\G})=\H^{m-1}$, and so
  $\G$ is maximal $\frac{\sigma_m}{m-1}$-monotone.}
 
\end{proof}

We next establish some properties of the \textit{reflected $\bfLambda$-warped resolvents} of $\F$ and $\G$, given by \eqref{eq:reflected_warped_resolvent}.

\begin{proposition}[Properties of reflected warped resolvents]
    Let $A_i:\H\rightrightarrows \H$ be $\sigma_i$-monotone for each $i=1,\dots,m$. 
    % with $\dom (J_{A_m})=\H$. 
    Let $\lambda,\lambda_1,\dots, \lambda_{m-1}\in (0,+\infty)$, $\bfLambda$ be given by \eqref{eq:Lambda}, and define
   \begin{equation}
   \textstyle \innerlambda{\x}{\y} \coloneqq \inner{\x}{\bfLambda \y} = \sum_{i=1}^{m-1} \lambda_i \inner{x_i}{y_i}\quad \text{and} \quad \normlambda{\x}  \coloneqq \sqrt{\innerlambda{\x}{\x}} ,
       \label{eq:newinnerproduct}
   \end{equation}
    for any $\x = (x_1,\dots,x_{m-1}), \y = (y_1,\dots,y_{m-1})\in \H^{m-1}.$
     \begin{enumerate}[(i)]
        \item 
        $\normlambda{{\a}' - {\b}'}^2 \leq \normlambda{\x-\y}^2 - 4\lambda \sum_{i=1}^{m-1}\sigma_i \norm{a_i-b_i}^2,$ for any $(\x,{\a}'), (\y,{\b}')\in \gra (\RLambda{\F})$,
        where $\a=(a_1,\dots,a_{m-1})\in \JLambda{\F}(\x)$ and $\b = (b_1,\dots,b_{m-1}) \in \JLambda{\F}(\y)$ are such that ${\a}'=2\a - \x$ and ${\b}'= 2\b - \y$. 

        \item 
        $\normlambda{\a' - \b '}^2 \leq \normlambda{\x-\y}^2 - 4\lambda\sigma_m \normlambda{\a-\b}^2, $ for any $(\x,\a'), (\y,\b')\in \gra (\RLambda{\G})$, 
        where $\a\in \JLambda{\G}(\x)$ and $\b \in \JLambda{\G}(\y)$ are such that ${\a}'=2\a - \x$ and ${\b}'= 2\b - \y$. 
    \end{enumerate}
\label{prop:properties_reflected_resolvents}
\end{proposition}
\begin{proof}
    We first prove part (i). Since $(\x,\a), (\y,\b)\in \gra (\JLambda{\F})$, we have $\x \in (\bfId + \lambda \bfLambda^{-1}  \F ) (\a)$ and  $\y \in (\bfId + \lambda \bfLambda^{-1}  \F ) (\b)$. Thus, there exist $\u \in \F(\a)$ and $\v \in \F(\b)$ such that  $\bfLambda \x = \bfLambda \a + \lambda \u$ and $\bfLambda \y = \bfLambda \b + \lambda \v$.  Consequently,
    \begin{align}
        \innerlambda{\x-\y}{\a-\b}  & = \inner{\bfLambda (\x -\y )}{\a-\b} 
        = \inner{\bfLambda (\a - \b) + \lambda (\u - \v)}{\a - \b}  \notag \\
        % & = \normlambda{\a - \b}^2 + \lambda \inner{\u - \v}{\a - \b} \notag \\
        & \textstyle  = \normlambda{\a - \b}^2+ \lambda \sum_{i=1}^{m-1} \inner{u_i-v_i}{a_i-b_i} \notag \\ 
        & \textstyle  \geq \normlambda{\a - \b}^2 + \lambda \sum_{i=1}^{m-1} \sigma_i \norm{a_i-b_i}^2 , \label{eq:JF-cocoercive}
    \end{align}
    where we have used the $\sigma_i$-monotonicity of $A_i$ in the last inequality. On the other hand, 
    \begin{align}
        \normlambda{\a' - \b'}^2 & = \normlambda{2(\a -\b) - (\x - \y)} ^2 = \normlambda{\x-\y}^2 - 4\innerlambda{\x-\y}{\a-\b} + 4\normlambda{\a-\b}^2. \label{eq:identity_normlambda}
    \end{align}
    Combining this with \eqref{eq:JF-cocoercive} proves the claim of part (i).

    To prove part (ii), we follow the same argument in part (i) to show that 
    \begin{equation}
        \innerlambda{\x - \y}{\a - \b} = \normlambda{\a - \b}^2 + \lambda \inner{\u-\v}{\a-\b},
        \label{eq:innerlambda}
    \end{equation}
where $\u \in \G(\a)$ and $\v \in \G (\b)$ such that  $\bfLambda \x = \bfLambda \a + \lambda \u$ and $\bfLambda \y = \bfLambda \b + \lambda \v$. By the definition of $\G$, there exist $\u'\in \K (\a)$, $\v'\in \K (\b)$ and $\n_1,\n_2 \in \D_{m-1}^{\perp}$ such that $\u = \u'+\n_1$ and $\v = \v' + \n_2$. Meanwhile, since $\a,\b \in \dom (\G) \subseteq \D_{m-1}$, then $\a = (a,\dots, a)$ and $\b = (b,\dots, b)$ for some $a,b\in \H$. Hence, if $\u' =(u'_1,\dots, u'_{m-1})$  and $\v' =(v'_1,\dots, v'_{m-1})$, then $u'_i\in \lambda_i A_m(a)$ and $v'_i 
\in \lambda_i A_m(b)$ for all $i$.   By the $\sigma_m$-monotonicity of $A_m$, it follows that  
    \begin{align*}
        \inner{\u-\v}{\a-\b} & \textstyle  = \inner{\u'-\v '}{\a - \b} + \inner{\n_1 - \n_2}{\a - \b}  = \sum_{i=1}^{m-1}\inner{u'_i - v'_i}{a-b} \\
        & \textstyle \geq \sum_{i=1}^{m-1} \sigma_m \lambda_i \norm{a-b}^2  = \sigma_m \normlambda{\a-\b}^2,
    \end{align*}
where the second equality holds by definition of orthogonal complement. Together with \eqref{eq:innerlambda}, we get
    \begin{equation}
       \innerlambda{\x - \y}{\a - \b} \geq (1+\lambda \sigma_m) \normlambda{\a - \b}^2. 
       \label{eq:JG_cocoercive}
    \end{equation}
Combining this with the identity \eqref{eq:identity_normlambda} proves part (ii). 
  \end{proof}

% An important consequence of \eqref{eq:JG_cocoercive} that was derived in the proof of \cref{prop:properties_reflected_resolvents}(ii) is the following result on the single-valuedness of the warped resolvent of $\G$.

% \begin{proposition}
% \label{prop:G_resolvent_maximal}
%     If $A_m:\H \to \H$ is maximal $\sigma_m$-monotone and $1+\lambda \sigma_m>0$, then $\JLambda{\G}$ is single-valued on $\H^{m-1}$ and 
%           \begin{equation}
%               J_{\lambda \G}^{\bfLambda}(\x)  =  \bfDelta_{m-1} \left( J_{\lambda A_m}\left( \bar{\lambda}^{-1}\sum_{i=1}^{m-1}\lambda_i x_i\right) \right), \quad \bar{\lambda}\coloneqq \sum_{i=1}^{m-1}\lambda_i,
%               \label{eq:JG_scaled2}
%           \end{equation}
%     for all $\x \in \H^{m-1}$.
% \end{proposition}
% \begin{proof}
%     By \eqref{eq:JG_cocoercive}, the hypothesis that $1+\lambda\sigma_m>0$ and the Cauchy-Schwarz inequality, it follows that $\JLambda{\G}$ is single-valued on its domain. Meanwhile, we have from \cref{lemma:maximal_single-valued-resolvent}(ii) that $\dom(A_m) = \H$. Hence, we see from \cref{prop:G_resolvent} that $\dom(\JLambda{\G}) = \H^{m-1}$. Invoking the single-valuedness of $\JLambda{\G}$ and using \eqref{eq:JG_scaled}, we obtain the formula \eqref{eq:JG_scaled2}. 
%   \end{proof}

\begin{remark}
\label{rem:cocoercive_and_singlevalued}
    We have
    $\inner{\x-\y}{\a-\b}_{\bfLambda}\leq \lambda_{\max} \norm{\x-\y} \norm{\a-\b}$
    by the Cauchy-Schwarz inequality, where $\ds \lambda_{\max} \coloneqq \max_{i=1,\dots,m-1}\lambda_i$. Thus, we have from \eqref{eq:JF-cocoercive} that 
        \begin{equation}
          \textstyle   \norm{\a - \b} \leq \frac{ \lambda_{\max} }{ \min_{i=1,\dots,m-1} (\lambda_i+\lambda \sigma_i)}\norm{\x-\y} \quad \forall (\x,\a), (\y,\b) \in \gra (\JLambda{\F}),
            \label{eq:JF_cocoercive_simplified}
        \end{equation}
    provided that $\lambda_i+\lambda\sigma_i>0$ for all $i=1,\dots,m-1$. Hence, $\JLambda{\F}$ is single-valued on its domain whenever the latter condition holds. On the other hand, we have from \eqref{eq:JG_cocoercive} that 
        \begin{equation}
           \textstyle  \normlambda{\a-\b}\leq \frac{1}{(1+\lambda\sigma_m) }\normlambda{\x-\y} \quad \forall   (\x,\a), (\y,\b) \in \gra (\JLambda{\G})
            \label{eq:JG_cocoercive_simplified}
        \end{equation}
    provided $1+\lambda\sigma_m>0$,  in which case, $\JLambda{\G}$ is single-valued on its domain. 
\end{remark}

 \subsection{Convergence results}
 First, we present the following proposition, which is a  straightforward application of the existing convergence results for the Douglas--Rachford algorithm for two-operator inclusion.

\begin{proposition}
    \label{prop:naive_douglas}
     Let $A_i:\H\toset \H$ be maximal $\sigma_i$-monotone for each $i=1,\dots,m$, and assume that $\zer\left( A_1+\cdots + A_m\right)\neq \emptyset$. Let $(\mu,\gamma)$ in \alert{\eqref{eq:DR_classic}} satisfy $\mu \in (0,2)$, $\gamma\in (0,+\infty)$, and suppose that either one of the following holds:
     \begin{enumerate}[(A)]
        \item $\ds \alert{\sigmabf{\F}} + \frac{\sigma_m}{m-1}>0$ and $1+\gamma \frac{\alert{\sigmabf{\F}}\sigma_m}{\alert{\sigmabf{\F}} (m-1)+ \sigma_m}>\frac{\mu}{2}$; or 
         \item $\ds \alert{\sigmabf{\F}}= \sigma_m =0$
     \end{enumerate}
     where $\ds \alert{\sigmabf{\F}} \coloneqq \min_{i=1,\dots,m-1}\sigma_i $. If $\{ \x^k\}$ is a sequence generated by \eqref{eq:DR_classic} from an arbitrary initial point $\x^0\in \H^{m-1}$, then there exists $\bar{\x}\in \Fix (\dr{\F}{\G})$ such that $\x^k\toweak \bar{\x}$ and $J_{\gamma 
     \F}(\bar{\x}) \in \bfDelta_{m-1} \left( \zer \left( \sum_{i=1}^m A_i\right)\right)$ with $\norm{(\bfId - \dr{\F}{\G})\x^k}  = o(1/\sqrt{k})$ as $k\to\infty$. Under the conditions in (A),  $J_{\gamma 
     \F}(\x^k)\to J_{\gamma 
     \F}(\bar{\x})$, $J_{\gamma 
     \G}R_{\gamma 
     \F}(\x^k) \to J_{\gamma 
     \F}(\bar{\x})$, and $\bfDelta_{m-1} \left( \zer \left( \sum_{i=1}^m A_i\right)\right) = \{ J_{\gamma 
     \F}(\bar{\x}) \}$. 
\end{proposition}
\begin{proof}
    From \cref{prop:F_monotone} and \cref{prop:G_monotone3}, we know that $\F$ is maximal $\alert{\sigmabf{\F}}$-monotone and $\G$ is maximal $\frac{\sigma_m}{m-1}$-monotone. The result then immediately follows from \cite[Theorem 4.5(ii)]{DaoPhan2019}.
  \end{proof}

As indicated in its proof, \cref{prop:naive_douglas} is a direct application of \cite[Theorem 4.5(ii)]{DaoPhan2019}, which provides the convergence of the  Douglas--Rachford algorithm (with equal weights $\lambda_1,\dots,\lambda_{m-1}$) for finding the zeros of the sum of two operators. It is worth noting that \cref{prop:naive_douglas}(A) also covers the situation where one operator among $A_1,\ldots,A_{m-1}$ is only $\sigma_i$-weakly monotone. In that case the condition enforces $
\sigma_m >-(m-1)\sigma_i,$
i.e., the strong monotonicity modulus required of $A_m$ grows linearly with $m$. For large $m$, this demands an impractically large modulus.

% Our main contribution involves improving upon \cref{prop:naive_douglas} in three ways: (i) we relax the condition $\min_{i=1,\dots,m-1}\sigma_i + \frac{\sigma_m}{m-1}>0$ to only require that $\sigma_1+\cdots+\sigma_m >0$; (ii) we provide larger window of stepsizes that guarantee convergence; and (iii) we derive results that hold for arbitrary weights $\lambda_1,\dots,\lambda_{m-1}$. 
% To this end, we draw inspiration from the techniques in \cite[Theorems 4.2 and 4.5(ii)]{DaoPhan2019}, which leverage the notion of Fej\'{e}r monotonicity, as is typical in the convergence analysis of algorithms involving monotone operators. 

Our contribution strengthens \cref{prop:naive_douglas} in three directions (see also \cref{rem:comparewithnaive}):
(i) we relax the moduli requirement from
$
\min_{i=1,\dots,m-1}\sigma_i+\tfrac{\sigma_m}{m-1}>0$
to the significantly weaker condition $
\sigma_1+\cdots+\sigma_m>0;$
(ii) we obtain a potentially strictly larger admissible stepsize window ensuring convergence; and
(iii) we establish convergence for arbitrary weights $\lambda_1,\dots,\lambda_{m-1}$ (not just the uniform choice).
Our analysis adapts and extends the techniques of \cite[Thms.~4.2 and 4.5(ii)]{DaoPhan2019}, relying on Fej\'{e}r monotonicity, a standard tool in convergence proofs for algorithms with monotone operators.

In the remainder of this paper, we also introduce the following notations: Given $\sigma_1,\dots,\sigma_{m-1}\in \Re$, we let 
\begin{equation}
    \begin{array}{rl}
        \I~ & \coloneqq \{ i\in \{1,\dots,m-1\} : \sigma_i \neq 0\},  \\ 
        \I^- & \coloneqq \{ i\in \I : \sigma_i < 0\} \quad \text{and} \quad \I^+ \coloneqq \I \setminus \I^-.
    \end{array}\label{eq:I_index}
\end{equation}

\alert{Before proceeding, we note that when $\I=\emptyset$ and $\sigma_m\geq 0$ (as in Condition (B) of Theorem 3), the convergence of
\cref{alg:dr_scaled} can be obtained directly from the classical
Douglas--Rachford theory \cite[Theorem~26.11]{Bauschke2017}, applied in the
weighted product space to the operators $\lambda\bfLambda^{-1}\F$ and
$\lambda\bfLambda^{-1}\G$. Nevertheless, we include the case $\I=\emptyset$ in
the estimates below for completeness. This is natural in the present treatment,
since the estimates are derived from identities that are valid for arbitrary
$\I$, and the purely monotone case is recovered simply by taking
$\I=\emptyset$ and $\sigma_m\geq 0$.}

\begin{proposition}
\label{prop:T_nonexpansive}
Let $A_i:\H\rightrightarrows \H$ be $\sigma_i$-monotone for each $i=1,\dots,m$, let $\lambda,\lambda_1,\dots, \lambda_{m-1}\in (0,+\infty)$  with $\sum_{i=1}^{m-1}\lambda_i = 1$ and let $\bfLambda$ be given by \eqref{eq:Lambda}. Suppose that $\JLambda{\F}$ and $\JLambda{\G}$ are single-valued on their domains. Define $U:\H^{m-1}\toset \H$ by 
$U(\x) \coloneqq  J_{\lambda A_m}\left( \sum_{i=1}^{m-1}\lambda_i R_{\frac{\lambda}{\lambda_i}A_i}(x_i)\right). $
Then the following hold:
\begin{enumerate}[(i)]
    \item The mappings $U$ and $\JLambda{\G}\RLambda{\F}$  are single-valued on $\dom (\dr{\F}{\G})$, and \\ $ \JLambda{\G}\RLambda{\F}(\x) = \bfDelta_{m-1}(U(\x))$. 
    \item Denote $\R \coloneqq  \bfId - \dr{\F}{\G}$ and its components $\R = (R_1,\dots,R_{m-1})$. Then 
        \begin{equation}
        \frac{1}{\mu} R_i(\x) =  J_{\frac{\lambda}{\lambda_i}A_i}(x_i) - U(\x) 
        \label{eq:Id-T}
    \end{equation} 
    for each $i=1,\dots,m-1$. 
   
    \item  Suppose that either $\I=\emptyset$, or $\I\neq\emptyset$ and
    $(\delta_i)_{i\in\I}$ satisfies
    $\sigma_i+\sigma_m\delta_i\neq0$ for every $i\in\I$ and
    $\sum_{i\in\I}\delta_i=1$. Then for any $\x,\y \in \dom (\dr{\F}{\G})$,
    \begin{align}
    & \textstyle \normlambda{\dr{\F}{\G} (\x) - \dr{\F}{\G} (\y) }^2  \leq  \normlambda{\x - \y}^2     -\frac{2}{\mu}\sum_{i=1}^{m-1} \lambda_i \kappa_i \norm{R_i(\x) - R_i(\y)}^2 \notag \\ 
    & \textstyle - 2\mu\lambda \sum_{i\in \I} \theta_i \norm{\sigma_i\left(  J_{\frac{\lambda}{\lambda_i}A_i} (x_i) -  J_{\frac{\lambda}{\lambda_i}A_i} (y_i)\right) + \sigma_m \delta_i (U(\x)-U(\y)) }^2 \notag  \\
    & \textstyle - 2\alpha \mu \lambda \sigma_m  \norm{ U(\x)-U(\y)}^2 ,\label{eq:T_nonexpansive}
    \end{align}
    where 
    \begin{equation}
        \alpha \coloneqq \begin{cases}
        0 & \text{if}~\I \neq \emptyset \\
        1 & \text{if}~\I = \emptyset
    \end{cases}, ~ \kappa_i \coloneqq \begin{cases}
            1+ \frac{\lambda}{\lambda_i}\frac{ \sigma_i\sigma_m\delta_i }{\sigma_i + \sigma_m\delta_i} - \frac{\mu}{2} & \text{if}~i\in \I \\
            1-\frac{\mu}{2} & \text{if}~i\notin \I
        \end{cases},~\theta_i \coloneqq \frac{1}{\sigma_i +\sigma_m \delta_i}. \label{eq:alpha,kappa,theta}
    \end{equation}
\end{enumerate}
\end{proposition}
\begin{proof}
We have $\dr{\F}{\G} = \bfId +\mu ( \JLambda{\G}\RLambda{\F} - \JLambda{\F})$ by noting \eqref{eq:dr_map} and the single-valuedness hypotheses. Then $\dom (\dr{\F}{\G} )= \dom (\JLambda{\G}\RLambda{\F})$ and $\JLambda{\G}\RLambda{\F}$ is single-valued on $\dom (\dr{\F}{\G} )$. The formula $ \JLambda{\G}\RLambda{\F}(\x) = \bfDelta_{m-1}(U(\x))$ holds by  \cref{prop:F_resolvent} and \cref{prop:G_resolvent}. From this formula, we also see that $U$ is single-valued on $\dom (\dr{\F}{\G})$. This proves part (i). Part (ii) follows from part (i) and the identity 
\begin{equation}
     \bfId - \dr{\F}{\G}= \mu ( \JLambda{\F} - \JLambda{\G}\RLambda{\F}).
     \label{eq:Id-T=mu(J-J)}
\end{equation}

We now prove part (iii). Using \eqref{eq:identity_squarednorm} and the equivalent expression for $\dr{\F}{\G}$ given in \eqref{eq:T_alternative}, we have 
    \begin{align}
       \textstyle  \normlambda{\dr{\F}{\G} (\x) - \dr{\F}{\G} (\y) }^2 =  \textstyle & \frac{2-\mu}{2}\normlambda{\x - \y}^2 + \frac{\mu}{2}\normlambda{\RLambda{\G}\RLambda{\F}(\x) - \RLambda{\G}\RLambda{\F}(\y)}^2 \notag \\
        & \textstyle  - \frac{\mu(2-\mu)}{4}\normlambda{(\bfId - \RLambda{\G}\RLambda{\F})(\x) - (\bfId - \RLambda{\G}\RLambda{\F})(\y)}^2 .\label{eq:T_diff}
    \end{align}
    From \eqref{eq:T_alternative}, we also obtain that $\bfId - \RLambda{\G}\RLambda{\F} = \frac{2}{\mu}(\bfId - \dr{\F}{\G}) = \frac{2}{\mu}\R$. Then, we further obtain from \eqref{eq:T_diff} that 
    \begin{align}
       \textstyle  \normlambda{\dr{\F}{\G} (\x) - \dr{\F}{\G} (\y) }^2 = & \textstyle \frac{2-\mu}{2}\normlambda{\x - \y}^2 + \frac{\mu}{2}\normlambda{\RLambda{\G}\RLambda{\F}(\x) - \RLambda{\G}\RLambda{\F}(\y)}^2 \notag \\
        & \textstyle - \frac{2-\mu}{\mu} \sum_{i=1}^{m-1}\lambda_i\norm{R_i(\x) - R_i(\y)}^2 .\label{eq:T_diff2}
    \end{align}
    Meanwhile, noting the single-valuedness of $\JLambda{\F}$ and $\JLambda{\G}$, we have
        \begin{align}
            & \normlambda{\RLambda{\G}\RLambda{\F}(\x) - \RLambda{\G}\RLambda{\F}(\y)}^2 \notag \\ 
            & \leq  \normlambda{\RLambda{\F}(\x) -\RLambda{\F}(\y)}^2 - 4\lambda \sigma_m \normlambda{\JLambda{\G}\RLambda{\F}(\x) - \JLambda{\G}\RLambda{\F}(\y)  }^2 \notag  \\
            & \textstyle \leq  \normlambda{\x - \y}^2 - 4\lambda \sum_{i=1}^{m-1}\sigma_i \norm{J_{\frac{\lambda}{\lambda_i} A_i}(x_i) - J_{\frac{\lambda}{\lambda_i} A_i}(y_i)}^2 \notag \\
            & - 4\lambda \sigma_m \normlambda{\JLambda{\G}\RLambda{\F}(\x) - \JLambda{\G}\RLambda{\F}(\y)  }^2 ,\label{eq:composition_diff}
        \end{align}
    where the first inequality holds by \cref{prop:properties_reflected_resolvents}(ii), while the second holds by combining \cref{prop:properties_reflected_resolvents}(i) and \cref{prop:F_resolvent}.  \alert{ When} $\I\neq \emptyset$, we have
\begin{align}
  &\textstyle  \sum_{i=1}^{m-1} \sigma_i \norm{J_{\frac{\lambda}{\lambda_i}A_i} (x_i) -  J_{\frac{\lambda}{\lambda_i}A_i} (y_i)}^2 + \sigma_m \normlambda{\JLambda{\G}\RLambda{\F}(\x) - \JLambda{\G}\RLambda{\F}(\y)}^2 \notag \\
 &\textstyle  = \sum_{i\in \I} \sigma_i \norm{J_{\frac{\lambda}{\lambda_i}A_i} (x_i) -  J_{\frac{\lambda}{\lambda_i}A_i} (y_i)}^2 + \sigma_m \normlambda{\JLambda{\G}\RLambda{\F}(\x) - \JLambda{\G}\RLambda{\F}(\y)}^2 \notag \\
& \textstyle \overset{(a)}{=}  \sum_{i\in \I}  \sigma_i \norm{J_{\frac{\lambda}{\lambda_i}A_i} (x_i) -  J_{\frac{\lambda}{\lambda_i}A_i} (y_i)}^2 + \sigma_m  \norm{U(\x) - U(\y)}^2 \notag \\
&\textstyle  \overset{(b)}{=} \sum_{i\in \I}   \left( \sigma_i \norm{J_{\frac{\lambda}{\lambda_i}A_i} (x_i) -  J_{\frac{\lambda}{\lambda_i}A_i} (y_i)}^2 + \sigma_m  \delta_i\norm{U(\x) - U(\y)}^2\right) \notag \\
&\textstyle  \overset{(c)}{=} \sum_{i\in \I}  \frac{\sigma_i \sigma_m \delta_i}{\sigma_i +\sigma_m \delta_i} \norm{\left(  J_{\frac{\lambda}{\lambda_i}A_i} (x_i) -  J_{\frac{\lambda}{\lambda_i}A_i} (y_i)\right) - (U(\x)-U(\y)) }^2 \notag \\
&\textstyle \quad + \sum_{i\in \I}  \frac{1}{\sigma_i +\sigma_m \delta_i} \norm{\sigma_i\left(  J_{\frac{\lambda}{\lambda_i}A_i} (x_i) -  J_{\frac{\lambda}{\lambda_i}A_i} (y_i)\right) + \sigma_m \delta_i (U(\x)-U(\y)) }^2 \notag \\
&\textstyle  \overset{(d)}{=} \frac{1}{\mu^2}\sum_{i\in \I} \frac{\sigma_i \sigma_m \delta_i}{\sigma_i +\sigma_m \delta_i} \norm{R_i(\x) - R_i(\y) }^2  \notag \\
&\textstyle  \quad + \sum_{i\in \I}  \frac{1}{\sigma_i +\sigma_m \delta_i} \norm{\sigma_i\left(  J_{\frac{\lambda}{\lambda_i}A_i} (x_i) -  J_{\frac{\lambda}{\lambda_i}A_i} (y_i)\right) + \sigma_m \delta_i (U(\x)-U(\y)) }^2 \label{eq:lasttwoterms} ,
\end{align}
where (a) holds by part (i); (b) holds since $\sum_{i\in \I}\delta_i = 1$; (c) holds by \eqref{eq:identity_squarednorm2}; and (d) holds by part (ii). Combining \eqref{eq:T_diff2}, \eqref{eq:composition_diff} and \eqref{eq:lasttwoterms}, we obtain  \eqref{eq:T_nonexpansive}. \alert{On the other hand, suppose that} $\I=\emptyset$. Then $\sigma_i=0$ for all $i=1,\dots,m-1$, and we obtain the inequality \eqref{eq:T_nonexpansive} by combining \eqref{eq:T_diff2}, \eqref{eq:composition_diff}, \alert{and the identity used in step (a)}.
  \end{proof}

\begin{theorem}
    \label{thm:general_convergence}
    Let $A_i:\H\toset \H$ be maximal $\sigma_i$-monotone for each $i=1,\dots,m$, and assume that $\zer\left( A_1+\cdots + A_m\right)\neq \emptyset$. Let $\mu\in (0,2)$, $\lambda_1,\dots, \lambda_{m-1}\in (0,+\infty)$ with $\sum_{i=1}^{m-1}\lambda_i = 1$, and let $\bfLambda$ be given by \eqref{eq:Lambda}. Let $\I$ be given by \eqref{eq:I_index}, and suppose that either one of the following holds:
    \begin{enumerate}[(A)]
        \item  $\I \neq \emptyset$, there exists $(\delta_i)_{i\in \I}$ such that $\sigma_i + \sigma_m \delta_i >0$ for all $i\in \I$ and $\sum_{i\in \I}\delta_i = 1$, and $\lambda\in (0,+\infty)$ is chosen such that $1+\frac{\lambda}{\lambda_i}\frac{\sigma_i\sigma_m\delta_i}{\sigma_i+\sigma_m\delta_i}>\frac{\mu}{2}$ for all $i\in \I$. 
        \item $\I = \emptyset$, $\sigma_m\geq 0$, and $\lambda \in (0,+\infty)$. 
    \end{enumerate}
   If $\{(\x^k, \z^k,\y^k)\}$ is a sequence generated by \eqref{eq:dr_stepbystep} from an arbitrary initial point $\x^0\in \H^{m-1}$, then 
   \begin{enumerate}[(i)]
       \item  There exists $\bar{\x}\in \Fix (\dr{\F}{\G})$ such that $\x^k\toweak \bar{\x}$, and $\bar{\z} \coloneqq \JLambda{\F}(\bar{\x}) \in \bfDelta_{m-1} \left( \zer \left( \sum_{i=1}^m A_i\right)\right)$; 
       \item $\norm{(\bfId - \dr{\F}{\G})\x^k}  = o(1/\sqrt{k})$ as $k\to \infty$;  and
       \item   $\norm{\y^{k+1} - \y^k}  = o(1/\sqrt{k})$ and $\norm{\z^{k+1} - \z^k}  = o(1/\sqrt{k})$ as $k\to\infty$. %In addition, $\{\z^k\}$ and $\{\y^k\}$ are bounded sequences.
   \end{enumerate}
Moreover, 
    \begin{enumerate}
        \item[(iv)] If either (A) holds or (B) holds with $\sigma_m>0$, then $ \z^k \to \bar{\z}$, $ \y^k \to \bar{\z}$, and $ \zer \left( \sum_{i=1}^m A_i\right) = \{ U(\bar{\x})\}$; and 
        \item[(v)] If (B) holds with $\sigma_m=0$, then $\z^k \toweak \bar{\z}$ and $ \y^k \toweak \bar{\z}$. 
    \end{enumerate}
\end{theorem}
\begin{proof}
    We first check that the single-valuedness assumptions of \cref{prop:T_nonexpansive} are met. 
    If condition (A) holds, note that for any $i\in \I$,
    \begin{equation*}
       \textstyle  1+\frac{\lambda}{\lambda_i} \sigma_i = \left( 1+\frac{\lambda}{\lambda_i}\frac{\sigma_i\sigma_m\delta_i}{\sigma_i+\sigma_m\delta_i} \right) + \frac{\lambda\sigma_i}{\lambda_i}\left( 1- \frac{\sigma_m\delta_i}{\sigma_i+\sigma_m\delta_i}\right)  >   \frac{\mu}{2}+ \frac{\lambda\sigma_i^2}{\lambda_i(\sigma_i+\sigma_m\delta_i)}>0.
        \label{eq:F_singlevalued_check}
    \end{equation*}
On the other hand, under condition (B), it is clear that $1+\frac{\lambda}{\lambda_i} \sigma_i>0$ for all $i=1,\dots,m-1$. By \cref{rem:cocoercive_and_singlevalued}, we see that $\JLambda{\F}$ is single-valued with domain $\H^{m-1}$. To prove that $\JLambda{\G}$ is likewise single-valued with domain $\H^{m-1}$, it is enough to show by \cref{rem:cocoercive_and_singlevalued} that $1+\lambda\sigma_m >0$. This is clearly true under condition (B). If (A) holds, note that since $\lambda_i + \frac{\lambda\sigma_i\sigma_m\delta_i}{\sigma_i+\sigma_m\delta_i} > \frac{\lambda_i \mu}{2}$ for each $i\in \I$, then $1+ \lambda \sigma_m \sum_{i\in \I} \frac{\sigma_i\delta_i}{\sigma_i+\sigma_m\delta_i} >\frac{\mu}{2}$ by taking the sum for $i=1$ to $i=m-1$. Thus,
    \begin{align*}
        1+\lambda\sigma_m & \textstyle  = \left( 1 + \lambda \sigma_m \sum_{i\in \I} \frac{\sigma_i\delta_i}{\sigma_i+\sigma_m\delta_i} \right) +\lambda\sigma_m \left( 1- \sum_{i\in \I} \frac{\sigma_i\delta_i}{\sigma_i+\sigma_m\delta_i}\right) \\
        & \textstyle  > \frac{\mu}{2 }+ \lambda \sigma_m \left( \sum_{i\in \I} \left( \delta_i - \frac{\sigma_i\delta_i}{\sigma_i+\sigma_m\delta_i}\right) \right)  = \frac{\mu}{2} + \lambda\sigma_m^2 \sum_{i\in \I} \frac{\delta_i^2}{\sigma_i+\sigma_m\delta_i}>0.
    \end{align*}
    Hence, we may now use \cref{prop:T_nonexpansive}. Set $\x = \x^k$ and let $\y \in \Fix (\dr{\F}{\G})$. Noting that $\x^{k+1} = \dr{\F}{\G}(\x^k)$, $\y = \dr{\F}{\G}(\y)$ and $\R(\y) = \bf{0}$, we obtain from \eqref{eq:T_nonexpansive} that 
     \begin{align}
    \normlambda{\x^{k+1} - \y }^2  & \textstyle 
 \leq  \normlambda{\x^k - \y}^2     -\frac{2}{\mu}\sum_{i=1}^{m-1} \lambda_i \kappa_i \norm{R_i(\x^k)}^2 \notag \\ 
    & \textstyle  - 2\mu\lambda \sum_{i\in \I} \theta_i \norm{\sigma_i\left(  J_{\frac{\lambda}{\lambda_i}A_i} (x_i^k) -  J_{\frac{\lambda}{\lambda_i}A_i} (y_i)\right) + \sigma_m \delta_i (U(\x^k)-U(\y)) }^2 \notag \\
    & - 2\alpha \mu \lambda \sigma_m \norm{U(\x^k)-U(\y)}^2.\label{eq:T_nonexpansive_case1}
    \end{align}
For $\kappa_i$, $\theta_i$ and $\alpha$ defined in \eqref{eq:alpha,kappa,theta},  we have $\kappa_i,\theta_i>0$ and $\alpha =0$ under condition (A), while $\kappa_i>0$, $\sigma_m\geq 0$ and $\alpha=1$ under condition (B). Then, we conclude that $\{\x^k\}$ is \alert{$\bfLambda$-}Fej\'{e}r monotone with respect to $\Fix (\dr{\F}{\G})$ and is bounded. By telescoping \eqref{eq:T_nonexpansive_case1}, \alert{$\forall \y\in \Fix(\dr{\F}{\G})$,}
\begin{align}
    & \textstyle  \frac{2}{\mu}\sum_{i=1}^{m-1}\sum_{k=0}^{\infty} \lambda_i \kappa_i \norm{R_i(\x^k)}^2  \notag \\
    &\textstyle  \quad + 2\mu\lambda \sum_{i\in \I} \sum_{k=0}^{\infty} \theta_i \norm{\sigma_i\left(  J_{\frac{\lambda}{\lambda_i}A_i} (x_i^k) -  J_{\frac{\lambda}{\lambda_i}A_i} (y_i)\right) + \sigma_m \delta_i (U(\x^k)-U(\y)) }^2 \notag \\ 
    & \textstyle  \quad + 2\alpha \mu\lambda \sigma_m  \sum_{k=0}^{\infty} \norm{U(\x^k)-U(\y) }^2 \leq \normlambda{\x^0 - \y}^2 <\infty.\label{eq:finite_sum} 
\end{align}  
Since $\lambda_i,\kappa_i>0$ for all $i=1,\dots,m-1$, then $R_i(\x^k) \to 0$ for all $i=1,\dots,m-1$, and so $( \bfId - \dr{\F}{\G} )\x^k = \R (\x^k)  \to \bf{0}$ as $k\to \infty$. Following the arguments in \cite[Theorem 4.2]{DaoPhan2019}, we see that $\{ \x^k\}$ converges weakly to a point $\bar{\x}\in \Fix (\dr{\F}{\G})$, and by \cref{lemma:JFix=zeros}, $\bar{\z}\coloneqq \JLambda{\F}(\bar{\x}) \in \bfDelta_{m-1} \left( \zer \left( \sum_{i=1}^m A_i\right)\right)$.

The rate $\norm{(\bfId - \dr{\F}{\G})\x^k}  = o(1/\sqrt{k})$ can be derived as follows. 
\alert{By the nonexpansiveness of $\dr{\F}{\G}$ with respect to $\norm{\cdot}_{\bfLambda}$
(by \eqref{eq:T_nonexpansive}), we have}
$
    \alert{
    \norm{\x^{k+1}-\x^k}_{\bfLambda}
    % =
    % \norm{\dr{\F}{\G}(\x^k)-\dr{\F}{\G}(\x^{k-1})}_{\bfLambda}
    \leq
    \norm{\x^k-\x^{k-1}}_{\bfLambda}.}
$
\alert{Moreover, by \eqref{eq:finite_sum},} $\alert{
    \sum_{k=0}^{\infty}\norm{(\bfId-\dr{\F}{\G})\x^k}_{\bfLambda}^2
    =
    \sum_{k=0}^{\infty}\norm{\x^{k+1}-\x^k}_{\bfLambda}^2
    <+\infty.}$
\alert{Hence,}
$\alert{\norm{\x^{k+1}-\x^k}_{\bfLambda}=o(1/\sqrt{k})}$.
\alert{Since $\norm{\cdot}_{\bfLambda}$ is equivalent to the original norm, we obtain}
$\norm{(\bfId - \dr{\F}{\G})\x^k}  = o(1/\sqrt{k})$.
This proves part (ii).

From (ii), we use \eqref{eq:JF_cocoercive_simplified} and \eqref{eq:zstep} to conclude that $\norm{\z^{k+1} - \z^k}  = o(1/\sqrt{k})$. These together with \eqref{eq:JG_cocoercive_simplified} and \eqref{eq:ystep} yield $\norm{\y^{k+1} - \y^k}  = o(1/\sqrt{k})$. To complete the proof of part (iii), we have $\norm{\z^k-\bar{\z}}\leq \frac{\lambda_{\max}}{\min _{i=1,\dots,m-1}(\lambda_i + \lambda\sigma_i)}\norm{\x^k-\bar{\x}}$ by \eqref{eq:JF_cocoercive_simplified}. Since $\{\x^k\}$ is bounded, then $\{\z^k\}$ is likewise bounded. Furthermore, since $\y^k-\z^k\to \mathbf{0}$ by using part (ii) and noting \eqref{eq:xstep}, we also obtain the boundedness of $\{\y^k\}$.

To prove part (iv), we show first that $U(\x^k) \to U(\y)$ for any $\y\in \Fix (\dr{\F}{\G})$. If condition (A) holds, then since $\theta_i >0$, we get from \eqref{eq:finite_sum} that for any $i\in \I$ and $\y \in \Fix(\dr{\F}{\G})$,
\begin{equation}
\textstyle  \sigma_i\left(  J_{\frac{\lambda}{\lambda_i}A_i} (x_i^k) -  J_{\frac{\lambda}{\lambda_i}A_i} (y_i)\right) + \sigma_m \delta_i (U(\x^k)-U(\y))  \to 0.
 \label{eq:sum_goes_to_zero}
\end{equation}
On the other hand,
\begin{align*}
\textstyle  \left(  J_{\frac{\lambda}{\lambda_i}A_i} (x_i^k) -  J_{\frac{\lambda}{\lambda_i}A_i} (y_i)\right) -  (U(\x^k)-U(\y)) \overset{\eqref{eq:Id-T}}{=} \frac{1}{\mu}R_i(\x^k) - \frac{1}{\mu}R_i(\y) = \frac{1}{\mu}R_i(\x^k),
\end{align*}
where the rightmost term approaches zero. Combining this with \eqref{eq:sum_goes_to_zero}, we see that $(\sigma_m \delta_i + \sigma_i)(U(\x^k)-U(\y)) \to 0$. Since $\sigma_m \delta_i + \sigma_i>0$, it follows that $U(\x^k)-U(\y)\to 0$ for any $\y\in \Fix (\dr{\F}{\G})$. On the other hand,  if condition (B) holds with $\sigma_m>0$, it is immediate from \eqref{eq:finite_sum} that $U(\x^k) \to U(\y)$ for any $\y\in \Fix (\dr{\F}{\G})$. With this, we use \eqref{eq:Id-T}, the fact that $\R (\x^k)\to \bf{0}$ and \cref{prop:F_resolvent} to conclude that $\z^k = \JLambda{\F}(\x^k) \to \bfDelta_{m-1}(U(\y))$. Meanwhile, since $\y \in \Fix (\dr{\F}{\G})$, we have from  \eqref{eq:Id-T=mu(J-J)} and \cref{prop:T_nonexpansive}(i) that $\JLambda{\F}(\y) =  \bfDelta_{m-1}(U(\y)) $. Putting the pieces together, we have shown that for any $\y\in \Fix(\dr{\F}{\G})$,  $\z^k \to \JLambda{\F}(\y) = \bfDelta_{m-1}(U(\y))$. This shows that $ \z^k\to \JLambda{\F}(\bar{\x}) = \bar{\z}$, and in addition, $\JLambda{\F}(\y) = \bar{\z} = \bfDelta_{m-1}(U(\bar{\x}))$ for all $\y\in \Fix(\dr{\F}{\G})$. Therefore,
 $\bfDelta_{m-1} \left( \zer \left( \sum_{i=1}^m A_i\right)\right) = \{ \JLambda{\F}(\bar{\x}) \}$ by \cref{lemma:JFix=zeros}, and consequently, $\bfDelta_{m-1} \left( \zer \left( \sum_{i=1}^m A_i\right)\right) = \{ \bfDelta_{m-1}(U(\bar{\x}))\}$.  On the other hand, we have from \eqref{eq:xstep} and part (ii) that $\y^k - \z^k \to \mathbf{0}$, which together with $\z^k\to \bar{\z}$ implies that $\y^k\to \bar{\z}$

\alert{Finally, part (v) follows from the standard Douglas--Rachford convergence
theory in the weighted product space induced by $\bfLambda$. Indeed, under
Condition (B), the operators $\lambda\bfLambda^{-1}\F$ and
$\lambda\bfLambda^{-1}\G$ are maximally monotone on $\H^{m-1}$ endowed with
the inner product $\inner{\x}{\y}_{\bfLambda}$.
Therefore, \eqref{eq:DR_scaled} is precisely the Douglas--Rachford
method applied to these two operators in this weighted Hilbert space. Hence the
claimed weak convergence follows from the classical Douglas--Rachford convergence
theorem; see, e.g., \cite[Theorem~26.11]{Bauschke2017}. For the related
preconditioned proximal-point interpretation of Douglas--Rachford splitting, see
also \cite[Section~3]{BrediesChencheneLorenzNaldi2022}.}
  \end{proof}

Condition (A) of \cref{thm:general_convergence} deserves more attention, as the outcome of the theorem depends on the existence of weights $(\delta_i)_{i\in \I}$ satisfying the indicated properties, and the magnitude of the step size parameter $\lambda$ depends on the chosen $(\delta_i)_{i\in \I}$. A sufficient condition for its existence is provided in the following proposition. 

\begin{proposition}
    \label{lemma:constraint_nonempty}
    Let $\sigma_1,\dots, \sigma_m\in \Re$ with $\sigma_m\neq 0$, and suppose that $\I \neq \emptyset$. For each $i\in \I$, let 
    $X_i \coloneqq \{ \delta_i \in \Re : \sigma_i + \sigma_m \delta_i \geq 0 \}$ and $X \coloneqq \prod_{i\in \I}X_i$. Let $S = \{ \delta = (\delta_i)_{i\in \I} : \sum_{i\in \I} \delta_i = 1\}$. Then the following hold:
    \begin{enumerate}[(i)]
        \item $X\cap S$ is compact;
        \item $N_S(\delta) = D_{|\I|} =\{ (c,\dots, c) \in \Re^{|\I|}: c\in \Re\} $ for any $\delta\in S$.
    \end{enumerate}
    Moreover, if $\sum_{i=1}^m \sigma_i >0$, then the following hold:
    \begin{enumerate}[(i),resume]
        \item $\interior (X) \cap S \neq \emptyset$, where $\interior(X)$ denotes the interior of $X$; in particular, $X\cap S \neq \emptyset$; and
        \item $N_{X\cap S}(\delta) = N_X(\delta) + N_S(\delta)$ for any $\delta \in X\cap S$;
     
    \end{enumerate}
\end{proposition}
\begin{proof}
It is clear that $X\cap S$ is closed. Suppose that there exists a sequence $\{\delta^k=(\delta_i^k)_{i\in \I}\} \subset X\cap S$ such that $\norm{\delta^k}\to \infty$ as $k\to \infty$. Without loss of generality, assume that $\alert{\bar{\delta}^k}\coloneqq \frac{\delta^k}{\norm{\delta^k}
}\to \bar{\delta}^*$, where $\norm{\bar{\delta}^*}= 1$. Since $\delta^k \in S$, it follows that 
\[\textstyle  \sum_{i\in\I}\bar{\delta}_i^k = \sum_{i\in \I} \frac{\delta_i^k}{\norm{\delta^k}} = \frac{1}{\norm{\delta^k}} \to 0.\]
Thus, $\sum_{i\in \I}\bar{\delta}_i^* = 0$. On the other hand, since $\delta^k \in X$, it follows that $\frac{\sigma_i}{\norm{\delta^k}} + \sigma_m\frac{\delta_i^k}{\norm{\delta^k}} \geq 0$, and therefore $\sigma_m \bar{\delta}_i^* \geq 0$ for all $i\in \I$. Hence, either $\bar{\delta}_i^* \leq 0$  $\forall i\in \I$, or $\bar{\delta}_i^* \geq 0$ $\forall i\in \I$. Since $\sum_{i\in \I}\bar{\delta}_i^* = 0$, it follows that $\bar{\delta}_i^* =0$ for all $i\in \I$, and therefore $\norm{\bar{\delta}^*}=0$, which is a contradiction. Hence, $X\cap S$ must be bounded. This completes the proof of part (i). For part (ii), note that
 \begin{equation*}
         \textstyle  N_{S}(\delta) = \left(S - S \right)^{\perp} = \left\lbrace w=(w_1,\dots,w_{|\I|}) \in \Re^{|\I|} : \sum_{i\in \I} w_i = 0 \right\rbrace ^{\perp} = (D_{|\I|}^{\perp})^{\perp} = D_{|\I|},
        \label{eq:normalcone}
    \end{equation*}
where the first equality holds by \cite[Example 6.43]{Bauschke2017}. To prove (iii), take $\delta=(\delta_i)_{i\in \I}$ with 
$ \delta_i = \frac{1}{|\I|} + \frac{\sum_{\substack{j \in \I \\ j \neq i}} \sigma_j - (|\I|-1)\sigma_i}{\sigma_m |\I|}.$
It can be verified that $\sum_{i\in \I}\delta_i = 1$, and using the hypothesis that $\sum_{i=1}^m \sigma_i >0$, it can be shown that $\sigma_i + \sigma_m \delta_i>0$. This proves part (iii). Part (iv) is a direct consequence of part (iii) and \cite[Section 1]{Burachik2005}.  
  \end{proof}

From \cref{lemma:constraint_nonempty}(iii), we see that provided that $\sum_{i=1}^{m}\sigma_i>0$, any $\delta = (\delta_i)_{i\in\I}$ from $\interior (X) \cap S \neq \emptyset$ can be chosen so as to satisfy the requirement of condition (A) of \cref{thm:general_convergence}. The last issue we address is how to choose the parameters $\delta$ from $\interior (X) \cap S$, in such a way that we maximize the allowable step size $\lambda$ as dictated by the last requirement stipulated in condition (A).

\begin{proposition}
\label{prop:optimaldelta}
    Let  $\lambda_1,\dots, \lambda_{m-1}\in (0,+\infty)$, $\mu\in (0,2)$, and $\sigma_1,\dots, \sigma_m\in \Re$. Let $\I $, $\I^-$ and $\I^+$ be given by \eqref{eq:I_index}, and let $X_i$ ($i\in \I$), $X$ and $S$ be as in \cref{lemma:constraint_nonempty}. Consider the optimization problem
        \begin{equation}
            \begin{array}{rrl}
               \ds\bar{\lambda}^* \coloneqq  & \ds  \max_{\delta \in \Re^{|\I|}, \bar{\lambda}\geq 0} ~& \bar{\lambda} \\
               & {\rm s.t.}~ & 1+ \frac{\bar{\lambda}}{\lambda_i} \frac{\sigma_i\sigma_m \delta_i}{\sigma_i + \sigma_m \delta_i} - \frac{\mu}{2} \geq  0 \quad  i\in \I, \\
               & & \delta = (\delta_i)_{i\in \I} \in X \cap S.
            \end{array}
            \label{eq:maximizing_stepsize}
        \end{equation}
If $\sum_{i=1}^m \sigma_i >0$ and $\I \neq \emptyset$, then the following holds:
     \begin{enumerate}[(i)] 
        \item  If either $\I^-\neq \emptyset$ and $\sigma_m\neq 0$, or $\I^- =\emptyset$ and $\sigma_m<0$, then \eqref{eq:maximizing_stepsize} has a solution. Moreover, if $(\delta^*,\bar{\lambda}^*)\in S\times \Re_+$ solves \eqref{eq:maximizing_stepsize}, then $\delta^* = (\delta_i^*)_{i\in \I}$ satisfies 
        \begin{subequations} \label{eq:example}
        \begin{align}
        \sigma_i + \sigma_m\delta_i^*>0 \quad \forall i\in \I, \label{eq:interiorcondition}\\
        % \sigma_m\delta_i^* <0 \quad \forall i\in \I^+ ,\label{eq:oppositesign} \\
         \textstyle -\frac{\lambda_i(\sigma_i+\sigma_m \delta_i^*)}{\sigma_i\sigma_m \delta_i^*} = -\frac{\lambda_j(\sigma_j+\sigma_m \delta_j^*)}{\sigma_j\sigma_m \delta_j^*} >0\quad \forall i,j\in \I, \label{eq:intersection} 
        \end{align}
        \end{subequations}
        and  $\bar{\lambda}^* = -\left( 1-\frac{\mu}{2}\right) \left( \frac{\lambda_i(\sigma_i+\sigma_m\delta_i^*)}{\sigma_i\sigma_m\delta_i^*}\right)$.
        \item If $\I^-=\emptyset$ and $\sigma_m\geq 0$, then \eqref{eq:maximizing_stepsize} is an unbounded optimization problem, i.e., $\bar{\lambda}^* = +\infty$. 
     \end{enumerate}  
\end{proposition}
\begin{proof}
 For each $i\in \I$, let 
 $Y_i = X_i$ if $i\in \I^-$, and $Y_i = \{\delta_i\in X_i: \delta_i\sigma_m<0\}$ if $i\in \I^+$
    and define $f_i:X_i\to [0,+\infty]$ by
    \begin{equation*}
       \textstyle  f_i(\delta_i) =  \begin{cases}
        \frac{\lambda_i(\sigma_i+\sigma_m \delta_i)}{-\sigma_i\sigma_m \delta_i} & \text{if}~\delta_i \in  Y_i ,\\
        +\infty & \text{if}~\delta_i \in X_i \setminus Y_i.\\
    \end{cases}
    \label{eq:f_i}
    \end{equation*}
Note that given $\delta_i \in X_i$, $f_i(\delta_i)$ represents the largest nonnegative (extended-real) number such that if $0< \lambda<\left( 1-\frac{\mu}{2}\right)f_i(\delta_i)$, then the inequality $1+\frac{\lambda}{\lambda_i} \frac{\sigma_i\sigma_m\delta_i}{\sigma_i+\sigma_m\delta_i} > \frac{\mu}{2}$ holds true.  Hence, the problem \eqref{eq:maximizing_stepsize} can be reformulated as 
    \begin{equation}
        \begin{array}{rl}
             \ds \max_{\delta\in\Re^{|\I|}} ~& \ds f(\delta) \coloneqq \min_{i\in \I} f_i(\delta_i)   \qquad 
             {\rm s.t.}~\qquad  \delta = (\delta_i)_{i\in \I}\in X\cap S 
        \end{array}.
        \label{eq:maximize_stepsize_2}
    \end{equation}
Moreover, if $\delta^*$ solves \eqref{eq:maximize_stepsize_2}, then $(\delta^*, \bar{\lambda}^*)$ solves \eqref{eq:maximizing_stepsize} where $\bar{\lambda}^* =\left( 1-\frac{\mu}{2}\right)f(\delta^*)$. 

We now show that $f$ is continuous on the set $Z$, 
where $Z=X$ if $\I^- \neq  \emptyset$, and $Z = X\setminus \Re^{|\I|}_- $ if $\I^- = \emptyset$ and $\sigma_m<0$, and
 $\Re^{|\I|}_- = \{ \delta\in \Re^{|\I|} : \delta_i\leq 0 ~\forall i \in \I\}$.  Let $\delta = (\delta_i)_{i\in \I}\in Z$. First, suppose that $\delta_i \in Y_i$ for all $i\in \I$. Note that each $f_i$ is continuous on $\mathcal{N}_i\cap Y_i$ for some neighborhood $\mathcal{N}_i$ of $\delta_i$. Thus, $f\equiv \min_{i\in \I}f_i$ on the set  $\mathcal{N}\alert{\cap} Y$, where $\mathcal{N} \coloneqq \prod_{i\in\I}\mathcal{N}_i$ and $Y \coloneqq \prod_{i\in \I} Y_i$. Since each $f_i$ is continuous on $\mathcal{N}_i\cap  Y_i$, the continuity of $f$ on $\mathcal{N}\alert{\cap} Y$ follows. Hence, $f$ is continuous at $\delta$. Suppose, on the other hand, that $\J (\delta)\coloneqq \{i\in \I : \delta_i \in X_i\setminus Y_i \}$ is nonempty. Observe that $\sigma_i>0$ for all $i\in \J (\delta)$. Since $\alert{f_i}\equiv +\infty $ on $X_i \setminus Y_i$ and $\lim_{\underset{\delta'_i  \in Y_i}{\delta'_i  \to 0}} f_i(\delta'_i )= +\infty$ for all $i\in \J(\delta)$,  there exists a neighborhood $\mathcal{N}$ of $\delta$ such that  $f\equiv  \min_{i\in \I \setminus \J (\delta)} f_i $ on $\mathcal{N}\cap X$. We note that the index set $\I \setminus \J (\delta)$ is nonempty under our hypotheses. Indeed, this is clear when $\I^-\neq \emptyset$ since $\I^- \subseteq \I \setminus \J (\delta)$. On the other hand, if $\I^- = \emptyset$ and $\sigma_m<0$, note that $Y_i = (0,-\sigma_i/\sigma_m]$ for all $i\in \I^+ = \I$. Since $\delta \in Z = X\setminus \Re^{|\I|}_-$, it follows that there exists $j\in \I$ such that $\delta_j >0$. Necessarily, $j\in \I\setminus \J(\delta)$, and so $\I \setminus \J(\delta)$ is nonempty, as claimed. Hence, \alert{on} $\mathcal{N}\cap X$, $f$ is the pointwise minimum of the continuous functions $f_i$'s with $i\in \I \setminus \J(\delta) \neq \emptyset $, and therefore $f$ is continuous on $\mathcal{N}\cap X$. This proves that $f$ is continuous on $Z$. As a side note, which will be useful later, the above arguments  show that for any $\delta\in Z$, there exists a neighborhood $\mathcal{N}$ of $\delta$ such that 
\begin{equation}
\textstyle f (\delta') = \min_{i\in \I \setminus \J(\delta)}f_i (\delta'_i) \quad \forall \delta' \in \mathcal{N}\cap Z.
\label{eq:f_piecewisesmooth}
\end{equation}
Since $f$ is continuous on $Z$, then $f$ is also continuous on $Z\cap S=X \cap S$, where the last equality holds since $S \cap \Re^{|\I|}_-=\emptyset$. Since $X\cap S$ is a nonempty compact set by \cref{lemma:constraint_nonempty}(i) and (iii), it follows that \eqref{eq:maximize_stepsize_2} has a solution, and so does \eqref{eq:maximizing_stepsize}. This proves the first claim of part (i). 
% Meanwhile, observe that $f$ is a nonnegative function, and $f(\delta) = 0$ if and only if $\delta_i = -\frac{\sigma_i}{\sigma_m}$ for some $i\in \I$. Moreover,  for any $i\in \I^-\neq \emptyset$, we have $f(\delta) \leq f_i(\delta_i) < -\sigma_i^{-1}$ for all $\delta \in X$. It follows that $f$ is bounded on $X$, and therefore bounded on $X\cap S$. Hence, the optimization problem \eqref{eq:maximize_stepsize_2} has a solution if $X\cap S \neq \emptyset$. Since the latter holds by \cref{lemma:constraint_nonempty}(i), it follows now that \eqref{eq:maximize_stepsize_2} has a solution, and so does \eqref{eq:maximizing_stepsize}. 

Now, let $\delta^*\in X\cap S$ be an optimal solution of \eqref{eq:maximize_stepsize_2}. Note that $f$ is a nonnegative function, and $f(\delta) = 0$ if and only if $\delta_i = -\frac{\sigma_i}{\sigma_m}$ for some $i\in \I$. Thus, $f(\delta^*)>0$ and  $\delta^*\in \interior(X)\cap S = \interior (Z) \cap S$. This implies that \eqref{eq:interiorcondition} holds and $N_X(\delta^*) = \{ 0\}$. In addition, by the optimality of $\delta^*$, we have from \cite[Theorems 10.1 and 10.10]{RW98} that 
    \begin{equation}
    0\in \partial (-f(\delta^*)) + N_{X\cap S}(\delta^*),  \label{eq:optimalitycondition}
    \end{equation}
where $\partial f$ denotes the Clarke subdifferential of $f$. Using \cite[Exercise 8.31]{RW98} and the representation \eqref{eq:f_piecewisesmooth}, we have 
    \begin{equation}
        \textstyle \partial (-f(\delta^*)) = \co \left\lbrace -\frac{\lambda_i}{\alert{\sigma_m}(\delta_i^*)^2}e_i : i\in \mathcal{A}(\delta^*)\right\rbrace,
        \label{eq:partial_f}
    \end{equation}
where $e_i$ is the standard unit vector in $\Re^{|\I|}$, $\mathcal{A}(\delta^*) \coloneqq \{ i: i\in \I\setminus \J (\delta^*) \text{ ~s.t.~}f(\delta^*)=f_i(\delta_i^*)  \}$, and ``$\co$'' denotes the convex hull. 
Using \cref{lemma:constraint_nonempty}(ii) and (iv) together with  \eqref{eq:optimalitycondition} and \eqref{eq:partial_f}, we conclude that there exists $\{ \alpha_i \in [0,1] : i\in \mathcal{A}(\delta^*)\}$ with $\sum_{i\in \mathcal{A}(\delta^*)} \alpha_i = 1$ and $\sum_{i\in \mathcal{A}(\delta^*)} \frac{\alpha_i\lambda_i}{\alert{\sigma_m}(\delta_i^*)^2}e_i \in D_{|\I|}. $ \alert{Meanwhile, note that $\delta_i^*\neq 0$ for every $i\in \mathcal{A}(\delta^*)$. Indeed, if $i\in \I^-$, then $\sigma_i<0$, and, since $\delta^*\in \operatorname{int}(X)$, $\delta_i^*=0$ would contradict $\sigma_i+\sigma_m\delta_i^*>0$. If $i\in \I^+$, then since $i\in \mathcal{A}(\delta^*)\subseteq \I\setminus \J(\delta^*)$ and $\delta_i^*\in X_i$, the definition of $\J(\delta^*)$ gives $\delta_i^*\in Y_i$, and hence $\delta_i^*\sigma_m<0$.} Since $\lambda_i >0$, \alert{$\sigma_m\neq 0$}, and $\frac{1}{(\delta_i^*)^2}\neq 0$ for any $\delta_i^*\in \Re\alert{\setminus \{0\}}$, we must necessarily have $\J(\delta^*)= \emptyset$ and $\mathcal{A}(\delta^*) = \I$. Thus, $f_i(\delta_i^*) = f_j(\delta_j^*)$ for all $i,j \in \I$, i.e., \eqref{eq:intersection} holds. This completes the proof of part (i).

Finally, we prove part (ii). Since $\I^- =\emptyset$,  $\sigma_i>0$ for all $i\in \I$. Together with $\sigma_m\geq 0$, we see that $\Re^{|\I|}_+\subseteq X$. For all $\delta\in \Re^{|\I|}_+ \cap S$, 
% (i.e., the unit simplex), 
the inequality constraints in \eqref{eq:maximizing_stepsize} are trivially satisfied since $\mu \in (0,2)$. Thus, the claim immediately follows. 
  \end{proof}

We now restate \cref{thm:general_convergence} based on \cref{lemma:constraint_nonempty} and \cref{prop:optimaldelta}. Note that the conditions in \cref{prop:optimaldelta}(ii) correspond to the maximal monotone case where at least one among $\sigma_1,\dots,\sigma_{m-1}$ is strictly positive, a case which was not included yet in condition (B) of \cref{thm:general_convergence}. We now include this in condition (B) of the following theorem to distinguish monotone cases from nonmonotone ones.

\begin{theorem}\label{thm:convergence_optimal}
    Let $A_i:\H\toset \H$ be maximal $\sigma_i$-monotone for each $i=1,\dots,m$, and assume that $\zer\left( A_1+\cdots + A_m\right)\neq \emptyset$. Let $\mu\in (0,2)$, $\lambda_1,\dots, \lambda_{m-1}\in (0,+\infty)$ with $\sum_{i=1}^{m-1}\lambda_i = 1$, and let $\bfLambda$ be given by \eqref{eq:Lambda}. Suppose that either one of the following holds:
    \begin{enumerate}
        \item[(A)] (Nonmonotone case). There exists $j\in \{1,\dots,m\}$ such that $\sigma_j<0$, $\sigma_m\neq 0$, $\sum_{i=1}^m \sigma_i >0$, and  $\bar{\lambda}^*$ is defined in  \eqref{eq:maximizing_stepsize};
        % \item[(A2)]  $\I^- = \emptyset$, $\sigma_m <0$, and  $\bar{\lambda}^*$ is defined in \cref{prop:optimaldelta}(i);
        \item[(B)] (Monotone case). $\sigma_i\geq 0$ and $\bar{\lambda}^*=+\infty$.
    \end{enumerate}
    If $\lambda\in (0,\bar{\lambda}^*)$ and $\{ (\x^k,\z^k,\y^k)\}$ is a sequence generated by \alert{\eqref{eq:dr_stepbystep}} from an arbitrary initial point $\x^0\in \H^{m-1}$, then 
   \begin{enumerate}[(i)]
       \item  There exists $\bar{\x}\in \Fix (\dr{\F}{\G})$ such that $\x^k\toweak \bar{\x}$ and \\ $\bar{\z} \coloneqq \JLambda{\F}(\bar{\x}) \in \bfDelta_{m-1} \left( \zer \left( \sum_{i=1}^m A_i\right)\right)$; 
       \item $\norm{(\bfId - \dr{\F}{\G})\x^k}  = o(1/\sqrt{k})$ as $k\to \infty$;  and
       \item $\norm{\y^{k+1} - \y^k}  = o(1/\sqrt{k})$ and $\norm{\z^{k+1} - \z^k}  = o(1/\sqrt{k})$ as $k\to\infty$.
   \end{enumerate}
Moreover, 
    \begin{enumerate}
        \item[(iv)] Suppose either (A) holds, or (B) holds together with $\exists\, j\in\{1,\dots,m\}$ such that $\sigma_j>0$. Then $ \z^k \to \bar{\z}$, $ \y^k \to \bar{\z}$, and $ \zer \left( \sum_{i=1}^m A_i\right) = \{ U(\bar{\x})\}$; and 
        \item[(v)] If (B) holds with $\sigma_i=0$ for all $i=1,\dots,m$, then $\z^k \toweak \bar{\z}$ and $ \y^k \toweak \bar{\z}$. 
    \end{enumerate}

\end{theorem}

\begin{remark}
\label{rem:comparewithnaive}
    Suppose that the weights are equal, i.e., $\lambda_1= \cdots = \lambda_{m-1} = \frac{1}{m-1}$, and $\lambda = \frac{\gamma}{m-1}$. Notice that condition (B) of \cref{prop:naive_douglas} is covered by condition (B) of \cref{thm:convergence_optimal}. On the other hand, \cref{thm:convergence_optimal} under condition (A) offers a significantly stronger result than \cref{prop:naive_douglas}(A). First, note that $\sigmabf{\F}+\frac{\sigma_m}{m-1} >0 $ implies that $\sum_{i=1}^{m} \sigma_i >0$, but the latter is a much weaker condition. In particular, the requirement $\sigmabf{\F}+\frac{\sigma_m}{m-1} >0 $ does not cover situations where $\I^-\neq \emptyset$ and $\sigma_m\leq 0$, or $\I^-=\emptyset$, $\I^+ \neq \I$ and $\sigma_m<0$ (c.f. condition (A) of \cref{thm:convergence_optimal} which summarizes the setting in \cref{prop:optimaldelta}(i)). For the cases that are covered, the range of step size for $\lambda$ prescribed by \cref{thm:convergence_optimal} is at least as large as, and potentially strictly larger than, provided in \cref{prop:naive_douglas}. In particular, as in condition (A) of \cref{prop:naive_douglas}, suppose that 
    \begin{equation}
        \textstyle 1+\gamma \frac{\alert{\sigmabf{\F}}\sigma_m \left( \frac{1}{m-1}\right) }{\alert{\sigmabf{\F}} + \sigma_m \left( \frac{1}{m-1}\right) }>\frac{\mu}{2}.
        \label{eq:stepsize_naive}
    \end{equation}
    \\
   \noindent \textit{Case 1.} Suppose $\I^-\neq \emptyset$ and $\sigma_m>0$. 
Set $\delta_i = 1/|\I|$ for all $i\in\I$, so that
$\delta=(\delta_i)_{i\in\I}\in S$.
 Since  $\alert{\sigmabf{\F}}+ \frac{\sigma_m}{m-1} >0$,  $\delta\in X$. With this choice of $\delta$ together with \eqref{eq:stepsize_naive} and the minimality of $\alert{\sigmabf{\F}}$, it is not difficult to show that the inequality constraints in \eqref{eq:maximizing_stepsize} are satisfied. In other words, $(\delta,\frac{\gamma}{m-1})$ is feasible to \eqref{eq:maximizing_stepsize}. Hence, the claim follows. \\
    \noindent \textit{Case 2.} Suppose that $\I^+ = \I$ and $\sigma_m<0$. To prove the claim, we only need to choose $\delta_i = \frac{1}{m-1}$ for all $i\in \I$, and argue as in the previous case. 
\end{remark}

We close this section with the convergence result for the DR algorithm 
\begin{equation}
    \x^{k+1} \in \dr{\G}{\F} (\x^k), 
    \label{eq:DR_scaled2}
\end{equation}
where 
$\dr{\G}{\F}:\H^{m-1} \toset \H^{m-1}$ is given by 
    \begin{equation}
        \dr{\G}{\F} (\x) \coloneqq \{ \x+\mu (\y - \z ): \z\in \JLambda{\G}(\x), ~\y \in \JLambda{\F} (2\z-\x)\}.
        \label{eq:dr_map2}
    \end{equation}
\alert{The convergence proof uses the same techniques as before, but it is not straightforward so we include its proof in \cref{app:dr_switched}. We denote 
$ \z^k  \in \JLambda{\G}(\x^k) , \y^k  \in \JLambda{\F}(2\z^k-\x^k), \text{ and } 
            \x^{k+1}  = \x^k + \mu (\y^k - \z^k). $}

\begin{theorem}
    \label{thm:general_convergence2}
    Suppose that the hypotheses of \cref{thm:convergence_optimal} hold.
   If $\lambda\in (0,\bar{\lambda}^*)$ and $\{(\x^k, \z^k,\y^k)\}$ is a sequence generated by \eqref{eq:DR_scaled2} from an arbitrary initial point $\x^0\in \H^{m-1}$, then 
   \begin{enumerate}[(i)]
       \item  There exists $\bar{\x}\in \Fix (\dr{\G}{\F})$ such that $\x^k\toweak \bar{\x}$, and $\bar{\z} \coloneqq \JLambda{\G}(\bar{\x}) \in \bfDelta_{m-1} \left( \zer \left( \sum_{i=1}^m A_i\right)\right)$; 
       \item $\norm{(\bfId - \dr{\G}{\F})\x^k}  = o(1/\sqrt{k})$ as $k\to \infty$;  and
       \item $\norm{\y^{k+1} - \y^k}  = o(1/\sqrt{k})$ and $\norm{\z^{k+1} - \z^k}  = o(1/\sqrt{k})$ as $k\to\infty$. %In addition, $\{\z^k\}$ and $\{\y^k\}$ are bounded sequences.
        \item Suppose either (A) holds, or (B) holds together with $\exists\, j\in\{1,\dots,m\}$ such that $\sigma_j>0$. Then $ \z^k \to \bar{\z}$, $ \y^k \to \bar{\z}$, and  $ \zer \left( \sum_{i=1}^m A_i\right) = \left\lbrace J_{\lambda A_m}\left( \sum_{i=1}^{m-1} \lambda_i \bar{x}_i\right) \right\rbrace$; and 
        \item If (B) holds with $\sigma_i=0$ for all $i=1,\dots,m$, then $\z^k \toweak \bar{\z}$ and $ \y^k \toweak \bar{\z}$. 
    \end{enumerate}
\end{theorem} 

% \begin{remark}
%     \alert{Campoy's DR algorithm in \cite{Campoy2022} corresponds to \eqref{eq:DR_scaled2}
%     with equal weights, under the assumption that all the operators are maximal
%     monotone. This is precisely the setting described in
%     \cref{thm:convergence_optimal} (B) and \cref{thm:general_convergence2} (B).
%     In this setting, \cref{thm:general_convergence2} recovers the convergence
%     result of \cite[Theorem 5.1]{Campoy2022}. Thus, in the maximal monotone and
%     equal-weight case, our result agrees with Campoy's result.}

%     \alert{The additional contribution of \cref{thm:general_convergence2} for the classical maximal monotone case is that it extends
%     this conclusion to arbitrary positive weights. Moreover, by
%     \cref{thm:general_convergence2}(iv), strong convergence of \(\{\z^k\}\) and
%     \(\{\y^k\}\) holds provided any one of the maximal monotone operators \(A_i\)
%     is maximal \(\sigma_i\)-monotone with \(\sigma_i>0\). Hence, unlike
%     \cite[Theorem 5.1(ii)]{Campoy2022}, the ordering of the operators does not
%     affect the strong convergence conclusion. Finally, compared with
%     \cite[Theorem 5.1]{Campoy2022}, \cref{thm:general_convergence2} additionally
%     provides convergence rates.}
% \end{remark}
\begin{remark}
    \alert{In the purely monotone and equal-weight setting, the algorithms
    \eqref{eq:DR_scaled} and \eqref{eq:DR_scaled2} have already appeared in
    several works. Campoy's Douglas--Rachford algorithm in \cite{Campoy2022}
    corresponds to \eqref{eq:DR_scaled2} with equal weights, under the
    assumption that all the operators are maximal monotone. This is precisely the
    setting described in \cref{thm:convergence_optimal}(B) and
    \cref{thm:general_convergence2}(B). An analogous parallel
    Douglas--Rachford algorithm and more general parallel Davis--Yin schemes (which include forward terms) were also studied from the
    viewpoint of degenerate preconditioned proximal point algorithms in
    \cite{BrediesChencheneLorenzNaldi2022}. Both orders of the parallel
    Douglas--Rachford method are included in the graph-based framework of
    \cite{BrediesChencheneNaldi2024,Naldi2025}. Related graph-based
    splitting frameworks with forward terms are studied in
    \cite{AragonArtachoCampoyLopezPastor2025,AkermanChencheneGiselssonNaldi2025}.}

    \alert{In the maximal monotone and equal-weight case,
    \cref{thm:general_convergence2} agrees with these existing monotone
    convergence results. Moreover, the strong convergence conclusion in
    \cref{thm:general_convergence2}(iv), which is independent of the ordering of
    the operators, is consistent with
    \cite[Proposition~3.11]{BrediesChencheneNaldi2024}. This order-independent
    strong convergence conclusion was not part of
    \cite[Theorem~5.1(ii)]{Campoy2022}. The present results further allow
    arbitrary positive weights and extend the analysis beyond the purely monotone
    setting.}
\end{remark}
\section{DR for structured classes of nonconvex optimization problems}
\label{sec:mainresults_optimization}
We now focus on the  problem 
    \begin{equation}
        \min _{x\in \H} f_1(x) + \cdots + f_m(x),
        \label{eq:optimization}
    \end{equation}
where $f_i:\H \to (-\infty,+\infty]$ is a proper closed function  for all $i=1,\dots,m$. 

\subsection{Nonconvex optimization under weak/strong convexity}\label{sec:dr_optimization_weakconvex}
To apply the Douglas--Rachford algorithm for solving \eqref{eq:optimization}, we consider an associated inclusion problem involving subdifferentials. In this section, we consider the case in which each function $f_i$ is \alert{$\sigma_i$}-convex for some $\sigma_i\in \Re$, $i=1,\dots,m$.
% \subsubsection{Convergence of the DR algorithm}
We need the following lemma.

\begin{lemma}
\label{lemma:optimization_sigmaconvex_case}
    If $f_i:\H \to (-\infty,+\infty]$ is $\sigma_i$-convex  for all $i=1,\dots,m$, then 
    \begin{enumerate}[(i)]
        \item $\partial f_i$ is maximal $\sigma_i$-monotone.
        \item For any $\gamma >0$ such that $1+\gamma \sigma_i>0$, $\prox_{\gamma f_i}$ is equal to $ J_{\gamma \partial f_i}$, is single-valued and has full domain.
        \item $\sum_{i=1}^m f_i$ is $\sum_{i=1}^m \sigma_i$-convex.
        \item If $\sum_{i=1}^m \sigma_i\geq 0$, then $\zer \left( \sum_{i=1}^m \partial f_i \right) \subseteq \zer \left( \partial\left( \sum_{i=1}^m  f_i \right)  \right) = \argmin \left( \sum_{i=1}^m f_i \right) $.  
    \end{enumerate}
\end{lemma}
\begin{proof}
    The proofs follow by invoking \cite[Proposition 1.107(ii)]{Mordukhovich2006} to show that $\hat{\partial}f = \partial f$, and then using the same arguments as in the proofs of \cite[Lemmas 5.2 and 5.3]{DaoPhan2019}.
  \end{proof}
In view of \cref{lemma:optimization_sigmaconvex_case}(iv), we may obtain solutions of \eqref{eq:optimization} by considering the  problem
    \begin{equation}
        \text{Find~}x\in\H ~\text{such that } 0 \in \partial f_1(x) + \cdots + \partial f_m(x),
        \label{eq:optimization_as_inclusion}
    \end{equation}
whenever $\sum_{i=1}^m \sigma_i \geq 0$. In \cref{alg:dr_scaled_opt}, we present the Douglas--Rachford algorithm (\cref{alg:dr_scaled}) applied to \eqref{eq:optimization_as_inclusion}. This algorithm also appeared in \cite[Section 9.1]{Condat2023} but the setting considered in the said work involves only convex functions $f_1,\dots,f_m$.

\begin{algorithm}%[tb]
    Input initial point $(x^0_1,\dots,x^0_{m-1})\in \H^{m-1}$ and parameters $\mu\in (0,2)$ and $\lambda,\lambda_1,\dots,\lambda_{m-1}\in (0,+\infty)$ with $\sum_{i=1}^{m-1}\lambda_i = 1$. \\
    For $k=0,1,2,\dots ,$
    \begin{equation*}
        \left[\begin{array}{rll}
            z_i^{k}&  \in \prox_{\frac{\lambda}{\lambda_i}f_i}(x_i^k), & (i=1,\dots,m-1) \\
			\ds y^{k}&   \in \prox_{\lambda f_m}\left( \sum_{i=1}^{m-1} \lambda_i (2z_i^k-x_i^k)\right) \\
			x_i^{k+1} & = x_i^k + \mu ( y^{k} - z_i^{k} ) & (i=1,\dots,m-1) .
        \end{array}\right.
		\end{equation*}
	\caption{Douglas--Rachford for  sum-of-$m$-functions optimization \eqref{eq:optimization}.}
	\label{alg:dr_scaled_opt}
\end{algorithm}

The convergence of \cref{alg:dr_scaled_opt} when the $f_i$'s are $\sigma_i$-convex is a direct consequence of \cref{thm:convergence_optimal}. This result can be viewed as an extension of \cite[Theorem 9.1]{Condat2023}, which only covers the convex case described in \cref{thm:dr_convergence_optimization}(ii). 

\begin{theorem}
\label{thm:dr_convergence_optimization}
 Let $f_i:\H\to (-\infty,+\infty]$ be $\sigma_i$-convex for each $i=1,\dots,m$, and suppose that $\zer\left( \partial f_1 + \cdots + \partial f_m\right)\neq \emptyset$. Let $\mu\in (0,2)$, $\lambda_1,\dots, \lambda_{m-1}\in (0,+\infty)$ with $\sum_{i=1}^{m-1}\lambda_i = 1$. Suppose that one of the following holds:
    \begin{enumerate}
        \item[(A)] (Nonconvex case). There exists $j\in \{1,\dots,m\}$ such that $\sigma_j<0$, $\sigma_m\neq 0$, $\sum_{i=1}^m \sigma_i >0$, and  $\bar{\lambda}^*$ is defined in \eqref{eq:maximizing_stepsize};
        % \item[(A2)]  $\I^- = \emptyset$, $\sigma_m <0$, and  $\bar{\lambda}^*$ is defined in \cref{prop:optimaldelta}(i);
        \item[(B)] (Convex case). $\sigma_i\geq 0$ and $\bar{\lambda}^*=+\infty$.
    \end{enumerate}
    If $\lambda\in (0,\bar{\lambda}^*)$ and $\{ (x_1^k, \dots, x_{m-1}^k, z_1^k,\dots,z_{m-1}^k, y^k)\}$ is a sequence generated by \cref{alg:dr_scaled_opt} from an arbitrary initial point $(x_1^0,\dots,x_{m-1}^0)\in \H^{m-1}$, then the following hold:
    \begin{enumerate}[(i)]
        \item $\{x_i^k\}$, $\{ y^k\}$ and $\{ z_i^k\}$ are bounded sequences, where $i=1,\dots,m-1$. 
        \item $\norm{x_i^{k+1}-x_i^k}=o(1/\sqrt{k})$, $\norm{y^{k+1}-y^k}=o(1/\sqrt{k})$ and \\ 
        $\norm{z_i^{k+1}-z_i^k}=o(1/\sqrt{k})$ as $k\to \infty$, where $i=1,\dots,m-1$. 
        \item If (A) holds or (B) holds with $\sigma_j>0$ for some $j\in \{1,\dots,m\}$, then \eqref{eq:optimization} has a unique solution $\bar{z}$. Moreover, the sequences $\{z_i^k\}$ and $\{y^k\}$ converge strongly to $\bar{z}$ for any $i=1,\dots, m-1$.
        \item If condition (B) holds with $\sigma_i=0$ for all $i=1,\dots,m$, then there exists $\bar{z}\in \argmin \left( \sum_{i=1}^m f_i \right)$ such that the sequences $\{z_i^k\}$ and $\{y^k\}$ converge weakly to $\bar{z}$ for any $i=1,\dots, m-1$.
    \end{enumerate}
\end{theorem}

\subsection{Nonconvex optimization for finite-dimensional Hilbert spaces under Lipschitz gradient conditions}\label{sec:nonconvexopt_finitedimensional}

The second setting we consider involves an arbitrary proper closed function $f_m$ (\textit{i.e.,} not necessarily $\sigma_m$-convex), but we additionally assume that for each $i=1,\dots,m-1$, $f_i$ is \alert{$L_{i}$}-smooth. The case $m=2$ was previously studied in \cite{li2016douglas,ThemelisPatrinos2018}. In this section, we also assume that $\H$ is finite-dimensional.

\begin{lemma}
\label{lemma:optimization_lsmooth}
    Let $f_i:\H \to (-\infty,+\infty]$ be an $L_i$-smooth function for all $i=1,\dots,m-1$ and $f_m$ is a proper closed function. 
    Then the following hold:
    \begin{enumerate}[(i)]
        \item For any $i=1,\dots,m-1$, $\partial f_i=\nabla f_i$ is maximal $\sigma_i$-monotone for some $\sigma_i \in [-L_i,L_i]$. Thus, $f_i$ is $\sigma_i$-convex for some $\sigma_i \in [-L_i,L_i]$. 
        \item For any $i=1,\dots,m-1$ and $\gamma >0$ such that $1-\gamma L_i>0$, $\prox_{\gamma f_i}$ is equal to $ J_{\gamma \partial f_i}$, is single-valued and has full domain.
        \item  $\argmin \left( \sum_{i=1}^m f_i \right) \subseteq \zer \left( \partial \left(\sum_{i=1}^m f_i \right)\right) = \zer \left(\sum_{i=1}^m \partial  f_i \right)  $.  
    \end{enumerate}
\end{lemma}
\begin{proof}
    Part (i) directly follows from \eqref{eq:descentlemma}, from where we also see that $\partial f_i$ is maximal $(-L_i)$-monotone, and so part (ii) follows by \cref{lemma:optimization_sigmaconvex_case}(ii).
    Part (iii) follows from  \cite[Theorem 10.1 and Exercise 8.8(c)]{RW98}.
  \end{proof}

From \cref{lemma:optimization_lsmooth}(iii), solving \eqref{eq:optimization_as_inclusion} provides candidate solutions to \eqref{eq:optimization}. Since $f_m$ may not be $\sigma_m$-convex, $\prox_{\gamma f_m}$ may differ from $J_{\gamma \partial f_m}$. However, by \eqref{eq:prox_subset_J}, \cref{alg:dr_scaled_opt} is a specific instance of \cref{alg:dr_scaled} for \eqref{eq:optimization_as_inclusion}. From \cref{lemma:optimization_lsmooth}(ii), $\prox_{\gamma f_i} = J_{\gamma \partial f_i}$ for $i=1,\dots,m-1$ if $\gamma<1/L_i$. While $\prox_{\gamma f_m}$ may not equal $J_{\gamma \partial f_m}$, it has full domain and compact values under a coercivity assumption. 

\begin{lemma}
\label{lemma:prox_fm}
    Let $f_i:\H \to (-\infty,+\infty]$ be an $L_i$-smooth function for all $i=1,\dots,m-1$, and let $f_m$ be a proper closed function. If $\sum_{i=1}^m f_i$ is coercive and $\gamma < \left( \sum_{i=1}^{m-1} L_i\right)^{-1}$, then $\prox_{\gamma f_m}$ has a full domain and is compact-valued. 
\end{lemma}
\alert{\begin{proof}
    % The last claim in fact holds by \eqref{eq:prox_subset_J}. To show that $\prox_{\gamma f_m}(x)\neq \emptyset$ and compact for all $x\in \H$, 
    % We argue as in the proof of \cite[Theorem 3.18]{AlcantaraLeeTakeda2024}: 
    Using \eqref{eq:descentlemma}, we can show that 
    \begin{align*}
        \min  \sum_{i=1}^m f_i & \leq  \sum_{i=1}^m f_i(x) \leq \sum_{i=1}^{m-1} \left( f_i(\bar{x}) + \lla \nabla f_i(\bar{x}),x-\bar{x}\rla\right)  + \frac{1}{2}\sum_{i=1}^{m-1} L_i\norm{x-\bar{x}}^2 + f_m(x),  
    \end{align*}
    where $\bar{x}\in \H$ is arbitrary and the minimum on the left-most side is finite by \cite[Theorem 1.9]{RW98}, noting the finite-dimensionality of $\H$ and coercivity hypothesis. It follows that $c+ \inner{y}{x} +  \frac{ \sum_{i=1}^{m-1} L_i}{2}\norm{x}^2+f_m(x) \geq 0$ for some $c\in \Re$ and $y\in \H$. The claim now follows from \cite[Exercise 1.24 and Theorem 1.25]{RW98}.
  \end{proof}}

\begin{theorem}
\label{thm:dr_convergence_optimization_subsequential}
Let $\mu\in (0,2)$, and $\lambda,\lambda_1,\dots,\lambda_{m-1} \in (0,+\infty)$ with $\sum_{i=1}^{m-1}\lambda_i=1$. For each $i=1,\dots,m-1$, denote   
    \begin{equation}
        \bar{\gamma}_i \coloneqq \begin{cases}
        \frac{1}{L_i} & \text{if } -2\sigma_i < (2-\mu)L_i \\
        -\frac{1}{\sigma_i}\left( 1-\frac{\mu}{2}\right) & \text{otherwise}
    \end{cases},
    \label{eq:bargamma_i}
    \end{equation} 
where   $\sigma_i \in [-L_i,0]$ such that $f_i - \frac{\sigma_i}{2}\norm{\cdot}^2$ is convex 
(which exists by \cref{lemma:optimization_lsmooth}(i)). Suppose the hypotheses of \cref{lemma:prox_fm} hold. If $\{(x_1^k,\dots, x_{m-1}^k,z_1^k,\dots,z_{m-1}^k,y^k\}$ is generated by \cref{alg:dr_scaled_opt} with $\frac{\lambda}{\lambda_i} \in (0,\bar{\gamma}_i)$ for all $i=1,\dots,m-1$, then 
    \begin{enumerate}[(i)]
        \item $\{(x_1^k,\dots, x_{m-1}^k,z_1^k,\dots,z_{m-1}^k,y^k)\}$ is bounded; 
        % \item $\norm{x_i^k-x_i^{k+1}}=o(1/\sqrt{k})$; and 
        \item $z_i^*,y^* \in \zer \left( \sum_{i=1}^m \partial f_i\right)$ if $z_i^*$ and $y^*$ are accumulation points of $\{z_i^k\}$ and $\{y^k\}$.
    \end{enumerate}
\end{theorem}
\begin{proof}
    From the $z$-step in \cref{alg:dr_scaled_opt}, we have $x_i^k = z_i^k + \frac{\lambda}{\lambda_i}\nabla f_i(z_i^k) = z_i^k + \gamma_i \nabla f_i(z_i^k)$, where $\gamma_i \coloneqq \frac{\lambda}{\lambda_i}$ for all $i=1,\dots,m-1$.
    With this and keeping in mind that $\sum_{i=1}^{m-1}\lambda
    _i=1$, it is not difficult to calculate that 
    \begin{align}
        &  \textstyle  \prox_{\lambda f_m} \left( \sum_{i=1}^{m-1} \lambda_i (2z_i^k-x_i^k)\right)  \notag  \\
           &   = \argmin _{z\in \H}   \sum_{i=1}^{m-1} \left( f_i(z_i^k) + \inner{\nabla f_i(z_i^k)}{z-z_i^k} + \frac{1}{2\gamma_i}\norm{z-z_i^k}^2\right)+ f_m(z). \label{eq:prox=}
    \end{align}
Now, using \cref{lemma:optimization_lsmooth}(i), there exists $\sigma_i \in [-L_i,0]$ such that $\tildef_i \coloneqq f_i - \frac{\sigma_i}{2}\norm{\cdot}^2$ is convex. Note that since 
\[ \textstyle  \tildef_i(y) - \tildef_i(x) - \inner{\nabla \tildef_i(x)}{y-x} \leq \frac{L_i-\sigma_i}{2}\norm{y-x}^2 \quad \forall x,y\in \H,\]
it follows from \cite[Theorem 5.8]{Beck17} that $\tildef_i$ is $(L_i-\sigma_i)$-smooth. Continuing from \eqref{eq:prox=} and by some simple computations, we get 
\begin{align}
    & \textstyle  \prox_{\lambda f_m} \left( \sum_{i=1}^{m-1} \lambda_i (2z_i^k-x_i^k)\right) \notag \\
     &  = \argmin _{z\in \H}  \!\! \textstyle{\sum_{i=1}^{m-1} \left( \tildef_i(z_i^k) + \inner{\nabla \tildef_i(z_i^k)}{z-z_i^k} + \frac{1-\gamma_i\sigma_i}{2\gamma_i}\norm{z-z_i^k}^2+ \frac{\sigma_i}{2}\norm{z}^2\right) \!+\! f_m(z) }\label{eq:prox=argmin}
\end{align}
Denoting the optimal value of the right-hand side of \eqref{eq:prox=argmin} by $V_k$ and by the definition of the $y$-update, we have
\begin{equation}
    V_k =  \textstyle  \sum_{i=1}^{m-1} \left( \tildef_i(z_i^k) + \inner{\nabla \tildef_i(z_i^k)}{y^k-z_i^k} + \frac{1-\gamma_i\sigma_i}{2\gamma_i}\norm{y^k-z_i^k}^2 + \frac{\sigma_i}{2}\norm{y^k}^2\right) \!+\! f_m(y^k).
    \label{eq:V_k}
\end{equation}
By definition of $V_k$, we also have from \eqref{eq:prox=argmin} that 
\begin{equation*}
    \begin{array}{rl}
      V_{k+1}  \leq & \! \! \sum_{i=1}^{m-1} \left( \tildef_i(z_i^{k+1}) + \inner{\nabla \tildef_i(z_i^{k+1})}{y^k-z_i^{k+1}} + \frac{1-\gamma_i\sigma_i}{2\gamma_i}\norm{y^k-z_i^{k+1}}^2 \right. \\ 
    & \left. + \frac{\sigma_i}{2}\norm{y^k}^2\right) + f_m(y^k). 
    \end{array}
\end{equation*}
Subtracting this from \eqref{eq:V_k}, we get 

\begin{align}
    & V_k-V_{k+1}\notag \\
    &\textstyle   \geq \sum_{i=1}^{m-1} \left( \tildef_i(z_i^k)  - \tildef_i(z_i^{k+1})  + \inner{\nabla \tildef_i(z_i^k)}{y^k-z_i^k} - \inner{\nabla \tildef_i(z_i^{k+1})}{y^k-z_i^{k+1}} \right. \notag \\
    & \textstyle  \left. \quad + \frac{1-\gamma_i\sigma_i}{2\gamma_i}\norm{y^k-z_i^k}^2 -  \frac{1-\gamma_i\sigma_i}{2\gamma_i}\norm{y^k-z_i^{k+1}}^2 \right) \notag \\
    & \textstyle  \geq \sum_{i=1}^{m-1} \left( -\inner{\nabla \tildef_i (z_i^{k+1}) - \nabla \tildef_i (z_i^{k}) }{y^k-z_i^k} + \frac{1}{2(L_i-\sigma_i)}\norm{\nabla \tildef_i (z_i^{k+1}) - \nabla \tildef_i (z_i^{k})}^2 \right. \notag \\
    & \textstyle   \left. \quad + \frac{1-\gamma_i\sigma_i}{2\gamma_i}\norm{y^k-z_i^k}^2 -  \frac{1-\gamma_i\sigma_i}{2\gamma_i}\norm{y^k-z_i^{k+1}}^2 \right)  \notag  \\
     & \textstyle   = \sum_{i=1}^{m-1} \left( -\inner{\nabla \tildef_i (z_i^{k+1}) - \nabla \tildef_i (z_i^{k}) }{y^k-z_i^k} + \frac{1}{2(L_i-\sigma_i)}\norm{\nabla \tildef_i (z_i^{k+1}) - \nabla \tildef_i (z_i^{k})}^2 \right. \notag \\
    & \textstyle  \left. \quad - \frac{1-\gamma_i\sigma_i}{2\gamma_i}\norm{z_i^{k+1}-z_i^k}^2 + \frac{1 - \gamma_i \sigma_i}{\gamma_i}\inner{z_i^{k+1}-z_i^k}{y^k-z_i^k}   \right) \label{eq:apply_squared_difference identity}
\end{align}
where the second inequality holds by \eqref{eq:descentlemma_convex} since $\tildef_i$ is convex and $(L_i-\sigma_i)$-smooth, while \eqref{eq:apply_squared_difference identity} holds since $\norm{y-z}^2 - \norm{y-z'}^2 = - \norm{z-z'}^2 +2 \inner{z-z'}{z-y}$. To simplify our notations, let us denote
$ \Deltak{g}  \coloneqq \nabla \tildef_i  (z_i^{k+1}) - \nabla \tildef_i (z_i^k) $ and 
        $\Deltak{z}  \coloneqq z_i^{k+1}- z_i^k.$
Meanwhile, for any $i=1,\dots,m-1$, 
$\mu (y^k-z_i^k) = x_i^{k+1} - x_i^k = \left( 1+ \gamma_i \sigma_i \right)\Deltak{z} + \gamma_i\Deltak{g} ,$
where the first and last equality hold by the $x$-and $z$-update rules in \cref{alg:dr_scaled_opt}.
% and the definition of $\tildef_i$. 
Continuing from \eqref{eq:apply_squared_difference identity} and after simplifying, we obtain
\begin{align}
    & V_k-V_{k+1} \notag \\
    & \textstyle \geq  \sum_{i=1}^{m-1} \left[   \left( \frac{1}{2(L_i-\sigma_i)} - \frac{\gamma_i}{\mu}\right) \norm{\Deltak{g}}^2  -  \frac{2\gamma_i^2\sigma_i^2 - \mu \gamma_i\sigma_i - (2-\mu)}{2\mu\gamma_i} \norm{\Deltak{z}}^2- \frac{2\gamma_i\sigma_i}{\mu} \inner{\Deltak{g}}{\Deltak{z}}\right] \notag \\
    &\textstyle  \geq \sum_{i=1}^{m-1} \left[   \left( \frac{1}{2(L_i-\sigma_i)} - \frac{\gamma_i}{\mu} - \frac{2\gamma_i\sigma_i}{\mu (L_i-\sigma_i)}\right) \norm{\Deltak{g}}^2  -  \frac{2\gamma_i^2\sigma_i^2 - \mu \gamma_i\sigma_i - (2-\mu)}{2\mu\gamma_i} \norm{\Deltak{z}}^2\right] \notag \\
    & \textstyle  = \sum_{i=1}^{m-1} \left[    \frac{\mu - 2\gamma_i (L_i+\sigma_i)}{2\mu (L_i-\sigma_i)}\norm{\Deltak{g}}^2  -  \frac{2\gamma_i^2\sigma_i^2 - \mu \gamma_i\sigma_i - (2-\mu)}{2\mu\gamma_i} \norm{\Deltak{z}}^2\right] \notag \\ 
    & \textstyle  = \sum_{i=1}^{m-1} \left(   a_i(\gamma_i) \norm{\Deltak{g}}^2  +  b_i(\gamma_i) \norm{\Deltak{z}}^2\right)  ,\label{eq:Vdiff_last}
\end{align}
with $a_i(\gamma_i)  \coloneqq  \frac{\mu - 2\gamma_i (L_i+\sigma_i)}{2\mu (L_i-\sigma_i)} $ and $b_i(\gamma_i) \coloneqq  -  \frac{2\gamma_i^2\sigma_i^2 - \mu \gamma_i\sigma_i - (2-\mu)}{2\mu\gamma_i}$, where the last inequality holds by \eqref{eq:descentlemma_convex2}, noting that $\sigma_i\leq 0$. 

We now claim that for each $i=1,\dots,m-1$, there exists $c_i(\gamma_i)$ and $\bar{\gamma}_i>0$ such that $c_i(\gamma_i)>0$ if $\gamma_i \in (0,\bar{\gamma}_i)$ and
\begin{equation}
    \textstyle V_k - V_{k+1} \geq \sum_{i=1}^{m-1}c_i (\gamma_i) \norm{\Deltak{z}}^2 \quad \forall k,
    \label{eq:Vdiff_deltaz only}
\end{equation} 
The coefficient $a_i(\gamma_i)$ in \eqref{eq:Vdiff_last} is positive if $0<\gamma_i<\alpha_i$, where $\alpha_i \coloneqq \frac{\mu}{2(L_i+\sigma_i)} \in (0, \infty]$, and the coefficient $b_i(\gamma_i)$ is positive if $0<\gamma_i<\beta_i$, where $\beta_i \coloneqq -\frac{1}{\sigma_i}\left( 1 - \frac{\mu}{2}\right) \in (0, \infty]$. Setting $\bar{\gamma}_i \coloneqq \min\{\alpha_i, \beta_i\}$ and $c_i(\gamma_i) \coloneqq b_i(\gamma_i)$ ensures the claim holds. We now show that if $\min\{\alpha_i, \beta_i\} = \alpha_i$, a larger $\bar{\gamma}_i$ can be chosen. 
% If $\alpha_i < \beta_i$ and $\bar{\gamma}_i \in [\alpha_i, \beta_i)$, then $a_i(\gamma_i) < 0$ for $\gamma_i \in (\alpha_i, \bar{\gamma}_i)$, and $a_i(\gamma_i) = 0$ at $\gamma_i = \alpha_i$ and \eqref{eq:Vdiff_deltaz only} holds with $c_i(\gamma_i) \coloneqq b_i(\gamma_i)$.
% Indeed, the coefficient $a_i(\gamma_i)$ in \eqref{eq:Vdiff_last} is strictly positive if $0<\gamma_i<\alpha_i$, where $\alpha_i\coloneqq \frac{\mu}{2(L_i+\sigma_i)} \in (0,+\infty]$. On the other hand,  $b_i(\gamma_i)>0$ if $ 0<\gamma_i < \beta_i $, where $\beta_i\coloneqq -\frac{1}{\sigma_i}\left( 1 - \frac{\mu}{2}\right)\in (0,\infty]$.  Hence, we can simply set $\bar{\gamma}_i \coloneqq \min\{ \alpha_i,\beta_i\}$ and set $c_i(\gamma_i) \coloneqq b_i(\gamma_i)$, and the claim now holds true. We further show that when $\min \{ \alpha_i , \beta_i\}=\alpha_i$, we can choose a larger $\bar{\gamma}_i$. 
To this end, suppose that $\alpha_i < \beta_i$ and let $\bar{\gamma}_i \in [\alpha_i,\beta_i)$. Then  $a_i(\gamma_i) <0$  for any $\gamma_i \in (\alpha_i, \bar{\gamma}_i)$, and $a_i(\gamma_i)=0$ if $\gamma_i = \alpha_i$. In the latter case, note that \eqref{eq:Vdiff_deltaz only} holds with the choice $c_i(\gamma_i) \coloneqq b_i(\gamma_i)$. 
On the other hand, if $\gamma_i \in (\alpha_i,\bar{\gamma}_i)$, 
\begin{align}
  a_i(\gamma_i) \norm{\Deltak{g}}^2 + b_i(\gamma_i) \norm{\Deltak{z}}^2 & \geq  \left( a_i(\gamma_i)(L_i-\sigma_i)^2  +  b_i(\gamma_i) \right) \norm{\Deltak{z}}^2 \notag \\
  &  \textstyle  =-\frac{2\gamma_i^2L_i^2 - \mu \gamma_iL_i - (2-\mu)}{2\mu\gamma_i} \norm{\Deltak{z}}^2 \label{eq:simplified_coefficient}
\end{align}
where the first inequality holds since $\tildef_i$ is $(L_i-\sigma_i)$-smooth and $a_i(\gamma_i)<0$, and the equality holds after simple calculations. The coefficient in \eqref{eq:simplified_coefficient} is strictly positive if $\gamma_iL_i <1$. Meanwhile, $\alpha_i < \beta_i$ is equivalent to  $-2\sigma_i < (2-\mu)L_i$, which implies that  $\alpha_i < \frac{1}{L_i} < \beta_i$. Hence, we can set $\bar{\gamma}_i \coloneqq \frac{1}{L_i}$. To summarize, we have shown that if we set $\bar{\gamma}_i$ as in \eqref{eq:bargamma_i}, 
% \begin{equation}
%     \bar{\gamma}_i \coloneqq \begin{cases}
%     \frac{1}{L_i} & \text{if } -2\sigma_i < (2-\mu)L_i \\
%     -\frac{1}{\sigma_i}\left( 1-\frac{\mu}{2}\right) & \text{otherwise}
% \end{cases},
% \end{equation} 
then \eqref{eq:Vdiff_deltaz only} holds such that when $\gamma_i \in (0,\bar{\gamma}_i)$, then $c_i(\gamma_i)$ given by

\begin{equation}
    c_i(\gamma_i) \coloneqq \begin{cases}
        -\frac{2\gamma_i^2L_i^2 - \mu \gamma_iL_i - (2-\mu)}{2\mu\gamma_i}  & \text{if }-2\sigma_i < (2-\mu)L_i  ~\text{and}~\\
        & \frac{\mu}{2(L_i+\sigma_i)}<\gamma_i <\frac{1}{ -\sigma_i}\left( 1-\frac{\mu}{2}\right) \\ 
        -  \frac{2\gamma_i^2\sigma_i^2 - \mu \gamma_i\sigma_i - (2-\mu)}{2\mu\gamma_i} & \text{otherwise}\\
    \end{cases}
    \label{eq:c_i(gamma_i)}
\end{equation}
is strictly positive. Using \eqref{eq:Vdiff_deltaz only}, the rest of the proof follows the same arguments as in \cite[Proposition 3.15, Theorem 3.18 and Theorem 3.19]{AlcantaraLeeTakeda2024}.
  \end{proof}

\begin{remark}
 For  $m = 2 $, this result recovers the convergence of \cite[Theorem 4.3]{ThemelisPatrinos2018} with a sharper constant estimate in \eqref{eq:c_i(gamma_i)} for $\sigma_i < 0$ (\textit{i.e.,} the nonconvex case). Thus, \cref{thm:dr_convergence_optimization_subsequential} improves upon \cite[Theorem 4.3]{ThemelisPatrinos2018} and extends it to $ m $-functions with $m \geq 3 $.

\end{remark}

\subsection{\alert{Numerical example}}
\alert{We illustrate the setting considered in \cref{sec:dr_optimization_weakconvex} with a numerical example.}
\begin{example}
    We consider the sparse low-rank matrix estimation problem in \cite{ParekhSelesnick2017} with an additional positive semidefinite constraint as follows:
        \begin{equation}
            \min _{x\in \Re^{p\times p}}~\underbrace{\alert{\delta_{\mathbb{S}_+^p} (x)}}_{F_1(x)} + \underbrace{\frac{1}{2}\norm{x-y}_F^2}_{F_2(x)} + \underbrace{\tau_0 \sum_{i=1}^{p} \phi (s_i(x);\omega_0) }_{F_3(x)}+\underbrace{ \tau_1 \sum_{i,j=1}^p \phi (x_{ij};\omega_1)}_{F_4(x)},\label{eq:experiment}
        \end{equation}
    where $\mathbb{S}_+^p$ denotes the set of $p\times p$ positive semidefinite matrices, $\norm{\cdot}_F$ denotes the Frobenius norm, $(s_1(x),\dots,s_p(x))$ denotes the  singular values of $x\in \Re^{p\times p}$, and $\phi$ is the penalty function given by 
        \[\phi(t;\omega) \coloneqq \frac{|t|}{1+\omega|t|/2}, \quad \omega\geq 0 ,\]
    which is a $-\omega$-convex function, i.e., $\omega$-weakly convex function; \alert{see \cite[Lemma 1]{ParekhSelesnick2017}}. Note that  $F_i$ is $\sigma_i$-convex where $(\sigma_1,\sigma_2,\sigma_3,\sigma_4) = (0,1,-\tau_0\omega_0,-\tau_1\omega_1)$.
    
    We consider the covariance matrix estimation problem in \cite{RichardSavalleVayatis2012,ZhouXiuLuoKong2015}.\footnote{Scripts used to generate the data:
\url{https://github.com/ShenglongZhou/ADMM} (accessed October 18, 2025).} Given $p>0$, we generate a block-diagonal population covariance $\Sigma_0\in\mathbb{R}^{p\times p}$ with $K$ blocks of random sizes that sum to $p$. For block $b$, draw $v_b\in\mathbb{R}^{p_b}$ i.i.d.\ from $\mathrm{Unif}[-1,1]$ and set the block to $v_b v_b^\top$; hence $\mathrm{rank}(\Sigma_0)=K$. Then draw $n$ i.i.d.\ samples $X_\ell\sim\mathcal{N}(0,\Sigma_0)$ (implemented as $X_\ell=\Sigma_0^{1/2}z_\ell$, $z_\ell\sim\mathcal{N}(0,I_p)$), compute the sample mean $\bar X=\tfrac{1}{n}\sum_{\ell=1}^n X_\ell$, form the unbiased sample covariance $\Sigma_n=\tfrac{1}{n-1}\sum_{\ell=1}^n (X_\ell-\bar X)(X_\ell-\bar X)^\top$, and set $y= \Sigma_n$. In our experiments, we set $(K,n,p)=(5,50,500)$ to generate the problem data, and set $\tau_i = 0.1 $ and $\omega_i = 1$ for $i=0,1$.

We test the performance of \cref{alg:dr_scaled_opt} with $(f_1,f_2,f_3,f_4) = (F_{a},F_b,F_c,F_d)$ for $(a,b,c,d)\in \{ (1,2,3,4), (1,4,3,2),(1,2,4,3)\}$,\footnote{We also tested other permutations with the last block $F_d\neq F_1$, since our theory guarantees convergence when the last component is strongly or
weakly convex. Empirically, the last block largely dictates
performance: for any fixed $F_d\neq F_1$, permuting the remaining
$\{F_1,F_2,F_3,F_4\}\setminus\{F_d\}$ produced indistinguishable accuracy
and iteration counts across all instances.} and we choose stepsize $\lambda$ according to \cref{prop:optimaldelta}(i). The algorithm is terminated when the maximum blockwise mean-squared residual of $\z^k$ is below $10^{-6}$, where the residual mapping is defined by the generalized gradient mapping (c.f. \cite[Definition 10.5]{Beck17})
\begin{align*}
    \text{Res} ( \z^k) & = \frac{1}{\lambda}\bfLambda \left( \z^k - J_{\lambda \G}^{\bfLambda}(\z^k - \lambda \bfLambda^{-1} \u^k)\right) \\ 
    & = \frac{1}{\lambda} (\lambda_1 (z_1^k-y^k), \ldots, \lambda_{m-1}(z_{m-1}^k-y^k))  ,
\end{align*}
where $\u^k:=\lambda^{-1}\bfLambda(x^k-z^k)
\in \F(z^k)$ is the selection determined by
the $z$-update in \cref{alg:dr_scaled_opt}.
    That is, we terminate when $\norm{\text{Res}(\z^k)}_{\infty,F}^2<10^{-6}$, where $\norm{\y}_{\infty,F}^2 \coloneqq  \max_i  \norm{y_i}_{F,p}^2 $ for any $\y = (y_1,\dots,y_{m-1})\in (\Re^{p\times p})^{m-1}$, and $\norm{y}_{F,p}^2 \coloneqq \frac{1}{p^2}\sum_{i,j=1}^p y_{ij}^2$, for any $y\in \Re^{p\times p}$. The $o(1/\sqrt{k})$ convergence rate established in \cref{thm:dr_convergence_optimization}(ii) is illustrated in \cref{fig:convergencerate}.

\begin{figure}[t]
    \centering
    \includegraphics[width=.7\linewidth]{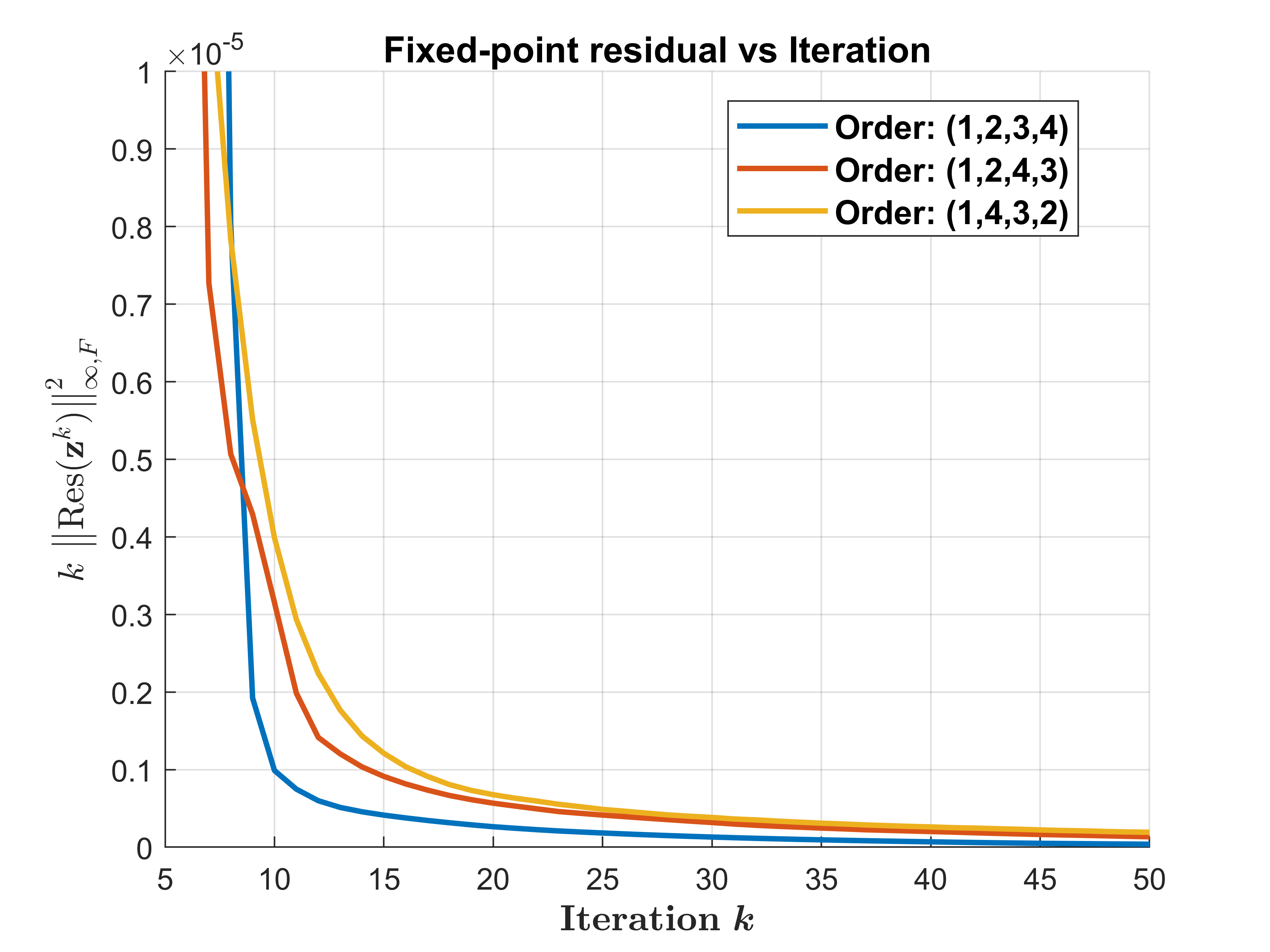}
    \caption{Convergence of $k\norm{\mathrm{Res}(\z^k)}_{\infty,F}^2$ 
    % \textcolor{red}{$k \norm{\mathrm{Res}(\z^k)}_{\infty,F}^2$ according to the y-axis in the plot?} 
    to zero for the orderings $(1,2,3,4)$, $(1,2,4,3)$, and $(1,4,3,2)$, with equal weights $\lambda_1=\lambda_2=\lambda_3=\frac{1}{3}$.}
    \label{fig:convergencerate}
\end{figure}

    For each ordering, we swept the Douglas--Rachford mixing weights $(\lambda_1,\lambda_2,\lambda_3)$  on the simplex for random synthetic instances. We report the average iteration count and mean squared error of $y^k$, i.e., $\text{MSE} (y^k) \coloneqq \norm{y^k-\Sigma_0}^2_{F,p} $, over 20 random instances. \cref{tab:best-metrics-argmin} reports, for each ordering, the minimum mean MSE and the minimum mean iteration count, together with all weight triples that attain those minima. The heatmaps in \cref{fig:mse-iter-tabular} 
    % use a common color scale per row (MSE, iterations) so shading is comparable across orderings; black squares mark all minimizers.
    show that both the ordering of the functions and the weights
$(\lambda_1,\lambda_2,\lambda_3)$ affect performance, with the impact being more
pronounced on speed than on accuracy. When the strongly convex block $F_2$ is placed among the first $m\!-\!1$ functions, assigning it a
moderate weight tends to yield faster convergence while preserving accuracy.
In contrast, placing the strongly convex block $F_2$ last, together with a
small weight on the merely convex block ($F_1$), consistently delivers both
accurate solutions and fast convergence.
\end{example}

% \begin{table}[t]\centering\small
% \begin{tabular}{lccc}
% \toprule
% Ordering & Residual (mean\,\textpm\,std) & MSE (mean\,\textpm\,std) & Iterations (mean) \\
% \midrule
% (1,2,3,4) & $3.9e-07\,\pm\,2.3e-07$ & $2.3e-03\,\pm\,9.3e-04$ & 8.8 \\
% (1,4,3,2) & $8.0e-07\,\pm\,9.6e-08$ & $2.3e-03\,\pm\,9.2e-04$ & 7.8 \\
% (1,2,4,3) & $8.0e-07\,\pm\,9.2e-08$ & $2.3e-03\,\pm\,8.9e-04$ & 8.8 \\
% \bottomrule
% \end{tabular}
% \caption{Effect of proximal operator order over 20 random instances using equal weights $\lambda_1=\lambda_2=\lambda_3 = \frac{1}{3}$.}
% \label{tab:pdr-order}
% \end{table}
\begin{figure}[H]
  \centering
  \begin{tabular}{@{}c@{}} % single centered column with no side padding
    \includegraphics[width=.7\linewidth]{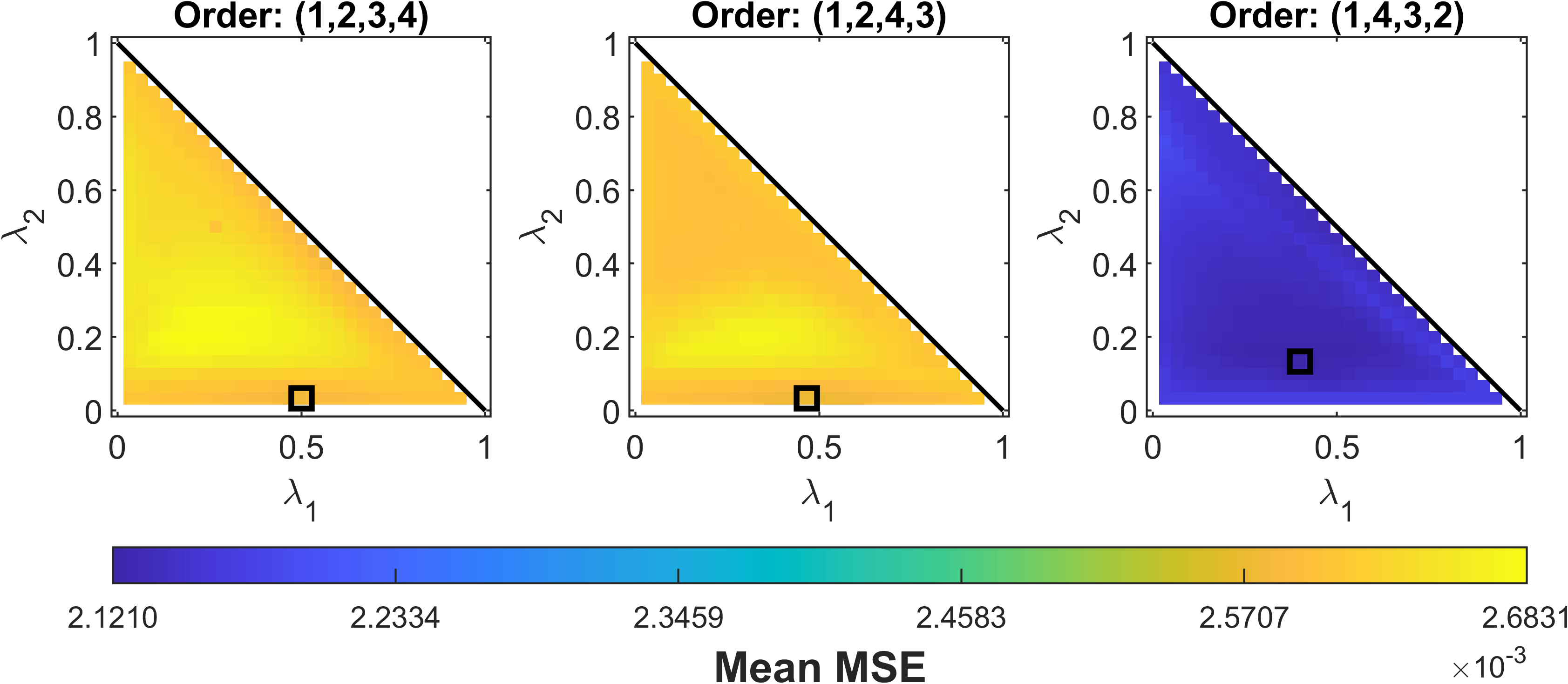} \\
    [2em]
    \includegraphics[width=.7\linewidth]{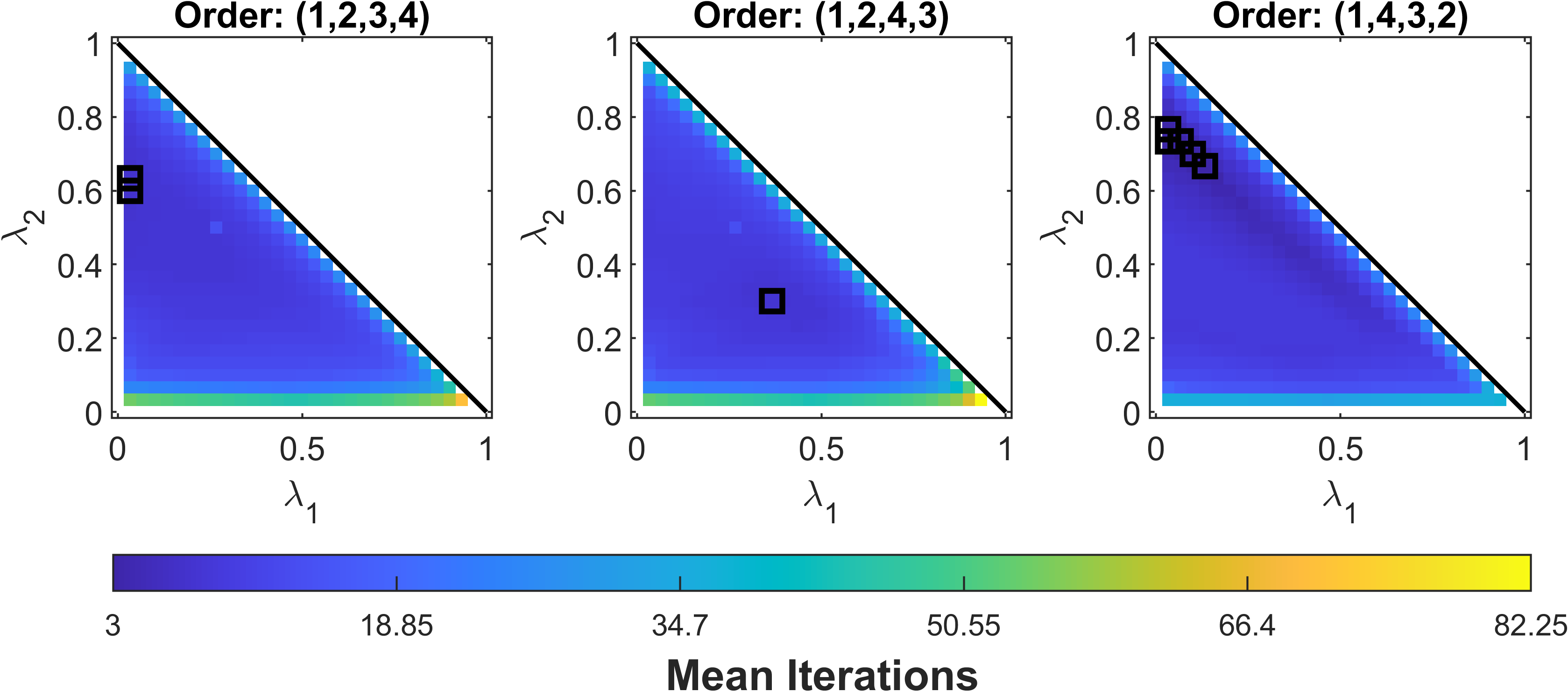}
  \end{tabular}
  \caption{{\small Heatmaps over $(\lambda_1,\lambda_2)$ (with $\lambda_3=1-\lambda_1-\lambda_2$) for three orderings. 
Color scales are shared across columns for each metric. 
Black squares indicate weight triples achieving the minimum.}}
  \label{fig:mse-iter-tabular}
\end{figure}

\begin{table}[H]
\centering
\small
\setlength{\tabcolsep}{6pt}
\renewcommand{\arraystretch}{1.15}
\caption{Minimum (mean) MSE and minimum (mean) iteration count for each ordering, listing all weight triples that achieve each minimum.}
\begin{tabular}{@{}%
l
c
>{\raggedright\arraybackslash}p{0.22\linewidth}
@{\hspace{6pt}} % small gap between the two blocks
c
>{\raggedright\arraybackslash}p{0.46\linewidth}
@{}}
\toprule
\multirow{2}{*}{Ordering}
  & \multicolumn{2}{c}{MSE}
  & \multicolumn{2}{c}{Iterations} \\
\cmidrule(lr){2-3}\cmidrule(lr){4-5}
 & Min & Argmin $(\lambda_1,\lambda_2,\lambda_3)$
 & Min & Argmin $(\lambda_1,\lambda_2,\lambda_3)$ \\
\midrule
1--2--3--4
 & $2.579\times10^{-3}$
 & (0.500, 0.033, 0.467)
 & $7.05$
 & (0.033, 0.600, 0.367), (0.033, 0.633, 0.333) \\
1--2--4--3
 & $2.573\times10^{-3}$
 & (0.467, 0.033, 0.500)
 & $7.70$
 & (0.367, 0.300, 0.333) \\
 1--4--3--2
 & $2.121\times10^{-3}$
 & (0.400, 0.133, 0.467)
 & $3.00$
 & (0.033, 0.733, 0.233), (0.033, 0.767, 0.200), (0.067, 0.733, 0.200), (0.100, 0.700, 0.200), (0.133, 0.667, 0.200) \\
\bottomrule
\end{tabular}
\label{tab:best-metrics-argmin}
\end{table}

\section{Conclusion}\label{sec:conclusion}
We proved that if the sum of the operators' monotonicity moduli is strictly positive, the shadow sequence of the proposed DR algorithm with an appropriate step size converges to the inclusion problem's solution. This generalizes prior work on two-operator inclusion problems with generalized monotone operators. Applications to unconstrained sum-of-$m$-functions optimization involving strongly and weakly convex functions are presented. Lastly, we established global subsequential convergence in finite dimensions, assuming all but one function has Lipschitz continuous gradients, with the remaining function being \alert{only} proper and closed. Preliminary experiments indicate that both the ordering of the functions and the choice of weights affect empirical performance. A key practical question is how to select these in a principled manner (e.g., whether the strongly convex block should systematically be placed last) so as to balance accuracy and speed.

\alert{It would also be interesting to investigate whether the generalized
monotonicity analysis developed here can be extended to other multioperator
splitting frameworks. Recent works have studied graph-based
and distributed extensions of the Douglas--Rachford method
\cite{BrediesChencheneNaldi2024}, as well as graph-based forward--backward
splitting schemes that include forward terms
\cite{AragonArtachoCampoyLopezPastor2025,AkermanChencheneGiselssonNaldi2025}.
The existing convergence guarantees for these frameworks are primarily developed
in monotone settings, possibly with cocoercive forward terms. Extending the
present strongly monotone plus weakly monotone analysis to such splitting schemes remains an interesting open direction  and is part of our ongoing work. Another related direction is to study other notions of nonmonotonicity beyond
weak monotonicity. For instance, a preliminary investigation of the
Douglas--Rachford framework considered in the present paper under comonotonicity
is carried out in our related work \cite{AlcantaraDaoTakeda2025}.\footnote{A broader treatment
is developed in our subsequent work
\cite{AlcantaraDaoTakeda2026}, which establishes
a splitting framework for composite semimonotone inclusions
that encompasses the comonotone Douglas--Rachford setting
and accommodates cocoercive forward terms.}}

\appendix
\section{\alert{Proof of \cref{thm:general_convergence2}}}
\label{app:dr_switched}

Once we establish analogues of \cref{lemma:JFix=zeros} and
\cref{prop:T_nonexpansive}, we can directly follow the same arguments in
\cref{thm:general_convergence} to prove the theorem.

\begin{proposition}
\label{lemma:JFix=zeros2}
    Let $A_i:\H\toset \H$, $i=1,\dots,m$, and let
    $\lambda,\lambda_1,\dots,\lambda_{m-1}\in (0,+\infty)$ with
    $\sum_{i=1}^{m-1}\lambda_i=1$.
    \begin{enumerate}[(i)]
        \item If $\x\in \Fix(\dr{\G}{\F})$, then there exists
        $\z\in \JLambda{\G}(\x)$ such that
        $
            \z\in
            \bfDelta_{m-1}\left(
            \zer\left(\sum_{i=1}^m A_i\right)
            \right).
        $

        \item If
        $
            \z\in
            \bfDelta_{m-1}\left(
            \zer\left(\sum_{i=1}^m A_i\right)
            \right),
        $
        then there exists $\x\in \Fix(\dr{\G}{\F})$ such that
        $\z\in \JLambda{\G}(\x)$.
    \end{enumerate}
    Consequently, if $\JLambda{\G}$ is single-valued, then
    \begin{equation*}
        \JLambda{\G}(\Fix(\dr{\G}{\F}))
        =
        \bfDelta_{m-1}\left(
        \zer\left(\sum_{i=1}^m A_i\right)
        \right).
    \end{equation*}
\end{proposition}

\begin{proof}
We first show that if $\x\in \Fix(\dr{\G}{\F})$, then there exists
$\z\in \zer(\F+\G)$ such that $\z\in \JLambda{\G}(\x)$. Let
$\x\in \Fix(\dr{\G}{\F})$. By \eqref{eq:dr_map2}, there exists
$\z\in \JLambda{\G}(\x)$ such that
$\z\in \JLambda{\F}(2\z-\x)$. Hence
$\x-\z\in \lambda\bfLambda^{-1}\G(\z)$ and
$(2\z-\x)-\z\in \lambda\bfLambda^{-1}\F(\z)$. Thus, there exist
$\u\in \G(\z)$ and $\v\in \F(\z)$ such that
$\x-\z=\lambda\bfLambda^{-1}\u$ and
$\z-\x=\lambda\bfLambda^{-1}\v$. Adding these two equalities gives
$\lambda\bfLambda^{-1}(\u+\v)=0$, and therefore $\u+\v=0$. Hence
$\bf{0}\in \G(\z)+\F(\z)$, and so $\z\in \zer(\F+\G)$. This proves (i).

To prove (ii), let
$
    \z\in
    \bfDelta_{m-1}\left(
    \zer\left(\sum_{i=1}^m A_i\right)
    \right).
$
By \cref{thm:campoy_new}, $\z\in \zer(\F+\G)$. Choose
$\u\in \G(\z)$ such that $-\u\in \F(\z)$, and define
$\x:=\z+\lambda\bfLambda^{-1}\u$. Then
$\x-\z=\lambda\bfLambda^{-1}\u\in \lambda\bfLambda^{-1}\G(\z)$, and hence
$\z\in \JLambda{\G}(\x)$. Moreover,
$(2\z-\x)-\z=\z-\x=-\lambda\bfLambda^{-1}\u
\in \lambda\bfLambda^{-1}\F(\z)$, and therefore
$\z\in \JLambda{\F}(2\z-\x)$. By \eqref{eq:dr_map2}, this gives
$\x\in \Fix(\dr{\G}{\F})$.  
\end{proof}

\begin{proposition}
\label{prop:T_nonexpansive2}
Let $A_i:\H\toset \H$ be $\sigma_i$-monotone for each $i=1,\dots,m$,
let $\lambda,\lambda_1,\dots,\lambda_{m-1}\in (0,+\infty)$ with
$\sum_{i=1}^{m-1}\lambda_i=1$, and let $\bfLambda$ be given by
\eqref{eq:Lambda}. Suppose that $\JLambda{\F}$ and $\JLambda{\G}$ are
single-valued on their domains. Define $V:\H^{m-1}\toset \H$ by
$V(\x)\coloneqq J_{\lambda A_m}\left(\sum_{i=1}^{m-1}\lambda_i x_i\right).$
For each $i=1,\dots,m-1$, define
$P_i(\x)\coloneqq
J_{\frac{\lambda}{\lambda_i}A_i}(2V(\x)-x_i).$
Then the following hold:
\begin{enumerate}[(i)]
    \item The mappings $V$ and $P_i$, $i=1,\dots,m-1$, are single-valued on
    $\dom(\dr{\G}{\F})$, and
    \[
        \JLambda{\G}(\x)=\bfDelta_{m-1}(V(\x)),
        \qquad
        \bigl(\JLambda{\F}\RLambda{\G}(\x)\bigr)_i=P_i(\x)
    \]
    for each $i=1,\dots,m-1$.

    \item Denote $\R\coloneqq \bfId-\dr{\G}{\F}$ and its components by
    $\R=(R_1,\dots,R_{m-1})$. Then
    \begin{equation}
        \frac{1}{\mu}R_i(\x)=V(\x)-P_i(\x)
        \label{eq:Id-T2}
    \end{equation}
    for each $i=1,\dots,m-1$.

    \item Suppose that either $\I=\emptyset$, or $\I\neq\emptyset$ and
    $(\delta_i)_{i\in\I}$ satisfies
    $\sigma_i+\sigma_m\delta_i\neq0$ for every $i\in\I$ and
    $\sum_{i\in\I}\delta_i=1$. Then, for any
    $\x,\y\in \dom(\dr{\G}{\F})$,
    \begin{align}
    & \textstyle \normlambda{\dr{\G}{\F}(\x)-\dr{\G}{\F}(\y)}^2
    \leq
    \normlambda{\x-\y}^2
    -\frac{2}{\mu}\sum_{i=1}^{m-1}
    \lambda_i\kappa_i\norm{R_i(\x)-R_i(\y)}^2 \notag \\
    & \textstyle
    -2\mu\lambda\sum_{i\in\I}\theta_i
    \norm{
    \sigma_i(P_i(\x)-P_i(\y))
    +
    \sigma_m\delta_i(V(\x)-V(\y))
    }^2 \notag \\
    & \textstyle
    -2\alpha\mu\lambda\sigma_m
    \norm{V(\x)-V(\y)}^2,
    \label{eq:T_nonexpansive2}
    \end{align}
    where
    \begin{equation}
        \alpha \coloneqq
        \begin{cases}
        0 & \text{if}~\I\neq\emptyset,\\
        1 & \text{if}~\I=\emptyset,
        \end{cases}
        \quad
        \kappa_i \coloneqq
        \begin{cases}
        1+\dfrac{\lambda}{\lambda_i}
        \dfrac{\sigma_i\sigma_m\delta_i}
        {\sigma_i+\sigma_m\delta_i}
        -\dfrac{\mu}{2}
        & \text{if}~i\in\I,\\[2mm]
        1-\dfrac{\mu}{2}
        & \text{if}~i\notin\I,
        \end{cases}
        \quad
        \theta_i\coloneqq
        \frac{1}{\sigma_i+\sigma_m\delta_i}.
        \label{eq:alpha,kappa,theta2}
    \end{equation}
\end{enumerate}
\end{proposition}

\begin{proof}
We have
$\dr{\G}{\F}=\bfId+\mu(\JLambda{\F}\RLambda{\G}-\JLambda{\G})$
by \eqref{eq:dr_map2}. Then
$\dom(\dr{\G}{\F})=\dom(\JLambda{\F}\RLambda{\G})$ and
$\JLambda{\F}\RLambda{\G}$ is single-valued on $\dom(\dr{\G}{\F})$.
The formula
$\JLambda{\G}(\x)=\bfDelta_{m-1}(V(\x))$ holds by
\cref{prop:G_resolvent}. Hence
$\RLambda{\G}(\x)=(2V(\x)-x_1,\dots,2V(\x)-x_{m-1})$, and
$\bigl(\JLambda{\F}\RLambda{\G}(\x)\bigr)_i=P_i(\x)$ follows from
\cref{prop:F_resolvent}. This proves part (i). Part (ii) follows from part (i)
and the identity
\begin{equation}
    \bfId-\dr{\G}{\F}
    =
    \mu(\JLambda{\G}-\JLambda{\F}\RLambda{\G}).
    \label{eq:Id-T=mu(J-J)2}
\end{equation}

We now prove part (iii). Using \eqref{eq:identity_squarednorm} and the equivalent
expression for $\dr{\G}{\F}$ given by
\begin{equation}
    \dr{\G}{\F}
    =
    \frac{(2-\mu)\bfId+\mu\RLambda{\F}\RLambda{\G}}{2},
    \label{eq:T_alternative2}
\end{equation}
we have
\begin{align}
& \textstyle
\normlambda{\dr{\G}{\F}(\x)-\dr{\G}{\F}(\y)}^2
=
\frac{2-\mu}{2}\normlambda{\x-\y}^2
+
\frac{\mu}{2}
\normlambda{\RLambda{\F}\RLambda{\G}(\x)
-
\RLambda{\F}\RLambda{\G}(\y)}^2
\notag \\
& \textstyle \quad
-
\frac{\mu(2-\mu)}{4}
\normlambda{
(\bfId-\RLambda{\F}\RLambda{\G})(\x)
-
(\bfId-\RLambda{\F}\RLambda{\G})(\y)
}^2 .
\label{eq:T_diff_2}
\end{align}
From \eqref{eq:T_alternative2}, we also obtain that
$\bfId-\RLambda{\F}\RLambda{\G}
=
\frac{2}{\mu}(\bfId-\dr{\G}{\F})
=
\frac{2}{\mu}\R$. Then, we further obtain from \eqref{eq:T_diff_2} that
\begin{align}
& \textstyle
\normlambda{\dr{\G}{\F}(\x)-\dr{\G}{\F}(\y)}^2
=
\frac{2-\mu}{2}\normlambda{\x-\y}^2
+
\frac{\mu}{2}
\normlambda{\RLambda{\F}\RLambda{\G}(\x)
-
\RLambda{\F}\RLambda{\G}(\y)}^2
\notag \\
& \textstyle \quad
-
\frac{2-\mu}{\mu}
\sum_{i=1}^{m-1}
\lambda_i\norm{R_i(\x)-R_i(\y)}^2 .
\label{eq:T_diff22}
\end{align}
Meanwhile, noting the single-valuedness of $\JLambda{\F}$ and $\JLambda{\G}$,
we have
\begin{align}
& \textstyle
\normlambda{\RLambda{\F}\RLambda{\G}(\x)
-
\RLambda{\F}\RLambda{\G}(\y)}^2
\notag \\
& \textstyle \leq
\normlambda{\RLambda{\G}(\x)-\RLambda{\G}(\y)}^2
-
4\lambda
\sum_{i=1}^{m-1}
\sigma_i\norm{P_i(\x)-P_i(\y)}^2
\notag \\
& \textstyle \leq
\normlambda{\x-\y}^2
-
4\lambda
\sum_{i=1}^{m-1}
\sigma_i\norm{P_i(\x)-P_i(\y)}^2
-
4\lambda\sigma_m\norm{V(\x)-V(\y)}^2 ,
\label{eq:composition_diff2}
\end{align}
where the first inequality holds by
\cref{prop:properties_reflected_resolvents}(i), while the second holds by
combining \cref{prop:properties_reflected_resolvents}(ii) and
\cref{prop:G_resolvent}.

When $\I\neq\emptyset$, we have
\begin{align}
& \textstyle
\sum_{i=1}^{m-1}\sigma_i\norm{P_i(\x)-P_i(\y)}^2
+
\sigma_m\norm{V(\x)-V(\y)}^2
\notag \\
& \textstyle =
\sum_{i\in\I}
\left(
\sigma_i\norm{P_i(\x)-P_i(\y)}^2
+
\sigma_m\delta_i\norm{V(\x)-V(\y)}^2
\right)
\notag \\
& \textstyle =
\sum_{i\in\I}
\frac{\sigma_i\sigma_m\delta_i}
{\sigma_i+\sigma_m\delta_i}
\norm{
\left(P_i(\x)-P_i(\y)\right)
-
\left(V(\x)-V(\y)\right)
}^2
\notag \\
& \textstyle \quad
+
\sum_{i\in\I}
\frac{1}{\sigma_i+\sigma_m\delta_i}
\norm{
\sigma_i\left(P_i(\x)-P_i(\y)\right)
+
\sigma_m\delta_i\left(V(\x)-V(\y)\right)
}^2
\notag \\
& \textstyle =
\frac{1}{\mu^2}
\sum_{i\in\I}
\frac{\sigma_i\sigma_m\delta_i}
{\sigma_i+\sigma_m\delta_i}
\norm{R_i(\x)-R_i(\y)}^2
\notag \\
& \textstyle \quad
+
\sum_{i\in\I}
\frac{1}{\sigma_i+\sigma_m\delta_i}
\norm{
\sigma_i\left(P_i(\x)-P_i(\y)\right)
+
\sigma_m\delta_i\left(V(\x)-V(\y)\right)
}^2 .
\label{eq:lasttwoterms2}
\end{align}
Here, the first equality uses $\sigma_i=0$ for all $i\notin\I$ and
$\sum_{i\in\I}\delta_i=1$; the second equality follows from
\eqref{eq:identity_squarednorm2}; and the last equality follows from part (ii).
Combining \eqref{eq:T_diff22}, \eqref{eq:composition_diff2}, and
\eqref{eq:lasttwoterms2}, we obtain \eqref{eq:T_nonexpansive2}.

On the other hand, suppose that $\I=\emptyset$. Then $\sigma_i=0$ for all
$i=1,\dots,m-1$, and we obtain \eqref{eq:T_nonexpansive2} by combining
\eqref{eq:T_diff22}, \eqref{eq:composition_diff2}, and the identity
$\normlambda{\bfDelta_{m-1}(V(\x)-V(\y))}^2
=
\norm{V(\x)-V(\y)}^2$.
 
\end{proof}
\ifdefined\submit

\noindent \textbf{Acknowledgements.} The authors sincerely thank the anonymous reviewers for
their careful reading of the manuscript and for their detailed and constructive
comments. Their suggestions led to significant improvements in the presentation,
the technical accuracy of the proofs, the discussion of related work, and the
overall clarity of the paper. In
particular, the reviewers' comments helped us clarify the relationship between
our product space reformulation and existing lifting techniques, and led to a
simplification of the maximal monotonicity analysis for arbitrary weights.

\section*{Declarations}

\paragraph{Ethics.}
Not applicable.

\paragraph{Clinical trial.}
Not applicable.

\paragraph{Funding.}  
This work is supported by the Grant-in-Aid for Scientific Research (B), JSPS, under Grant No. 23H03351.

\paragraph{Author contributions.}
Both authors contributed to the conception and development of the study. Jan
Harold Alcantara carried out the main analysis and wrote the first draft of the
manuscript. Akiko Takeda supervised the work, contributed to the development of
the analysis, and reviewed and edited the manuscript. Both authors read and
approved the final manuscript.

\paragraph{Conflict of interest.} 
The authors declare no competing interests. 
\medskip 

\paragraph{Data Availability.}  The datasets generated and/or analyzed during the current study are available in the GitHub repository: \url{https://github.com/jhalcantara/douglas-rachford-multioperator}.
\medskip 

\paragraph{Code Availability.} The code employed is available in the GitHub repository: \url{https://github.com/jhalcantara/douglas-rachford-multioperator}. 

\else 
\noindent \subsection*{Acknowledgements} The authors sincerely thank the anonymous reviewers for
their careful reading of the manuscript and for their detailed and constructive
comments. Their suggestions led to significant improvements in the presentation,
the technical accuracy of the proofs, the discussion of related work, and the
overall clarity of the paper. In
particular, the reviewers' comments helped us clarify the relationship between
our product space reformulation and existing lifting techniques, and led to a
simplification of the maximal monotonicity analysis for arbitrary weights.
\fi 
\setlength{\bibsep}{1pt}
\bibliography{bibfile}
\end{document}